\documentclass{amsart}
\usepackage{amsmath,amscd,xypic,amssymb,combelow,color,enumitem,graphicx,float,comment}

\usepackage[normalem]{ulem}
\usepackage[dvipsnames]{xcolor}
\makeatletter
\def\squiggly{\bgroup \markoverwith{\textcolor{black}{\lower3.5\p@\hbox{\sixly \char58}}}\ULon}

\usepackage{xr-hyper}
\usepackage[
pdftex,
bookmarks=false,
colorlinks=true,
debug=true,
pdfnewwindow=true]{hyperref}

\makeatletter
  \@addtoreset{equation}{section}
\makeatother

\newtheorem{theorem}[subsection]{Theorem}

\newtheorem{proposition}[subsection]{Proposition}
\newtheorem{lemma}[subsection]{Lemma}
\newtheorem{corollary}[subsection]{Corollary}

\newtheorem{definition}[subsection]{Definition}

\theoremstyle{remark}
\newtheorem{claim}[subsection]{Claim}
\newtheorem{example}[subsection]{Example}
\newtheorem{remark}[subsection]{Remark}

\def\fa{{\mathfrak{a}}}
\def\fg{{\mathfrak{g}}}
\def\fh{{\mathfrak{h}}}
\def\fn{{\mathfrak{n}}}
\def\fsl{{\mathfrak{sl}}}

\def\BC{{\mathbb{C}}}

\def\BZ{{\mathbb{Z}}}

\def\slaws{\text{standard Lyndon words}}
\def\aslaw{\text{affine standard Lyndon word}}
\def\aslaws{\text{affine standard Lyndon words}}
\def\SL{\mathrm{SL}}
\def\rtim{\mathrm{times}}
\def\rk{\mathrm{rank}}
\def\wI{\widehat{I}}
\def\wQ{\widehat{Q}}
\def\wDelta{\widehat{\Delta}}
\def\sb{\mathsf{b}}

\def\hgt{\text{ht}}

\newcommand\iso{\,\vphantom{j^{X^2}}\smash{\overset{\sim}{\vphantom{\rule{0pt}{0.20em}}\smash{\longrightarrow}}}\,}

\def\uup{U_q(\fn^+)}
\def\wt{\widetilde}

\begin{document}

\title[Affine Standard Lyndon words: A-type]
      {\Large{\textbf{Affine Standard Lyndon words: A-type}}}
	
\author[Yehor Avdieiev and Oleksandr Tsymbaliuk]{Yehor Avdieiev and Alexander Tsymbaliuk}

\address{Y.A.: University of Bonn, Department of Mathematics, Bonn, Germany}
\email{egor27avdeev@gmail.com}

\address{A.T.: Purdue University, Department of Mathematics, West Lafayette, IN, USA}
\email{sashikts@gmail.com}

\dedicatory{To the memory of Yulia Zdanovska}

\maketitle

\begin{abstract}
We generalize an algorithm of Leclerc~\cite{L} describing explicitly the bijection of Lalonde-Ram~\cite{LR}
from finite to affine Lie algebras. In type $A_n^{(1)}$, we compute all affine standard Lyndon words for any
order of the simple roots, and establish some properties of the induced orders on the positive affine roots.
\end{abstract}


\medskip

\section{Introduction}\label{sec:intro}


\subsection{Summary}\label{ssec:summary}
\

An interesting basis of the free Lie algebra generated by a finite family $\{e_i\}_{i\in I}$ was constructed
in the 1950s using the combinatorial notion of \emph{Lyndon} words (we recall these in
Definitions~\ref{def:lyndon}--\ref{def:lyndon-2}). A few decades later, this was generalized in~\cite{LR}
to any finitely generated Lie algebra $\fa$. Explicitly, if $\fa$ is generated by $\{e_i\}_{i\in I}$,
then any order on the finite alphabet $I$ gives rise to the combinatorial basis $\sb[\ell]$ as $\ell$ ranges through
all \emph{standard Lyndon} words (these will be recalled in Definition~\ref{def:standard}). Here, the standard
bracketing $\sb[\ell]$ is defined inductively with $\sb[i]=e_i$ (see Definition~\ref{def:bracketing}).

The key application of~\cite{LR} was to a simple finite-dimensional $\fg$, more precisely,
to its maximal nilpotent subalgebra $\fn^+$. According to the root space decomposition:
\begin{equation}\label{eqn:root vectors intro}
  \fn^+ = \bigoplus_{\alpha \in \Delta^+} \BC \cdot e_\alpha \,, \qquad
  \Delta^+=\Big\{\mathrm{positive\ roots}\Big\} .
\end{equation}
We note that the one-dimensional direct summands above are canonical as they are distinct eigenspaces for the adjoint
action of the Cartan subalgebra $\fh$ of $\fg$. However, picking a specific basis of \emph{root vectors} $\{e_\alpha\}_{\alpha\in \Delta^+}$
is non-canonical. Appealing to an additional grading by the root lattice of $\fg$,~\cite{LR} derived a natural bijection
\begin{equation}\label{eqn:1-to-1 intro}
  \ell \colon \Delta^+ \iso \Big\{\text{standard Lyndon words}\Big\}.
\end{equation}

A decade later, this bijection played a pivotal role in~\cite{L}, which studied the image of the dual canonical basis of $\uup$,
the positive half of a quantum group of $\fg$, under the embedding to the quantum shuffle algebra of~\cite{G,R1,S}. To this end,
\cite{L} obtained an explicit algorithm (see Proposition~\ref{prop:Leclerc algorithm}) for the above bijection~\eqref{eqn:1-to-1 intro}.
The key ingredient that allows for the quantum group generalization is the fact (attributed to~\cite{R2} in~\cite{L}) that
the order on $\Delta^+$ induced via~\eqref{eqn:1-to-1 intro} from a lexicographical order on words is \emph{convex}
in the sense of Definition~\ref{def:convex} (see Proposition~\ref{prop:fin.convex}).

$\ $

The motivation of the present note is to extend the above discussion to affine root systems. To this end,
we recall an enigmatic remark from the very end of~\cite{LR}:
  ``\emph{Preliminary computations seem to indicate that it will be very instructive to study root
          multiplicities for Kac-Moody Lie algebras by way of standard Lyndon words}''.

Let $\widehat{\fg}$ be the affinization of $\fg$, whose Dynkin diagram is obtained by extending the Dynkin diagram
of $\fg$ with one vertex $0$. Thus, on the combinatorial side, we consider the alphabet $\wI=I\sqcup \{0\}$.
The corresponding \emph{positive} subalgebra $\widehat{\fn}^+\subset \widehat{\fg}$ still admits the root space
decomposition $\widehat{\fn}^+=\bigoplus_{\alpha\in \wDelta^+} \widehat{\fn}^+_{\alpha}$, with
$\wDelta^+=\{\mathrm{positive\ affine\ roots}\}$. The key difference with~\eqref{eqn:root vectors intro}
is that not all $\widehat{\fn}^+_{\alpha}$ are one-dimensional:
\begin{equation}
  \dim \widehat{\fn}^+_{\alpha}=1 \quad \forall\, \alpha\in \wDelta^{+,\mathrm{re}} \,, \qquad
  \dim \widehat{\fn}^+_{\alpha}=|I| \quad \forall\, \alpha\in \wDelta^{+,\mathrm{im}}.
\end{equation}
Here, $\wDelta=\wDelta^{+,\mathrm{re}}\sqcup \wDelta^{+,\mathrm{im}}$ is the decomposition into real and imaginary
affine roots, with $\wDelta^{+,\mathrm{im}}=\{k\delta|k\geq 1\}$. It is therefore natural to consider an extended set
$\widehat{\Delta}^{+,\mathrm{ext}}$ of~\eqref{eq:extended-affine-roots}, counting imaginary roots with appropriate multiplicities.
Then, the degree reasoning similar to the one used in~\cite{LR} provides a natural analogue of~\eqref{eqn:1-to-1 intro}:
\begin{equation}\label{eqn:affine 1-to-1 intro}
  \SL \colon \wDelta^{+,\mathrm{ext}} \iso \Big\{ \text{affine standard Lyndon words} \Big\}.
\end{equation}
Our first result (Proposition~\ref{prop:generalized Leclerc}) is an inductive algorithm describing this bijection,
slightly generalizing Leclerc's algorithm describing~\eqref{eqn:1-to-1 intro}. As the first application, we use it
to find all $\aslaws$ for the simplest case of $\widehat{\fsl}_2$.

Our major technical result is the explicit description of all affine standard Lyndon words for $\widehat{\fsl}_{n+1}\ (n\geq 2)$.
To this end, we first straightforwardly treat the special order~\eqref{eq:stand-order} in Theorem~\ref{thm:sln-standard}.
We then derive a similar pattern for an arbitrary order in Theorem~\ref{thm:sln-general}.
The key feature is that all $\aslaws$ are determined by those of length $\leq n$.
Furthermore, we crucially use Rosso's convexity result for $\fsl_{n+1}$ to obtain an explicit description
of $n$ $\aslaws$ in degree $\delta$, which are key to establishing the general ``periodicity'' pattern.

The induced order~\eqref{eqn:affine-induces} on $\widehat{\Delta}^{+,\mathrm{ext}}$ is quite different
from the orders in the literature on affine quantum groups (\cite{B,KT}). While for $\widehat{\fsl}_2$
one gets a usual order~(\cite{D})
\begin{equation*}
  \alpha_1 < \alpha_1+\delta < \alpha_1 + 2\delta < \cdots < \cdots < 3\delta < 2\delta < \delta
  < \cdots < 2\delta+\alpha_0 < \delta+\alpha_0 < \alpha_0,
\end{equation*}
the imaginary roots are not placed consequently in other affine types. We use Theorem~\ref{thm:sln-general}
to establish two properties of this order for $\widehat{\fsl}_{n+1}$, see Propositions~\ref{prop:order-prop-1},~\ref{prop:order-prop-2}.


\subsection{Outline}\label{ssec:outline}
\

\noindent
The structure of the present paper is the following:
\begin{itemize}[leftmargin=0.5cm]

\medskip
\item[$\bullet$]
In Section~\ref{sec:SL-words basics}, we recall the notion of (standard) Lyndon words, their basic properties,
and the application to simple Lie algebras, following~\cite{LR} and~\cite{L}.

\medskip
\item[$\bullet$]
In Section~\ref{sec:affine-Leclerc}, we generalize Leclerc's algorithm of~\cite{L} from simple Lie algebras
to affine Lie algebras, and illustrate its application in the simplest case of $A^{(1)}_1$.

\medskip
\item[$\bullet$]
In Section~\ref{sec:A-type}, the heart of the paper, we compute $\aslaws$ for $A^{(1)}_n$ ($n\geq 2$) with any
order on the corresponding alphabet $\wI=\{0,1,\ldots,n\}$. The resulting set of $\aslaws$ is determined by
a finite subset of those of length $\leq n$ as well as manifests a compelling  periodicity~pattern.

\medskip
\item[$\bullet$]
In Section~\ref{sec:orders}, we use the explicit formulas for $\aslaws$ from Theorem~\ref{thm:sln-general}
to establish some properties of the order on $\widehat{\Delta}^{+,\mathrm{ext}}$,
induced via~\eqref{eqn:affine 1-to-1 intro} from the lexicographical order on the $\aslaws$.

\medskip
\item[$\bullet$]
In Appendix~\ref{sec:appendix}, we provide a link to the Python code and explain how it inductively computes
$\aslaws$ in all types and for any orders.

\end{itemize}


\subsection{Acknowledgement}\label{ssec:acknowl}
\

This paper is written as a part of the project under the Yulia's Dream initiative, a subdivision of the
MIT PRIMES program aimed at Ukrainian high-school students.
We are grateful to the referees for their useful suggestions that improved the exposition.
A.T.\ is deeply indebted to Andrei Negu\c{t} for sharing his invaluable insights, teaching the Lyndon word's theory,
and stimulating discussions through the entire project. A.T.\ is grateful to IHES for the hospitality and wonderful
working conditions in the Spring 2023, when the final version of this note was prepared. The work of A.T.\ was
partially supported by NSF Grants DMS-$2037602$ and DMS-$2302661$.


\medskip

\section{Lyndon words approach to Lie algebras}\label{sec:SL-words basics}

In this section, we recall the results of~\cite{LR} and~\cite{L} that provide a combinatorial construction of
an important basis of finitely generated Lie algebras, with the main application to the maximal nilpotent
subalgebra of a simple Lie algebra.


\subsection{Lyndon words}\label{ssec:L-words}
\

Let $I$ be a finite ordered alphabet, and let $I^*$ be the set of all finite length words in the alphabet $I$.
For $u=[i_1 \dots i_k]\in I^*$, we define its \emph{length} by $|u|=k$. We introduce the \emph{lexicographical order}
on $I^*$ in a standard way:
$$
  [i_1 \dots i_k] < [j_1 \dots j_l] \quad \text{if }\
  \begin{cases}
    i_1=j_1, \dots, i_a=j_a, i_{a+1} < j_{a+1} \text{ for some } a \geq 0 \\
      \text{\ or} \\
    i_1=j_1, \dots, i_k=j_k \text{ and } k < l
  \end{cases}.
$$

\begin{definition}\label{def:lyndon}
A word $\ell=[i_1\dots i_k]$ is called \textbf{Lyndon} if it is smaller than all of its cyclic permutations:
\begin{equation}\label{eq:lyndon}
  [i_1 \dots i_{a-1} i_a \dots i_k] < [i_a \dots i_k i_1 \dots i_{a-1}] \qquad \forall\, a \in \{2,\dots,k\}.
\end{equation}
\end{definition}

For a word $w = [i_1 \dots i_k]\in I^*$, the subwords:
\begin{equation}\label{eq:pre-suf}
  w_{a|} =  [i_1 \dots i_a] \qquad \text{and} \qquad w_{|a} = [i_{a+1} \dots i_k]
\end{equation}
with $0\leq a\leq k$ will be called a \emph{prefix} and a \emph{suffix} of $w$, respectively.
We call such a prefix or a suffix \emph{proper} if $0<a<k$. It is straightforward to show that
Definition~\ref{def:lyndon} is equivalent to the following one:

\begin{definition}\label{def:lyndon-2}
A word $w$ is \textbf{Lyndon} if it is smaller than all of its proper suffixes:
\begin{equation}\label{eq:Lyndon-condition-2}
  w < w_{|a} \qquad \forall\ 0<a<|w|.
\end{equation}
\end{definition}

As an immediate corollary, we obtain the following well-known result:

\begin{lemma}\label{lemma:lyndon}
If $\ell_1 < \ell_2$ are Lyndon, then $\ell_1\ell_2$ is also Lyndon, and so $\ell_1 \ell_2 < \ell_2 \ell_1$.
\end{lemma}

\begin{proof}
Let $\ell_1=i_1 i_2 \dots i_k$ and $\ell_2=i_{k+1}i_{k+2} \dots i_n$. Any cyclic permutation of the word
$\ell_1 \ell_2$ is of the form $u_j=i_j i_{j+1} \dots i_n i_1 i_2 \dots i_{j-1}$ with $1<j\leq k$ or $k<j\leq n$.
\begin{itemize}[leftmargin=0.5cm]

\item[$\bullet$] Case 1: $1<j\leq k$.
Since $\ell_1$ is Lyndon, we have ${\ell_1}_{|j-1}=i_j \dots i_k > \ell_1$ by~\eqref{eq:Lyndon-condition-2}.
As $|\ell_1|>|{\ell_1}_{|j-1}|$, there is $p\in \{j,j+1,\ldots,k\}$ such that $i_1=i_j,\ldots,i_{p-j}=i_{p-1}$
and $i_{p-j+1}<i_p$. This immediately implies the desired inequality $\ell_1 \ell_2<u_j$.

\item[$\bullet$] Case 2: $k<j\leq n$.
Since $\ell_2$ is Lyndon, we have ${\ell_2}_{|j-k-1}=i_j \dots i_n \geq \ell_2$ by~\eqref{eq:Lyndon-condition-2}
and so ${\ell_2}_{|j-k-1}=i_j \dots i_n > \ell_1$ as $\ell_2>\ell_1$. If $\ell_1$ is not a prefix of ${\ell_2}_{|j-k-1}$,
then $i_j=i_1, i_{j+1}=i_2,\ldots, i_{j+p-2}=i_{p-1}$ and $i_{j+p-1}>i_p$ for some $1\leq p\leq \min\{k,n-j+1\}$,
so that $\ell_1 \ell_2<u_j$. On the other hand, if $\ell_1$ is a prefix of ${\ell_2}_{|j-k-1}$, then
${\ell_2}_{|j-k-1}=\ell_1 i_{j+k} \dots i_n=\ell_1 {\ell_2}_{|j-1}$. In the latter case, the desired inequality
$\ell_1 \ell_2<u_j$ follows from ${\ell_2}_{|j-1}>\ell_2$, a consequence of~\eqref{eq:Lyndon-condition-2}.

\end{itemize}
This completes the proof of the first claim that $\ell_1\ell_2$ is Lyndon.
The second claim, the inequality $\ell_1\ell_2<\ell_2\ell_1$, follows now from~\eqref{eq:lyndon}.
\end{proof}

We recall the following two basic facts from the theory of Lyndon words:

\begin{proposition}\label{prop:cost.factor}(\cite[Proposition 5.1.3]{Lo})
Any Lyndon word $\ell$ has a factorization:
\begin{equation}\label{eqn:cost.factor}
  \ell = \ell_1 \ell_2
\end{equation}
defined by the property that $\ell_2$ is the longest proper suffix of $\ell$ which is also a Lyndon word.
Under these circumstances, $\ell_1$ is also a Lyndon word.
\end{proposition}

The factorization~\eqref{eqn:cost.factor} is called a \textbf{costandard factorization} of a Lyndon word.

\begin{proposition}\label{prop:canon.factor}(\cite[Proposition 5.1.5]{Lo})
Any word $w$ has a unique factorization:
\begin{equation}\label{eqn:canon.factor}
  w = \ell_1 \dots \ell_k
\end{equation}
where $\ell_1 \geq \dots \geq \ell_k$ are all Lyndon words.
\end{proposition}

The factorization~\eqref{eqn:canon.factor} is called a \textbf{canonical factorization}.


\subsection{Standard bracketing}\label{ssec:brackets}
\

Let $\fa$ be a Lie algebra generated by a finite set $\{e_i\}_{i\in I}$ labelled by the alphabet~$I$.

\begin{definition}\label{def:bracketing}
The standard bracketing of a Lyndon word $\ell$ is given inductively~by:
\begin{itemize}[leftmargin=0.7cm]

\item[$\bullet$]
$\sb[i]=e_i\in \fa$ for $i \in I$,

\item[$\bullet$]
$\sb[\ell] = [\sb[\ell_1], \sb[\ell_2]]\in \fa$, where $\ell=\ell_1\ell_2$ is
the costandard factorization~\eqref{eqn:cost.factor}.

\end{itemize}
\end{definition}

The major importance of this definition is due to the following result of Lyndon:

\begin{theorem}\label{thm:Lyndon theorem}(\cite[Theorem 5.3.1]{Lo})
If $\fa$ is a free Lie algebra in the generators $\{e_i\}_{i\in I}$, then the set
$\big\{\sb[\ell]|\ell\mathrm{-Lyndon\ word}\big\}$ provides a basis of $\fa$.
\end{theorem}


\subsection{Standard Lyndon words}\label{ssec:SL-words}
\

It is natural to ask if Theorem~\ref{thm:Lyndon theorem} admits a generalization to Lie algebras $\fa$ generated
by $\{e_i\}_{i\in I}$ but with some defining relations. The answer was provided a few decades later in~\cite{LR}.
To state the result, define ${_we},e_w\in U(\fa)$ for any $w\in I^*$:
\begin{itemize}[leftmargin=0.7cm]

\item[$\bullet$]
For a word $w = [i_1 \dots i_k]\in I^*$, we set
\begin{equation}\label{eqn:word}
  _we = e_{i_1} \dots e_{i_k} \in U(\fa)
\end{equation}

\item[$\bullet$]
For a word $w\in I^*$ with a canonical factorization $w=\ell_1 \dots \ell_k$ of~\eqref{eqn:canon.factor}, we set
\begin{equation}\label{eqn:bracket.word}
  e_w = e_{\ell_1} \dots e_{\ell_k} \in U(\fa)
\end{equation}
with $e_{\ell} = \sb[\ell]\in \fa$ for any Lyndon word $\ell$, cf.\ Definition~\ref{def:bracketing}.

\end{itemize}

It is well-known that the elements~\eqref{eqn:word} and~\eqref{eqn:bracket.word} are connected
by the following triangularity property:
\begin{equation}\label{eqn:upper}
  e_w=\sum_{v \geq w} c^v_w \cdot {_ve} \qquad \mathrm{with} \quad
  c^v_w\in \BZ \quad \mathrm{and}\quad c_w^w = 1.
\end{equation}

The following definition is due to \cite{LR}:

\begin{definition}\label{def:standard}
(a) A word $w$ is called \textbf{standard} if $_we$ cannot be expressed as a linear combination of
$_ve$ for various $v>w$, with $_we$ as in~\eqref{eqn:word}.

\noindent
(b) A Lyndon word $\ell$ is called \textbf{standard Lyndon} if $e_{\ell}$ cannot be expressed as
a linear combination of $e_m$ for various Lyndon words $m>\ell$, with $e_{\ell}=\sb[\ell]$ as above.
\end{definition}

The following result is nontrivial and justifies the above terminology:

\begin{proposition}\label{prop:standard}(\cite{LR})
A Lyndon word is standard iff it is standard Lyndon.
\end{proposition}

The major importance of this definition is due to the following result:

\begin{theorem}\label{thm:standard Lyndon theorem}(\cite[Theorem 2.1]{LR})
For any Lie algebra $\fa$ generated by a finite collection $\{e_i\}_{i\in I}$, the set
$\big\{\sb[\ell]|\ell\mathrm{-standard\ Lyndon\ word}\big\}$ provides a basis of $\fa$.
\end{theorem}


\subsection{Application to simple Lie algebras}\label{ssec:LR-bijection}
\

Let $\fg$ be a simple Lie algebra with a root system $\Delta=\Delta^+ \sqcup \Delta^-$.
Let $\{\alpha_i\}_{i\in I}\subset \Delta^+$ be the simple roots, and $Q=\bigoplus_{i\in I} \BZ\alpha_i$
be the root lattice. We endow $Q$ with the symmetric pairing $(\cdot,\cdot)$ so that the Cartan matrix
$(a_{ij})_{i,j\in I}$ of $\fg$ is given by $a_{ij} = \frac {2(\alpha_i,\alpha_j)}{(\alpha_i,\alpha_i)}$.
The Lie algebra $\fg$ admits the standard \textbf{root space decomposition}:
\begin{equation}\label{eq:root.decomp}
  \fg=\fh \oplus \bigoplus_{\alpha \in \Delta} \fg_{\alpha} \,,
  \quad \fh\subset\fg -\mathrm{Cartan\ subalgebra},
\end{equation}
with $\dim(\fg_{\alpha})=1$ for all $\alpha\in \Delta$.
We pick \emph{root vectors} $e_\alpha\in \fg_\alpha$ so that $\fg_\alpha=\BC\cdot e_\alpha$.

Consider the \emph{positive} Lie subalgebra $\fn^+=\bigoplus_{\alpha \in \Delta^+} \fg_{\alpha}$ of $\fg$.
Explicitly, $\fn^+$ is generated by $\{e_i\}_{i\in I}$ subject to the classical \emph{Serre} relations:
\begin{equation}\label{eqn:Serre}
  \underbrace{[e_i,[e_i,\dots,[e_i,e_j]\dots]]}_{1-a_{ij} \text{ Lie brackets}}\, =\, 0 \qquad \forall\ i\neq j.
\end{equation}
Let $Q^+=\bigoplus_{i\in I} \BZ_{\geq 0}\alpha_i$. The Lie algebra $\fn^+$ is naturally $Q^+$-graded via $\deg(e_i)=\alpha_i$.

Fix any order on the set $I$. According to Theorem~\ref{thm:standard Lyndon theorem}, $\fn^+$ has a basis consisting
of the $e_{\ell}$'s, as $\ell$ ranges over all standard Lyndon words. Evoking the above $Q^+$-grading of the Lie
algebra $\fn^+$, it is natural to define the grading of words as follows:
\begin{equation}\label{eqn:degree.word}
  \deg [i_1 \dots i_k] = \alpha_{i_1} + \dots + \alpha_{i_k} \in Q^+.
\end{equation}
Due to the decomposition~\eqref{eq:root.decomp} and the fact that the root vectors
$\{e_\alpha\}_{\alpha\in \Delta^+} \subset \fn^+$ all live in distinct degrees $\alpha \in Q^+$,
we conclude that there exists a bijection~\cite{LR}:
\begin{equation}\label{eqn:associated word}
  \ell \colon \Delta^+ \,\iso\, \big\{\slaws \big\}
\end{equation}
such that $\deg \ell(\alpha) = \alpha$ for all $\alpha \in \Delta^+$.
We call~\eqref{eqn:associated word} the \emph{Lalonde-Ram}'s bijection.


\subsection{Results of Leclerc and Rosso}\label{ssec:L-and-R}
\

The Lalonde-Ram's bijection~\eqref{eqn:associated word} was described explicitly in~\cite{L}.
To state the result, we recall that for a root $\gamma=\sum_{i\in I} n_i\alpha_i\in \Delta^+$,
its \emph{height} is $\hgt(\gamma)=\sum_{i\in I} n_i$.

\begin{proposition}\label{prop:Leclerc algorithm}(\cite[Proposition 25]{L})
The bijection $\ell$ is inductively given by:
\begin{itemize}[leftmargin=0.7cm]

\item[$\bullet$]
for simple roots, we have $\ell(\alpha_i)=[i]$

\item[$\bullet$]
for other positive roots, we have the following \emph{Leclerc's} algorithm:
\begin{equation}\label{eqn:inductively}
  \ell(\alpha) =
  \max\left\{ \ell(\gamma_1)\ell(\gamma_2) \Big|
               \alpha=\gamma_1+\gamma_2 \,,\, \gamma_1,\gamma_2\in \Delta^+ \,,\, \ell(\gamma_1) < \ell(\gamma_2) \right\}.
\end{equation}
\end{itemize}
\end{proposition}

\noindent
Formula~\eqref{eqn:inductively} recovers $\ell(\alpha)$ once we know $\ell(\gamma)$ for all
$\{\gamma\in \Delta^+ \,|\, \mathrm{ht}(\gamma)<\mathrm{ht}(\alpha)\}$.

\begin{remark}
While Lalonde-Ram computed explicitly the $\slaws$ for any simple $\fg$ and a specific order
in~\cite[Theorem 3.4]{LR}, the above Leclerc's algorithm allows to find $\slaws$ for any simple $\fg$
and any ordering of its simple roots. Moreover, this algorithm is easy to program on a computer.
\end{remark}

We shall also need one more important property of $\ell$. To the end, let us recall:

\begin{definition}\label{def:convex}
A total order on the set of positive roots $\Delta^+$ is \textbf{convex} if:
\begin{equation}\label{eqn:convex}
  \alpha < \alpha+\beta < \beta
\end{equation}
for all $\alpha < \beta \in \Delta^+$ such that $\alpha+\beta$ is also a root.
\end{definition}

\begin{remark}
It is well-known (\cite{P}) that convex orders on $\Delta^+$ are in bijection with reduced decompositions
of the longest element in the Weyl group of $\fg$.
\end{remark}

The following result is~\cite[Proposition 26]{L}, where it is attributed to the preprint of Rosso~\cite{R2}
(a detailed proof can be found in~\cite[Proposition 2.34]{NT}):

\begin{proposition}\label{prop:fin.convex}
Consider the order on $\Delta^+$ induced from the lexicographical order on $\slaws$:
\begin{equation}\label{eqn:induces}
  \alpha < \beta \quad \Longleftrightarrow \quad \ell(\alpha) < \ell(\beta)  \ \ \mathrm{ lexicographically}.
\end{equation}
This order is convex.
\end{proposition}

\begin{remark}
We note that both Proposition~\ref{prop:Leclerc algorithm} and Proposition~\ref{prop:fin.convex}
are of crucial importance for the further application to quantum groups $U_q(\fg)$, see~\cite{L}.
\end{remark}


\medskip

\section{Generalization to affine Lie algebras}\label{sec:affine-Leclerc}

In this section, we generalize Proposition~\ref{prop:Leclerc algorithm} to the case of affine Lie algebras~$\fg$.
As an example, we compute all $\aslaws$ for $\fg$ of type $A^{(1)}_1$.


\subsection{Affine Lie algebras}\label{ssec:affineLie}
\

In this section, we consider the next simplest class of Kac-Moody Lie algebras after the simple ones,
the affine Lie algebras. Let $\fg$ be a simple finite-dimensional Lie algebra, $\{\alpha_i\}_{i\in I}$ be
the simple roots, and $\theta\in \Delta^+$ be the highest root (with the maximal value of $\hgt(\theta)$).
We define $\wI = I \sqcup \{0\}$. Consider the affine root lattice $\wQ=Q \times \BZ$ with the generators
$\{(\alpha_i,0)\}_{i\in I}$ and $\alpha_0:=(-\theta,1)$. We endow $\wQ$ with the symmetric pairing defined by:
\begin{equation}\label{eq:affine pairing}
  \big((\alpha,n),(\beta,m)\big)=(\alpha,\beta) \qquad \forall\ \alpha,\beta\in Q \,,\, n,m \in \BZ.
\end{equation}
This leads to the affine Cartan matrix $(a_{ij})_{i,j\in \wI}$ and the \textbf{affine Lie algebra} $\widehat{\fg}$.
The associated affine root system $\wDelta=\wDelta^+ \sqcup \wDelta^-$ has the following explicit description:
\begin{align}
  & \wDelta^+ = \big\{ \Delta^+ \times \BZ_{\geq 0} \big\}
    \sqcup \big\{ 0 \times \BZ_{>0} \big\}
    \sqcup \big\{ \Delta^- \times \BZ_{>0} \big\},
  \label{eqn:hat plus} \\
  & \wDelta^- = \big\{ \Delta^- \times \BZ_{\leq 0} \big\}
    \sqcup \big\{ 0 \times \BZ_{<0} \big\}
    \sqcup \big\{ \Delta^+ \times \BZ_{<0} \big\},
  \label{eqn:hat minus}
\end{align}
where $\BZ_{\geq 0}$, $\BZ_{>0}$, $\BZ_{\leq 0}$, $\BZ_{<0}$ denote the obvious subsets of $\BZ$. Here,
$\delta=\alpha_0+\theta=(0,1) \in Q \times \BZ$ is the minimal \emph{imaginary root} of the affine root system
$\wDelta$. With this notation, we have the following root space decomposition, cf.~\eqref{eq:root.decomp}:
\begin{equation}\label{eq:aff.root.decomp}
  \widehat{\fg}=\widehat{\fh} \oplus \bigoplus_{\alpha \in \wDelta} \widehat{\fg}_{\alpha} \,,
  \quad \widehat{\fh}\subset \widehat{\fg} -\mathrm{Cartan\ subalgebra}.
\end{equation}

Let us now recall another realization of $\widehat{\fg}$. To this end, consider the Lie algebra
\begin{equation}\label{eq:loops}
\begin{split}
  & \widetilde{\fg}=\fg\otimes \BC[t,t^{-1}]\oplus \BC\cdot \mathsf{c}
    \quad \mathrm{with\ a\ Lie\ bracket\ given\ by}\\
  & [x\otimes t^n,y\otimes t^m]=[x,y]\otimes t^{n+m}+n\delta_{n,-m}(x,y)\cdot \mathsf{c}
    \quad \mathrm{and} \quad
    [\mathsf{c},x\otimes t^n]=0
\end{split}
\end{equation}
where $x,y\in \fg$, $n,m\in \BZ$, and $(\cdot,\cdot)\colon \fg\times\fg\to \BC$ is
a non-degenerate invariant pairing.

The rich theory of affine Lie algebras is mainly based on the following key result:

\begin{claim}
There exists a Lie algebra isomorphism:
\begin{equation}\label{eq:affine-iso}
  \widehat{\fg}\ \iso\ \widetilde{\fg}
\end{equation}
determined on the generators by the following formulas:
\begin{align*}
  & e_i \mapsto e_i \otimes t^0 & &
    f_i \mapsto f_i \otimes t^0 & &
    h_i \mapsto h_i \otimes t^0 \qquad \forall\, i\in I \,,\\
  & e_0 \mapsto f_\theta \otimes t^1 & &
    f_0 \mapsto e_\theta \otimes t^{-1} & &
    h_0 \mapsto [f_\theta,e_\theta]\otimes t^0 + (f_\theta,e_\theta)\cdot \mathsf{c} \,,
\end{align*}
where $e_\theta$ and $f_\theta$ are root vectors of degrees $\theta$ and $-\theta$, respectively.
\end{claim}

In view of this result, we can explicitly describe the root subspaces from~\eqref{eq:aff.root.decomp}:
\begin{align}
  & \widehat{\fg}_{(\alpha,k)}=\fg_\alpha\otimes t^k \quad \mathrm{for} \
    (\alpha,k)\in \wDelta^{+,\mathrm{re}}:=\big\{\Delta^+ \times \BZ_{\geq 0} \big\} \sqcup \big\{ \Delta^- \times \BZ_{>0} \big\},
  \label{eqn:real-roots}\\
  & \widehat{\fg}_{k\delta}=\fh \otimes t^k \quad \mathrm{for} \
    k\delta\in \wDelta^{+,\mathrm{im}}:=\big\{ 0 \times \BZ_{>0} \big\}.
  \label{eqn:imaginary-roots}
\end{align}
As $\dim(\fg_\alpha)=1$ for any $\alpha\in \Delta$ and $\dim(\fh)=\rk(\fg)=|I|$, we thus obtain:
\begin{equation}\label{eq:aff-dim}
  \dim(\widehat{\fg}_\alpha)=1 \quad \forall\ \alpha\in \wDelta^{+,\mathrm{re}} \,, \qquad
  \dim(\widehat{\fg}_\alpha)=|I| \quad \forall\ \alpha\in \wDelta^{+,\mathrm{im}}.
\end{equation}

\noindent
\textbf{Notation:} In what follows, we shall always simply write $xt^n$ instead of $x\otimes t^n$.


\subsection{Affine standard Lyndon words}\label{ssec:aslaws}
\

It is natural to ask if the above results can be generalized to affine Lie algebras $\widehat{\fg}$.
On the Lie algebraic side, we consider only the \emph{positive} subalgebra
  $\widehat{\fn}^+=\bigoplus_{\alpha\in \wDelta^+} \widehat{\fg}_{\alpha}$.
Thus, $\widehat{\fn}^+$ is generated by $\{e_i\}_{i\in \wI}$ subject to the Serre
relations~\eqref{eqn:Serre} for $i\ne j\in \wI$. On the combinatorial side, we consider
the finite alphabet $\wI$ with any order on it, which allows to define Lyndon and standard Lyndon words
(with respect to $\widehat{\fn}^+$). We shall use the term \textbf{$\aslaws$} in the present setup.

The key difference with the case of simple $\fg$ is that some root subspaces are not one-dimensional,
see~\eqref{eq:aff-dim}. Thus, we do not get such a simple bijection as~\eqref{eqn:associated word} for
simple Lie algebras. However, the degree reasoning as in Subsection~\ref{ssec:LR-bijection} implies that
there is a unique $\aslaw$ in each real degree $\alpha\in \wDelta^{+,\mathrm{re}}$, denoted by $\SL(\alpha)$,
and $|I|$ $\aslaws$ in each imaginary degree $\alpha\in \wDelta^{+,\mathrm{im}}$, denoted by
$\SL_1(\alpha),\ldots,\SL_{|I|}(\alpha)$, listed in the decreasing order.

The main result of this section is the following \emph{generalized Leclerc's} algorithm:

\begin{proposition}\label{prop:generalized Leclerc}
The $\aslaws$ (with respect to $\widehat{\fn}^+$) are determined inductively by the following rules:

\noindent
(a) For simple roots, we have $\SL(\alpha_i)=[i]$. For other real $\alpha\in \wDelta^{+,\mathrm{re}}$, we have:
\begin{equation}\label{eq:generalized Leclerc}
  \SL(\alpha) =
  \max\left\{\SL_*(\gamma_1)\SL_*(\gamma_2) \Big|
   \substack{\alpha=\gamma_1+\gamma_2,\, \gamma_k\in \wDelta^+\\ \SL_*(\gamma_1)<\SL_*(\gamma_2)\\
             [\sb[\SL_*(\gamma_1)],\sb[\SL_*(\gamma_2)]]\neq 0} \right\},
\end{equation}
where $\SL_*(\gamma)$ denotes $\SL(\gamma)$ for $\gamma\in \wDelta^{+,\mathrm{re}}$
and any of $\{\SL_k(\gamma)\}_{k=1}^{|I|}$ for $\gamma\in \wDelta^{+,\mathrm{im}}$.

\noindent
(b) For imaginary $\alpha\in \wDelta^{+,\mathrm{im}}$, the corresponding $|I|$ $\aslaws$ $\{\SL_k(\alpha)\}_{k=1}^{|I|}$
are the $|I|$ lexicographically largest words from the list as in the right-hand side of~\eqref{eq:generalized Leclerc}
whose standard bracketings are linearly independent.
\end{proposition}

\begin{remark}\label{rem:gen Lyndon rmk}
Since $[\widehat{\fg}_{a\delta},\widehat{\fg}_{b\delta}]=0$ for any $a,b>0$, we shall assume that
$\gamma_1,\gamma_2\in \wDelta^{+,\mathrm{re}}$ when applying part~(b). Thus, $\SL_1(\alpha)$ is given precisely
by~\eqref{eq:generalized Leclerc}, $\SL_2(\alpha)$ is the next largest word among the above concatenations
whose bracketing is not a multiple of $\sb[\SL_1(\alpha)]$, and so on, up to $\SL_{|I|}(\alpha)$ which is
the largest of the remaining concatenations whose standard bracketing is linearly independent with
$\{\sb[\SL_k(\alpha)]\}_{k=1}^{|I|-1}$.
\end{remark}

\begin{proof}[Proof of Proposition~\ref{prop:generalized Leclerc}]
(a) Consider the costandard factorization $\SL(\alpha)=\ell_1\ell_2$ as in~\eqref{eqn:cost.factor}. Then,
$\ell_1=\SL_*(\gamma_1), \ell_2=\SL_*(\gamma_2)$ for some $\gamma_1,\gamma_2\in \wDelta^+$ and $\ell_1<\ell_2$.
Finally, $\sb[\SL(\alpha)]\ne 0$ implies that $[\sb[\SL_*(\gamma_1)],\sb[\SL_*(\gamma_2)]]\ne 0$. Therefore,
$\ell_1\ell_2$ is an element from the right-hand side of~\eqref{eq:generalized Leclerc}. It thus remains
to show that $\SL(\alpha)$ is $\geq$ any concatenation $\SL_*(\gamma_1)\SL_*(\gamma_2)$ featuring in
the right-hand side of~\eqref{eq:generalized Leclerc}.

The proof of the latter is completely analogous to that of~\cite[Proposition 2.23]{NT}.
Consider any $\gamma_1, \gamma_2 \in \wDelta^+$ such that $\gamma_1+\gamma_2=\alpha$. Let us write
$\ell_1 = \SL_*(\gamma_1)$, $\ell_2 = \SL_*(\gamma_2)$, $\ell=\SL(\alpha)$. We may assume, without loss
of generality, that $\ell_1 < \ell_2$. Evoking the notations of Subsection~\ref{ssec:SL-words}, we have:
\begin{equation}\label{eqn:fi}
  \sb[\ell_k]=e_{\ell_k} = \sum_{v_k \geq \ell_k} c^{v_k}_{\ell_k} \cdot {_{v_k}e}
\end{equation}
$\forall\, k \in \{1,2\}$, due to the triangularity property \eqref{eqn:upper}.
Thus, due to the degree reasons (see~\cite[Footnote 2]{NT}), we get:
\begin{equation}\label{eqn:hi}
  \sb[\ell_1] \sb[\ell_2] = e_{\ell_1} e_{\ell_2} \, = \sum_{v \geq \ell_1\ell_2} x_v \cdot {_ve}
\end{equation}
for some coefficients $x_v$. As a consequence of $\ell_2\ell_1>\ell_1\ell_2$ (Lemma~\ref{lemma:lyndon}),
we also~get:
\begin{equation}\label{eqn:hi-2}
  \sb[\ell_2] \sb[\ell_1] = e_{\ell_2} e_{\ell_1} \, = \sum_{v \geq \ell_1\ell_2} x'_v \cdot {_ve}
\end{equation}
for some coefficients $x'_v$. Hence, we obtain the following formula for the commutator:
\begin{equation}\label{eqn:bracket-commutator-0}
  [\sb[\ell_1],\sb[\ell_2]]=[e_{\ell_1}, e_{\ell_2}] \, = \sum_{v \geq \ell_1\ell_2} y_v \cdot {_ve}
\end{equation}
for various coefficients $y_v$. Furthermore, we may restrict the sum above to standard $v$'s, since by
the very definition of this notion, any $_ve$ can be inductively written as a linear combination of $_ue$'s
for standard $u \geq v$. By the same reason, we may restrict the right-hand side of \eqref{eqn:upper} to
standard $v$'s, and conclude that $\{e_w\}_{w - \text{standard}}$ provide a basis of $U(\widehat{\fn}^+)$
which is upper triangular in terms of the basis $\{_we\}_{w - \text{standard}}$. With the above observations
in mind, \eqref{eqn:bracket-commutator-0} implies:
\begin{equation}\label{eqn:bracket-commutator-1}
  [\sb[\ell_1],\sb[\ell_2]]=[e_{\ell_1}, e_{\ell_2}] \ =
  \mathop{\sum_{v \geq \ell_1\ell_2}}_{v-\text{standard}} z_v \cdot e_v
\end{equation}
for various $z_v$. Meanwhile, the assumption $[\sb[\ell_1],\sb[\ell_2]]\ne 0$ and
$\widehat{\fg}_\alpha=\BC\cdot \sb[\ell]$ imply:
\begin{equation}\label{eqn:bracket-commutator-2}
  [\sb[\ell_1],\sb[\ell_2]]=[e_{\ell_1}, e_{\ell_2}] \in \BC^\times \cdot e_{\ell}.
\end{equation}
As $\{e_v\}_{v-\text{standard}}$ is a basis of $U(\widehat{\fn}^+)$,
comparing~(\ref{eqn:bracket-commutator-1},~\ref{eqn:bracket-commutator-2}) we obtain $\ell \geq \ell_1\ell_2$,
precisely as claimed above.

(b) The proof of part (b) is completely analogous to that of part (a), with the only difference that
we need to find $|I|$ $\aslaws$. Thus, we just use Definition~\ref{def:standard}(b) to complement
the above argument in the present setup.
\end{proof}


\subsection{Affine standard Lyndon words in type $A^{(1)}_1$}\label{ssec:sl2}
\

As the first simplest example, let us compute $\aslaws$ in the simplest case of $A^{(1)}_1$, which corresponds
to the affinization $\widehat{\fsl}_2$ of the unique rank $1$ simple Lie algebra $\fsl_2$. In this case:
there are two simple roots $\alpha_0,\alpha_1$ and $\delta=\alpha_0+\alpha_1$. The set of positive roots is
  $\widehat{\Delta}^+=\{k\delta+\alpha_1,k\delta+\alpha_0,(k+1)\delta|k\geq 0\}$.
Without loss of generality, we can assume that $1<0$, due to the $0\leftrightarrow 1$ symmetry.

\begin{proposition}\label{prop:sl2-case}
The $\aslaws$ for $\widehat{\fsl}_2$ with the order $1<0$ on the corresponding alphabet $\wI=\{0,1\}$ are:
\begin{itemize}[leftmargin=0.7cm]

\item[$\bullet$]
For $k\geq 1$, we have:
\begin{align}
  & \SL(k\delta+\alpha_1)= 1\underbrace{10}_{k \, \rtim} ,
    \label{eq:sl2-one} \\
  & \SL(k\delta+\alpha_0)= \underbrace{10}_{k \, \rtim}0 ,
    \label{eq:sl2-two} \\
  & \SL((k+1)\delta)= 1\underbrace{10}_{k \, \rtim}0 .
    \label{eq:sl2-three}
\end{align}

\item[$\bullet$]
For the remaining roots, we have:
\begin{equation}\label{eq:sl2-four}
  \SL(\alpha_1)= 1 \,,\quad
  \SL(\alpha_0)= 0 \,,\quad
  \SL(\delta)=10.
\end{equation}

\end{itemize}
\end{proposition}

\begin{proof}
Formulas~\eqref{eq:sl2-four} are obvious, while the proof of~\eqref{eq:sl2-one}--\eqref{eq:sl2-three}
will proceed by induction on $k$. The base $k=1$ case is easy. We shall now prove the induction step, just
by using the generalized Leclerc's algorithm from Proposition~\ref{prop:generalized Leclerc}.

1) Root $\alpha=k\delta+\alpha_1$.
Any decomposition $\alpha=\gamma_1+\gamma_2$ has the following form:
  $\{\gamma_1,\gamma_2\}=\{a\delta, b\delta+\alpha_1 \,|\, a+b=k, 1\leq a\leq k\}$.
By the induction hypothesis:
  \[ \SL(b\delta+\alpha_1)=1\underbrace{10}_{b \, \rtim} < \, 1\underbrace{10}_{(a-1) \, \rtim}0=\SL(a\delta). \]
Following~\eqref{eq:generalized Leclerc}, consider the lexicographically largest word among all possible concatenations
$1\underbrace{10}_{b \, \rtim}1\underbrace{10}_{(a-1) \, \rtim}0$, which is $1\underbrace{10}_{k \, \rtim}$.
Let us show by induction on $k$ that its standard bracketing is $(-2)^kE_{12} t^k$, thus completing the proof of~\eqref{eq:sl2-one}:
\begin{equation*}
  \sb[1\underbrace{10}_{k \, \rtim}]=[\sb[1\underbrace{10}_{(k-1) \, \rtim}],\sb[10]]=
  [(-2)^{k-1}E_{12} t^{k-1},(E_{11}-E_{22}) t]=(-2)^kE_{12} t^k.
\end{equation*}

2) Root $\alpha=k\delta+\alpha_0$. Any decomposition $\alpha=\gamma_1+\gamma_2$ has the following form:
$\{\gamma_1,\gamma_2\}=\{a\delta, b\delta+\alpha_0 \,|\, a+b=k, 1\leq a\leq k\}$. As in 1), one combines the
inductive hypothesis with~\eqref{eq:generalized Leclerc} to find: $\SL(\alpha)=\underbrace{10}_{k \, \rtim}0$
with the standard bracketing
\begin{equation*}
  \sb[\underbrace{10}_{k \, \rtim}0]=(-2)^kE_{21}t^{k+1}.
\end{equation*}

3) Let us now treat the imaginary root $\alpha=(k+1)\delta$. As $\rk(\fsl_2)=1$, there is
only one $\aslaw$ in degree $\alpha$, which can be found by~\eqref{eq:generalized Leclerc}.
Any decomposition $\alpha=\gamma_1+\gamma_2$ that contributes into $\SL(\alpha)$ is of the~form:
$\{\gamma_1,\gamma_2\}=\{a\delta+\alpha_1, b\delta+\alpha_0 \,|\, a+b=k, 0\leq a\leq k\}$.
By the induction hypothesis:
  \[ \SL(a\delta+\alpha_1)=1\underbrace{10}_{a \, \rtim}<\underbrace{10}_{b \, \rtim}0=\SL(b\delta+\alpha_0). \]
Following~\eqref{eq:generalized Leclerc}, consider the lexicographically largest word among all the corresponding
concatenations $\SL(a\delta+\alpha_1)\SL(b\delta+\alpha_0)=1\underbrace{10}_{k \, \rtim}0$, which completes
the proof of~\eqref{eq:sl2-three}. Let us evaluate its standard bracketing:
\begin{equation*}
  \sb[1\underbrace{10}_{k \, \rtim}0]=[\sb[1],\sb[\underbrace{10}_{k \, \rtim}0]]=
  [E_{12},(-2)^kE_{21}t^{k+1}]=(-2)^k (E_{11}-E_{22})t^{k+1}.
\end{equation*}
This completes the proof of the induction step.
\end{proof}


\medskip

\section{Affine standard Lyndon words in type $A^{(1)}_n$ for $n\geq 2$}\label{sec:A-type}

In this section, we describe $\aslaws$ in affine type $A^{(1)}_n$ for $n\geq 2$ and any order on
$\wI=\{0,1,2,\ldots,n\}$. First, we treat the simplest case (of the \emph{standard order}) to which
Proposition~\ref{prop:generalized Leclerc} can be easily applied.
We then crucially utilize the convexity property of Proposition~\ref{prop:fin.convex} to derive
the structure of $\aslaws$ for an arbitrary order on~$\wI$.


\subsection{Standard order}\label{ssec:standard}
\

We start by computing all $\aslaws$ for type $A^{(1)}_n$ with
\begin{equation}\label{eq:stand-order}
  \mathrm{the}\ \textit{standard}\ \mathrm{order\ on}\ \wI\colon \quad 1<2<3<\dots<n<0.
\end{equation}
There are $n+1$ simple roots $\alpha_0,\alpha_1,\ldots,\alpha_n$, and $\delta=\alpha_0+\alpha_1+\dots+\alpha_n$.
It is convenient to place the letters of the alphabet $\wI=\{0,1,2,\ldots,n\}$ on a circle counterclockwise.
For any counterclockwise oriented arch from $i$ to $j$, we define
\begin{equation}\label{eq:arch-root}
  \alpha_{i \to j}:=\alpha_i+\alpha_{i+1}+\dots+\alpha_j\in Q.
\end{equation}
Using this notation, the positive affine roots can be explicitly described as follows:
\begin{equation}
  \wDelta^+=\big\{k\delta+\alpha_{i \to j}, (k+1)\delta \, \big|\, k\geq 0 \,,\, i,j\in \wI \,,\, j\ne \overline{i-1}\big\}.
\end{equation}
Here, for any $k\in \BZ$ we define $\overline{k}\in \wI$ via:
\begin{equation}\label{eq:residue}
  \overline{k}:=k \ \mathrm{mod}\ (n+1).
\end{equation}
We also use $[i\to j)$ to denote all letters on the arch from $i$ (included) to $j$ (excluded):
\begin{equation}\label{eq:arch-elements}
  [i\to j) := \big\{i,\overline{i+1},\ldots,\overline{j-1}\big\}.
\end{equation}

\begin{theorem}\label{thm:sln-standard}
The $\aslaws$ for $\widehat{\fsl}_{n+1}$ with the standard order $1<2<\dots<n<0$ on the corresponding
alphabet $\wI=\{0,1,\ldots,n\}$ are as follows:
\begin{itemize}[leftmargin=0.7cm]

\item[$\bullet$]
For $k\geq 1$, we have:
\begin{align}
  & \SL(k\delta+\alpha_{i \to j}) = \underbrace{10n \ldots i\, 23 \ldots \overline{i-1}}_{k \, \rtim} i\,\overline{i+1} \ldots j \,,
    \quad \mathrm{for} \ 2<i\leq j\leq 0 ,
    \label{eq:stand-one} \\
  & \SL(k\delta+\alpha_2)=\underbrace{10n \ldots 32}_{k \, \rtim} 2,
    \label{eq:stand-two} \\
  & \SL(k\delta+\alpha_{2 \to j})=
    \begin{cases}
      \underbrace{10n \ldots 32}_{\frac{k}{2} \, \rtim} 2\underbrace{10n \ldots 32}_{\frac{k}{2} \, \rtim} 34 \ldots j
        & \mathrm{if}\ 2\mid\, k\\
      \underbrace{10n \ldots 32}_{\frac{k+1}{2} \, \rtim} 34 \ldots j\underbrace{10n \ldots 32}_{\frac{k-1}{2} \, \rtim}2
        & \mathrm{if}\ 2\nmid\, k
    \end{cases} \,, \quad \mathrm{for} \ 2<j\leq 0 ,
    \label{eq:stand-three} \\
  & \SL(k\delta+\alpha_{1\to i})=
    123 \ldots n\underbrace{1023 \ldots n}_{(k-1) \, \rtim} 1023 \ldots i \,, \quad \mathrm{for}\ 1\leq i<0,
    \label{eq:stand-four}
\end{align}
\begin{multline}\label{eq:stand-five}
  \SL(k\delta+\alpha_{j \to i})=\SL(k\delta + \alpha_{j \to 0} + \alpha_{1 \to i}) =
    \hfill \mathrm{for} \ i<\overline{i+1}<j \\
  10n \ldots j\, 23 \ldots \overline{j-2}\, \underbrace{10n \ldots \overline{j-1}\, 23 \ldots \overline{j-2}}_{(k-1) \, \rtim}
  10n \ldots \overline{j-1}\, 23 \ldots i ,
\end{multline}
\begin{multline}\label{eq:stand-six}
  \SL_n((k+1)\delta)=
    123 \ldots n \underbrace{1023 \ldots n}_{k \, \rtim}0,\\
  \SL_r((k+1)\delta)=
    10n \ldots \overline{r+2}\, 23 \ldots r
    \underbrace{10n \ldots (r+1)23 \ldots r}_{k \, \rtim}(r+1) \,,\ \mathrm{for}\ r<n .
\end{multline}

\item[$\bullet$]
For the remaining roots, we have:
\begin{align}
  & \SL(\alpha_{i \to j})=i(i+1) \ldots j
    \,, \quad \mathrm{for} \ i\leq j \ \mathrm{and}\ (i,j)\ne (1,0) ,
    \label{eq:stand-seven} \\
  & \SL(\alpha_{j \to i})=\SL(\alpha_{j \to 0} + \alpha_{1 \to i}) = 10n \ldots j\, 23 \ldots i
    \,, \quad \mathrm{for} \ i<\overline{i+1}<j ,
    \label{eq:stand-eight} \\
  & \SL_r(\delta)=10 \ldots \overline{r+2}\, 23 \ldots \overline{r+1} \,,\quad \mathrm{for}\ 1\leq r\leq n .
    \label{eq:stand-nine}
\end{align}

\end{itemize}
\end{theorem}

\begin{proof}
The proof will proceed by induction on the height $\hgt(\alpha)$. Let $h=\hgt(\delta)=n+1$ be
the \emph{Coxeter number} of $\fsl_{n+1}$. The base of induction is $\hgt(\alpha)<2h$, that is,
$k=0,1$ cases for real roots $k\delta+\alpha_{i \to j}$ and $k=0$ case for imaginary roots $(k+1)\delta$.

\medskip
\noindent
\underline{Base of Induction (part I)}

First, let us verify~\eqref{eq:stand-seven}--\eqref{eq:stand-nine} and find bracketings of the corresponding words.

\medskip
\noindent
$\bullet$ Proof of~\eqref{eq:stand-seven}.

Consider the costandard factorization $\ell=\ell_1\ell_2$ of any Lyndon word $\ell$ with $\deg \ell=\alpha_{i \to j}$.
As $i<i+1$ are the two smallest letters of $\ell$, the word $\ell_1$ starts with $i$ and $\ell_2$ starts with $i+1$.
If furthermore $\ell$ is standard Lyndon, so is $\ell_1$, hence, $\deg \ell_1\in \wDelta^+$. For degree reasons,
this is only possible if $\ell_1=i$ and $\deg\ell_2=\alpha_{(i+1)\to j}$. Arguing by induction on the height of
$\alpha_{i\to j}$, we thus immediately derive the desired formula~\eqref{eq:stand-seven}. Moreover, we also
inductively get the explicit formula for the corresponding standard bracketing:
\begin{equation*}
  \sb[\SL(\alpha_{i \to j})] = \sb[i(i+1)\ldots j]=[\sb[i],\sb[(i+1)\ldots j]]=
  \begin{cases}
     E_{i,j+1}t^0 & \ \mathrm{if}\ j\leq n \\
     E_{i,1}t   & \ \mathrm{if}\ j=0
  \end{cases}.
\end{equation*}

\noindent
\textbf{Notation:} Henceforth, we shall use the matrix $E_{0,p}$ to denote $E_{n+1,p}$.

\medskip
\noindent
$\bullet$ Proof of~\eqref{eq:stand-eight} for $i=1$.

In this case, we shall rather use~\eqref{eq:generalized Leclerc} and argue by induction on the height of
$\alpha_{j \to 1}$ (i.e.\ a descending induction of $j\in \wI$). The possible decompositions of $\alpha_{j\to 1}$
into the (unordered) sum of two positive roots are as follows:
\begin{equation*}
  \alpha_{j\to 1}=\alpha_{j\to k} + \alpha_{\overline{k+1}\to 1} \quad (j\leq k\leq n) \,,\qquad
  \alpha_{j\to 1}=\alpha_{j\to 0} + \alpha_1.
\end{equation*}
Combining the induction hypothesis with formula~\eqref{eq:stand-seven}, we get the following list of
concatenated words featuring in the right-hand side of~\eqref{eq:generalized Leclerc} for $\alpha=\alpha_{j \to 1}$:
\begin{equation}\label{eq:list-stand-eight}
  10n \ldots \overline{k+1}\, j\, \overline{j+1} \ldots k \quad (j\leq k\leq 0).
\end{equation}
Clearly, $10n \ldots j$ is the lexicographically largest word from this list~\eqref{eq:list-stand-eight}.
Let us evaluate its standard bracketing:
\begin{equation*}
  \sb[10n \ldots j] = [\sb[10n \ldots \overline{j+1}],\sb[j]]=[(-1)^{n-j-1}E_{j+1,2}t,E_{j,j+1}]=(-1)^{n-j}E_{j,2}t,
\end{equation*}
where we use the induction hypothesis for the value of $\sb[10n \ldots (j+1)]$. We thus obtain
$\SL(\alpha_{j\to 1})=10n \ldots j$ as claimed in~\eqref{eq:stand-eight}, since the bracketing is nonzero.

\medskip
\noindent
$\bullet$ Proof of~\eqref{eq:stand-eight} for $i>1$.

In the present case, we can argue alike in our verification of~\eqref{eq:stand-seven}. Consider the costandard
factorization $\SL(\alpha_{j \to i})=\ell_1\ell_2$. Since $1<2$ are the two smallest letters, $\ell_1$ starts
with $1$ and $\ell_2$ starts with $2$. Moreover, we have $\deg \ell_1,\deg \ell_2\in \wDelta^+$. For degree reasons,
this is only possible if $\deg\ell_1=\alpha_{j \to 1}$ and $\deg\ell_2=\alpha_{2 \to i}$. We thus have
$\ell_1=10n \ldots j$ and $\ell_2=23 \ldots i$ by above, and~\eqref{eq:stand-eight} follows. Furthermore:
\begin{equation*}
  \sb[\SL(\alpha_{j \to i})] = \sb[10n \ldots j\, 23\ldots i]= [\sb[10n \ldots j],\sb[23\ldots i]]=(-1)^{n-j}E_{j,i+1}t.
\end{equation*}

\noindent
$\bullet$ Proof of~\eqref{eq:stand-nine}.

Let us now treat the case of the smallest imaginary root $\delta$. The possible decompositions of
$\delta$ into the (unordered) sum of two positive roots are as follows:
\begin{equation*}
  \delta=\alpha_{1\to i} + \alpha_{\overline{i+1}\to 0} \ \ (1\leq i\leq n) \,, \qquad
  \delta=\alpha_{i\to j}+ \alpha_{\overline{j+1}\to (i-1)} \ \ (2\leq i\leq j\leq n).
\end{equation*}
Using already verified formulas~\eqref{eq:stand-seven} and~\eqref{eq:stand-eight}, we thus get the following list
of concatenated words featuring in the right-hand side of~\eqref{eq:generalized Leclerc} for $\alpha=\delta$:
\begin{equation*}
  12 \ldots i\, \overline{i+1} \ldots n0 \,, \qquad
  10n \ldots \overline{j+1}\, 23 \ldots (i-1)i(i+1) \ldots j \quad (2\leq j\leq n).
\end{equation*}
Since this list contains exactly $n$ different words (we note the independence of $i$), all of them are
precisely $\SL_1(\delta), \ldots,\SL_n(\delta)$. Ordering them lexicographically, we derive the desired
formula~\eqref{eq:stand-nine}. Let us compute their standard bracketings:
\begin{equation}\label{eq:delta-brack-standard}
\begin{split}
  & \sb[\SL_r(\delta)]=\sb[10 \ldots \overline{r+2}\, 23 \ldots \overline{r+1}]=
    [\sb[10 \ldots \overline{r+2}],\sb[23 \ldots \overline{r+1}]]=\\
  & \quad [(-1)^{n-r}E_{r+2,2}t,E_{2,r+2}]=(-1)^{n-r+1}(E_{22}-E_{r+2,r+2})t \quad \mathrm{if}\ r\leq n-1,\\
  & \sb[\SL_n(\delta)]=\sb[123 \ldots n0]=[\sb[1],\sb[23 \ldots n0]]=(E_{11}-E_{22})t.
\end{split}
\end{equation}

\noindent
\underline{Base of Induction (part II)}

As a continuation of the induction base, let us now verify~\eqref{eq:stand-one}--\eqref{eq:stand-five} for $k=1$.

\medskip
\noindent
$\bullet$ Proof of~\eqref{eq:stand-one} for $k=1$.

We verify the formula for $\SL(\delta+\alpha_{i\to j})$ with $2<i\leq j$ by induction on $\hgt(\alpha_{i\to j})$.
(1) The base of induction is $i=j$. The possible decompositions of $\delta+\alpha_i$ into the (unordered) sum of
two positive roots are as follows:
\begin{equation}\label{eq:split-stand-one-one}
  \delta+\alpha_i=(\delta)+(\alpha_i) \,,\qquad
  \delta+\alpha_i=\alpha_{i\to \jmath} + \alpha_{\overline{\jmath+1} \to i} \quad (\jmath \ne i,\overline{i-1}).
\end{equation}
Using already verified formulas~\eqref{eq:stand-seven}--\eqref{eq:stand-nine}, we get the following list of
concatenated words featuring in the right-hand side of~\eqref{eq:generalized Leclerc} for $\alpha=\delta+\alpha_i$:
\begin{equation}\label{eq:list-stand-one-one}
\begin{split}
  & 10n \ldots i\, 23 \ldots \overline{i-1}\, i \,,\quad 10n \ldots \overline{i+1}\, 23\ldots i\,i ,\\
  & 10n \ldots \overline{\jmath+1}\, 23 \ldots i\,i(i+1) \ldots \jmath \quad \mathrm{for}\ i<\jmath \leq n,\\
  & 10n \ldots i\, 23 \ldots \jmath\, (\jmath+1) \ldots i \quad \mathrm{for}\ 1\leq \jmath < \overline{i-1},\\
  & 12\ldots i\,i\,\overline{i+1}\ldots 0.
\end{split}
\end{equation}
Here, the two words in the first line correspond to the fact that $[\sb[\SL_r(\delta)],\sb[i]]\ne 0$ only for
$\overline{r+2}=i,i-1$, due to~\eqref{eq:delta-brack-standard}, while the last three lines just correspond to the cases $i<\jmath\leq n$,
$1\leq \jmath<\overline{i-1}$, and $\jmath=0$ in~\eqref{eq:split-stand-one-one}. Clearly, $10n \ldots i\, 23 \ldots \overline{i-1}\,i$
is the lexicographically largest word from the list~\eqref{eq:list-stand-one-one}. Therefore, $\SL(\delta+\alpha_i)$ is
indeed given by~\eqref{eq:stand-one} as the corresponding standard bracketing does not vanish:
\begin{equation*}
  \sb[\SL(\delta+\alpha_i)] = \sb[10n \ldots i\, 23 \ldots \overline{i-1}\, i] =
  \begin{cases}
    (-1)^{n-i}E_{i,i+1}t & \ \mathrm{if}\ 2<i\leq n \\
    -E_{n+1,1}t^2   & \ \mathrm{if}\ i=0
  \end{cases}.
\end{equation*}
(2) Let us now prove the induction step: compute $\SL(\delta+\alpha_{i\to j})$ for $\hgt(\alpha_{i\to j})=p+1$
using the formulas for $\SL(\delta+\alpha_{\iota\to \jmath})$ with $\hgt(\alpha_{\iota\to \jmath})\leq p$. The
possible decompositions of $\delta+\alpha_{i\to j}$ into the (unordered) sum of two positive roots are as follows:
\begin{equation}\label{eq:split-stand-one-two}
\begin{split}
  & \delta+\alpha_{i\to j}=(\delta) + (\alpha_{i\to j}) \\
  & \delta+\alpha_{i\to j}=(\delta+\alpha_{i \to \jmath}) + (\alpha_{\overline{\jmath+1}\to j})
    \quad \mathrm{for}\ \jmath\in [i\to j) \\
  & \delta+\alpha_{i\to j}=(\alpha_{i \to \jmath}) + (\delta+\alpha_{\overline{\jmath+1}\to j})
    \quad \mathrm{for}\ \jmath\in [i\to j) \\
  & \delta+\alpha_{i\to j} = (\alpha_{i\to \jmath}) + (\alpha_{\overline{\jmath+1}\to j})
    \quad \mathrm{for}\ \jmath\in \big[\overline{j+1}\to (i-1)\big)
\end{split}
\end{equation}
The corresponding list of concatenations is as follows:
\begin{equation}\label{eq:list-stand-one-two}
\begin{split}
  & 10n \ldots i\,23 \ldots (i-1)i\ldots j \,, \quad
    10n \ldots \overline{j+1}\, 23\ldots j\, i\ldots j ,\\
  & 10n \ldots i\,23 \ldots (i-1)i\ldots \jmath\, \overline{\jmath+1} \ldots j
    \quad \mathrm{for}\ \jmath\in [i\to j),\\
  & 10n \ldots \overline{\jmath+1}\, 23 \ldots \ldots j\, i\, \overline{i+1} \ldots \jmath
    \quad \mathrm{for}\ \jmath\in [i\to j),\\
  & 10n \ldots \overline{\jmath+1}\, 23 \ldots j\, i(i+1) \ldots j\, \overline{j+1} \ldots \jmath
    \quad \mathrm{for}\ j<\jmath \leq n,\\
  & 10n \ldots i\, 23\ldots \jmath\, \overline{\jmath+1} \ldots j
    \quad \mathrm{for}\ 1\leq\jmath <i-1,\\
  & 123 \ldots j\, i\, \overline{i+1}\ldots 0.
\end{split}
\end{equation}
The two words in the first line correspond to the fact that $[\sb[\SL_r(\delta)],\sb[\SL(\alpha_{i\to j})]]\ne 0$
only when $\overline{r+2}=i,\overline{j+1}$, while the words from the last three lines correspond to the cases
$j<\jmath\leq n$, $1\leq \jmath<i-1$, and $\jmath=0$ in the last decomposition
of~\eqref{eq:split-stand-one-two}. Clearly, $10n \ldots i\, 23 \ldots j$ is the lexicographically largest
word from the list~\eqref{eq:list-stand-one-two}. Therefore, $\SL(\delta+\alpha_{i\to j})$ is indeed given
by~\eqref{eq:stand-one} as the corresponding standard bracketing does not vanish:
\begin{equation*}
  \sb[\SL(\delta+\alpha_{i\to j})]=\sb[10n \ldots i\, 23 \ldots j]=
  \begin{cases}
    (-1)^{n-i}E_{i,j+1}t & \ \mathrm{if}\ 2<i<j\leq n \\
    (-1)^{n-i}E_{i,1}t^2 & \ \mathrm{if}\ 2<i<j=0
  \end{cases}.
\end{equation*}

\noindent
$\bullet$ Proof of~\eqref{eq:stand-two} for $k=1$.

The possible decompositions of $\delta+\alpha_2$ into the (unordered) sum of two positive roots are as follows:
\begin{equation}\label{eq:split-stand-two}
  \delta+\alpha_2=(\delta)+(\alpha_2) \,,\qquad
  \delta+\alpha_2=\alpha_{2\to \jmath} + \alpha_{\overline{\jmath+1} \to 2} \quad (\jmath \ne 1,2).
\end{equation}
Thus, the concatenated words in the right-hand side of~\eqref{eq:generalized Leclerc} for $\alpha=\delta+\alpha_2$ are:
\begin{equation}\label{eq:list-stand-two}
\begin{split}
  & 10n \ldots \overline{r+2}\, 23 \ldots \overline{r+1}\,2 \quad \mathrm{for}\ 1\leq r\leq n,\\
  & 10n \ldots \overline{\jmath+1}\,223 \ldots \jmath \quad (2<\jmath \leq n) \,,\qquad
    1223\ldots n0.
\end{split}
\end{equation}
Here, the $n$ words in the first line correspond to the fact that $[\sb[\SL_r(\delta)],\sb[2]]\ne 0$ for all
$1\leq r\leq n$, according to~\eqref{eq:delta-brack-standard}. Clearly, $10n \ldots 322$ is the lexicographically
largest word from the list~\eqref{eq:list-stand-two}. Therefore, $\SL(\delta+\alpha_2)$ is indeed given
by~\eqref{eq:stand-two} as the corresponding standard bracketing does not vanish:
\begin{equation*}
  \sb[\SL(\delta+\alpha_2)]=\sb[10n \ldots 322]=[\sb[10n \ldots 32],\sb[2]]=2(-1)^nE_{23}t.
\end{equation*}

\noindent
$\bullet$ Proof of~\eqref{eq:stand-three} for $k=1$.

Let us prove by induction on $j$ that:
\begin{equation}\label{eq:stand-three-k1}
  \SL(\delta+\alpha_{2\to j})=10n \ldots 3234 \ldots j 2 \quad \mathrm{for}\ 2\leq j\leq 0.
\end{equation}
(1) The base of induction is $j=2$, for which the result was just proved above.

\noindent
(2) Let us now prove the induction step: prove~\eqref{eq:stand-three-k1} for $\SL(\delta+\alpha_{2\to j})$
utilizing the same formula for $\SL(\delta+\alpha_{2\to \jmath})$ with $2\leq \jmath<j$. The possible
decompositions of $\delta+\alpha_{2\to j}$ into the (unordered) sum of two positive roots are as follows:
\begin{equation}\label{eq:split-stand-three}
\begin{split}
  & \delta+\alpha_{2\to j}=(\delta) + (\alpha_{2\to j})\\
  & \delta+\alpha_{2\to j}=(\delta+\alpha_{2\to \jmath}) + (\alpha_{\overline{\jmath+1} \to j})
      \quad \mathrm{for}\ \jmath\in [2\to j) \\
  & \delta+\alpha_{2\to j}=(\delta+\alpha_{\overline{\jmath+1} \to j})+(\alpha_{2\to \jmath})
      \quad \mathrm{for}\ \jmath\in [2\to j) \\
  & \delta+\alpha_{2\to j}=(\alpha_{2 \to \jmath}) + (\alpha_{\overline{\jmath+1}\to j})
      \quad \mathrm{for}\ \jmath\in \big[\overline{j+1}\to 1\big)
\end{split}
\end{equation}
Thus, the concatenated words in the right-hand side of~\eqref{eq:generalized Leclerc} for
$\alpha=\delta+\alpha_{2\to j}$ are:
\begin{equation}\label{eq:list-stand-three}
\begin{split}
  & 10n \ldots \overline{r+2}\, 23 \ldots \overline{r+1}\,23\ldots j \,,\quad \mathrm{for}\ 1\leq r\leq n,\\
  & 10n \ldots 3234 \ldots \jmath\, 2\, \overline{\jmath+1} \ldots j \quad \mathrm{for}\ \jmath\in [2\to j) ,\\
  & 10n \ldots \overline{\jmath+1}\, 23 \ldots j\, 23 \ldots \jmath \quad \mathrm{for}\ \jmath\in [2\to j) ,\\
  & 10n \ldots \overline{\jmath+1}\, 23 \ldots j\, 23 \ldots \jmath \quad (j<\jmath\leq n) \,,
    \qquad  12 \ldots j\, 23 \ldots n0.
\end{split}
\end{equation}
The $n$ words in the first line correspond to the fact that $[\sb[\SL_r(\delta)],\sb[\SL(\alpha_{2\to j})]]\ne 0$
for all $1\leq r\leq n$, according to~\eqref{eq:delta-brack-standard}. Clearly, $10n \ldots 3234 \ldots j2$ is the
lexicographically largest word from the list~\eqref{eq:list-stand-three}. Therefore, $\SL(\delta+\alpha_{2\to j})$
is indeed given by~\eqref{eq:stand-three-k1} as the corresponding standard bracketing does not vanish:
\begin{equation*}
  \sb[\SL(\delta+\alpha_{2\to j})] = \sb[10n \ldots 3234 \ldots j 2] =
  \begin{cases}
    (-1)^{n}E_{2,j+1}t & \ \mathrm{if}\ 2<j\leq n \\
    (-1)^{n}E_{21}t^2 & \ \mathrm{if}\ j=0
  \end{cases}.
\end{equation*}

\noindent
$\bullet$ Proof of~\eqref{eq:stand-four} for $k=1$.

Let us prove by induction on $i$ that:
\begin{equation}\label{eq:stand-four-k1}
  \SL(\delta+\alpha_{1\to i})=123 \ldots n\, 1023\ldots i \quad \mathrm{for}\ 1\leq i\leq n.
\end{equation}
(1) The base of induction is $i=1$. The possible decompositions of $\delta+\alpha_1$ into the (unordered)
sum of two positive roots are as follows:
\begin{equation}\label{eq:split-stand-four-one}
  \delta+\alpha_{1}=(\delta) + (\alpha_1) \,, \qquad
  \delta+\alpha_{1}=(\alpha_{1 \to \jmath}) + (\alpha_{\overline{\jmath+1}\to 1}) \quad (\jmath\ne 0,1).
\end{equation}
Thus, the concatenated words in the right-hand side of~\eqref{eq:generalized Leclerc} for $\alpha=\delta+\alpha_1$ are:
\begin{equation}\label{eq:list-stand-four-one}
\begin{split}
  & 1\, 10n \ldots \overline{r+2}\, 23 \ldots \overline{r+1} \quad \mathrm{for}\ 1\leq r\leq n , \\
  & 123 \ldots \jmath\, 10n \ldots (\jmath+1) \quad (1<\jmath<n) \,,\qquad
    123 \ldots n\, 10.
\end{split}
\end{equation}
Here, the $n$ words in the first line correspond to the fact that $[\sb[\SL_r(\delta)],\sb[1]]\ne 0$ for
all $1\leq r\leq n$, according to~\eqref{eq:delta-brack-standard}. Clearly, $123 \ldots n\, 10$ is the
lexicographically largest word from the list~\eqref{eq:list-stand-four-one}. Therefore, $\SL(\delta+\alpha_{1})$
is indeed given by~\eqref{eq:stand-four-k1} as the corresponding standard bracketing does not vanish:
\begin{equation*}
  \sb[\SL(\delta+\alpha_1)]=\sb[123 \ldots n\, 10]=[\sb[123 \ldots n],\sb[10]]=-E_{12}t.
\end{equation*}
(2) Let us now prove the induction step: prove~\eqref{eq:stand-four-k1} for $\SL(\delta+\alpha_{1\to i})$
utilizing the same formula for $\SL(\delta+\alpha_{1\to \iota})$ with $1\leq \iota<i$. The possible decompositions
of $\delta+\alpha_{1\to i}$ into the (unordered) sum of two positive roots are as follows:
\begin{equation}\label{eq:split-stand-four-two}
\begin{split}
  & \delta+\alpha_{1\to i}=(\delta) + (\alpha_{1 \to i}) \\
  & \delta+\alpha_{1\to i}=(\delta+\alpha_{1 \to \iota}) + (\alpha_{(\iota+1) \to i})
    \quad \mathrm{for}\ \iota\in [1\to i) \\
  & \delta+\alpha_{1\to i}=(\delta+\alpha_{(\iota+1) \to i})+(\alpha_{1 \to \iota})
    \quad \mathrm{for}\ \iota\in [1\to i) \\
  & \delta+\alpha_{1\to i}=(\alpha_{1\to \iota}) + (\alpha_{\overline{\iota+1}\to i})
    \quad \mathrm{for}\ \iota\in \big[\overline{i+1}\to 0\big)
\end{split}
\end{equation}
Thus, the concatenated words in the right-hand side of~\eqref{eq:generalized Leclerc} for
$\alpha=\delta+\alpha_{1\to i}$ are:
\begin{equation}\label{eq:list-stand-four-two}
\begin{split}
  & 123 \ldots i\, 10n \ldots \overline{i+1}\, 23 \ldots i \,,\qquad
    123 \ldots i\, 123 \ldots n0,\\
  & 123 \ldots n\, 1023 \ldots \iota\, (\iota+1) \ldots i \quad \mathrm{for}\ 1\leq \iota<i ,\\
  & 123 \ldots \iota\, 10n \ldots (\iota+1)\, 23 \ldots i \quad \mathrm{for}\ 1<\iota<i \,,\qquad
    1\, 10n \ldots 3234 \ldots i2,\\
  & 123 \ldots \iota\, 10n \ldots \overline{\iota+1}\, 23 \ldots i \quad \mathrm{for}\ i<\iota\leq n.
\end{split}
\end{equation}
The two words in the first line correspond to the fact that $[\sb[\SL_r(\delta)],\sb[\SL(\alpha_{1\to i})]]\ne 0$
only when $r=i-1,n$ (for $1<i\leq n$), while the words in the third line correspond to the cases $1<\iota<i$ and
$\iota=1$ in the third line of~\eqref{eq:split-stand-four-two}. Clearly, $123 \ldots n\, 1023 \ldots i$ is
the lexicographically largest word from the list~\eqref{eq:list-stand-four-two}. Therefore,
$\SL(\delta+\alpha_{1\to i})$ is indeed given by~\eqref{eq:stand-four-k1} as the corresponding
standard bracketing does not vanish:
\begin{equation*}
  \sb[\SL(\delta+\alpha_{1\to i})]=\sb[123 \ldots n\, 1023 \ldots i]=
  [\sb[123 \ldots n],\sb[1023 \ldots i]]=-E_{1,i+1}t.
\end{equation*}

\noindent
$\bullet$ Proof of~\eqref{eq:stand-five} for $k=1$.

Let us prove by induction on $\hgt(\alpha_{j\to i})$ that
\begin{equation}\label{eq:stand-five-k1}
  \SL(\delta+\alpha_{j \to i})=10n \ldots j\, 23 \ldots \overline{j-2}\, 10n \ldots \overline{j-1}\, 23\ldots i
  \quad \mathrm{for}\ i<\overline{i+1}<j.
\end{equation}
(1) The base of induction is $(j,i)=(0,1)$. The possible decompositions of $\delta+\alpha_{0\to 1}$
into the (unordered) sum of two positive roots are as follows:
\begin{equation}\label{eq:split-stand-five-one}
\begin{split}
  & \delta+\alpha_{0\to 1}=(\delta) + (\alpha_{0\to 1}) ,\\
  & \delta+\alpha_{0\to 1}=(\delta+\alpha_0) + (\alpha_{1}) \,,\qquad
    \delta+\alpha_{0\to 1}=(\delta+\alpha_1) + (\alpha_0) ,\\
  & \delta+\alpha_{0\to 1}=(\alpha_{0\to \iota})+(\alpha_{(\iota+1)\to 1})
    \quad \mathrm{for}\ 1< \iota< n.
\end{split}
\end{equation}
Thus, the concatenated words in the right-hand side of~\eqref{eq:generalized Leclerc} for
$\alpha=\delta+\alpha_{0\to 1}$ are:
\begin{equation}\label{eq:list-stand-five-one}
\begin{split}
   & 1010 \ldots \overline{r+2}\, 23 \ldots (r+1) \quad \mathrm{for}\ 1<r\leq n-1 \,,\qquad
     123 \ldots n010,\\
   & 11023 \ldots n0 \,,\qquad 123 \ldots n100 ,\\
   & 1023 \ldots \iota\,10n \ldots (\iota+1) \quad \mathrm{for}\ 1<\iota<n.
\end{split}
\end{equation}
Here, the $n$ words in the first line correspond to the fact that $[\sb[\SL_r(\delta)],\sb[10]]\ne 0$
for all $1\leq r\leq n$, according to~\eqref{eq:delta-brack-standard}. Clearly, $1023 \ldots (n-1)10n$ is the
lexicographically largest word from the list~\eqref{eq:list-stand-five-one}. Therefore, $\SL(\delta+\alpha_{0\to 1})$
is indeed given by~\eqref{eq:stand-five-k1} as the corresponding standard bracketing does not vanish:
\begin{equation*}
  \sb[\SL(\delta+\alpha_{0\to 1})]=
  \sb[1023 \ldots (n-1)10n]=[\sb[1023 \ldots (n-1)],\sb[10n]] = -E_{n+1,2}t^2.
\end{equation*}
(2) Let us now prove the induction step: prove~\eqref{eq:stand-five-k1} for $\SL(\delta+\alpha_{j\to i})$ utilizing
the same formula for $\SL(\delta+\alpha_{\jmath\to \iota})$ with $[\jmath \to \iota)\subsetneq [j \to i)$. The
possible decompositions of $\delta+\alpha_{j\to i}$ into the (unordered) sum of two positive roots are as follows:
\begin{equation}\label{eq:split-stand-five-two1}
\begin{split}
  & \delta+\alpha_{j\to i}=(\delta+\alpha_{j\to \jmath}) + (\alpha_{\overline{\jmath+1}\to i})
     \quad \mathrm{for}\ \jmath\in [j\to i)\\
  & \delta+\alpha_{j\to i}=(\alpha_{j\to \jmath}) + (\delta+\alpha_{\overline{\jmath+1}\to i})
     \quad \mathrm{for}\ \jmath\in [j\to i)\\
  & \delta+\alpha_{j\to i}=(\alpha_{j\to \jmath}) + (\alpha_{\overline{\jmath+1}\to i})
     \quad \mathrm{for}\ \jmath\in [(i+1)\to \overline{j-1})
\end{split}
\end{equation}
as well as
\begin{equation}\label{eq:split-stand-five-two2}
  \delta+\alpha_{j\to i}=(\delta) + (\alpha_{j\to i}).
\end{equation}
The concatenated words in the right-hand side of~\eqref{eq:generalized Leclerc} for
$\alpha=\delta+\alpha_{j\to i}$ arising through~\eqref{eq:split-stand-five-two1} are:
\begin{equation}\label{eq:list-stand-five-two1}
\begin{split}
  & 10n \ldots \overline{\jmath+1}\, 23 \ldots i\, 10n \ldots j\, 23 \ldots \jmath
     \quad \mathrm{for}\ j\leq \jmath\leq 0 ,\\
  & 10n \ldots j\, 23 \ldots \overline{j-2}\, 10n \ldots \overline{j-1}\, 23 \ldots \jmath\, (\jmath+1) \ldots i
     \quad \mathrm{for}\ 1\leq \jmath<i ,\\
  & 10n \ldots \overline{\jmath+1}\, 23 \ldots (\jmath-1)10n \ldots \jmath\, 23 \ldots i\, j(j+1)\ldots \jmath
     \quad \mathrm{for}\ j\leq \jmath\leq n ,\\
  & 123 \ldots n\, 1023 \ldots i\, j(j+1) \ldots n\, 0,\\
  & 10n \ldots j\, 10n \ldots 32\, 34 \ldots i\, 2, \\
  & 10n \ldots j\, 23 \ldots \jmath\, 10n \ldots (\jmath+1)\, 23 \ldots i
    \quad \mathrm{for}\ 2\leq \jmath<i,\\
  & 10n \ldots j\, 23 \ldots \jmath\, 10n \ldots (\jmath+1)\, 23 \ldots i
    \quad \mathrm{for}\ \jmath\in [(i+1) \to \overline{j-1}),
\end{split}
\end{equation}
where the words in the first two lines of~\eqref{eq:list-stand-five-two1} correspond to the first line
of~\eqref{eq:split-stand-five-two1}, depending on whether $\jmath\geq j$ or $\jmath<i$, while the words
in the third--sixth lines of~\eqref{eq:list-stand-five-two1} correspond to the second line
of~\eqref{eq:split-stand-five-two1}, depending on whether $j\leq \jmath<0$, $\jmath=0$, $\jmath=1$, or
$1<\jmath<i$. Meanwhile, the concatenated words in the right-hand side of~\eqref{eq:generalized Leclerc}
for $\alpha=\delta+\alpha_{j\to i}$ arising through the decomposition~\eqref{eq:split-stand-five-two2}
depend on whether $i=1$ or $i\ne 1$:
\begin{equation}\label{eq:list-stand-five-two2}
  10n \ldots j\, 23 \ldots i\, 10n \ldots j\, 23 \ldots \overline{j-1} \,,\
  10n \ldots j\, 23 \ldots i\, 10n \ldots (i+1) 23 \ldots i
\end{equation}
if $i\ne 1$, and
\begin{equation}\label{eq:list-stand-five-two3}
\begin{split}
  & 10n \ldots \overline{r+2}\, 23 \ldots \overline{r+1}\, 10n \ldots j \quad \mathrm{for}\ j-2<r\leq n,\\
  & 10n \ldots j\, 10n \ldots (r+2) 23 \ldots (r+1) \,  \quad \mathrm{for}\ 1\leq r\leq j-2
\end{split}
\end{equation}
if $i=1$. It is easy to see that
  $10n \ldots j\, 23 \ldots \overline{j-2}\, 10n \ldots \overline{j-1}\, 23\ldots i$
is the lexicographically largest word from the above
lists~\eqref{eq:list-stand-five-two1}--\eqref{eq:list-stand-five-two3}. Thus, $\SL(\delta+\alpha_{j\to i})$
is indeed given by~\eqref{eq:stand-five-k1} as the corresponding standard bracketing does not vanish:
\begin{equation*}
  \sb[\SL(\delta+\alpha_{j\to i})]=
  [\sb[10n \ldots j\, 23 \ldots \overline{j-2}],\sb[10n \ldots \overline{j-1}\, 23\ldots i]]=-E_{j,i+1}t^2.
\end{equation*}

\noindent
\underline{Step of Induction}

Let us now prove the step of induction, proceeding by the height of a root.
We shall thus verify the stated formulas for $\aslaws$ $\SL_*(\alpha)$ with
\begin{equation}\label{eq:induction-step-block}
  (d+1)h\leq \hgt(\alpha)<(d+2)h \,,\qquad \mathrm{where}\ h=n+1=\hgt(\delta),
\end{equation}
assuming the validity of the stated formulas for all $\SL_*(\beta)$ with $\hgt(\beta)<\hgt(\alpha)$. In other words,
we verify~\eqref{eq:stand-six} for $k=d$ and formulas~\eqref{eq:stand-one}--\eqref{eq:stand-five} for $k=d+1$.

When evaluating the standard bracketings $\sb[\cdots]$ below, we will only need their values up to
nonzero scalar factors. To this end, we shall use the following notation:
\begin{equation}\label{eq:doteq}
  A\doteq B \quad \text{if} \quad  A=c\cdot B \quad \text{for some} \ c\in \BC\backslash\{0\}.
\end{equation}

\noindent
$\bullet$ Proof of~\eqref{eq:stand-six} for $k=d$.

The possible decompositions of $(d+1)\delta$ into the (unordered) sum of two positive real roots are as follows:
\begin{align}
  & (d+1)\delta=(a\delta+\alpha_{1}) + ((d-a)\delta+\alpha_{2\to 0} ),
  \label{eq:split-imaginary-step-zero}  \\
  & (d+1)\delta=(a\delta+\alpha_{1\to j}) + ((d-a)\delta+\alpha_{\overline{j+1} \to 0} )
    \quad \mathrm{for}\ 2\leq j\leq n,
  \label{eq:split-imaginary-step-one}  \\
  & (d+1)\delta=(a\delta+\alpha_{2\to j}) + ((d-a)\delta+\alpha_{\overline{j+1}\to 1} )
    \quad \mathrm{for}\ 2\leq j\leq n,
  \label{eq:split-imaginary-step-two}  \\
  & (d+1)\delta=(a\delta+\alpha_{i\to j}) + ((d-a)\delta+\alpha_{\overline{j+1}\to (i-1)})
    \quad \mathrm{for}\ 2<i\leq j\leq n
  \label{eq:split-imaginary-step-three},
\end{align}
with $0\leq a\leq d$. By the induction hypothesis, we get the following concatenations:
\begin{equation}\label{eq:nasty-0}
  \footnotesize{
  \ell_{0}^{(a)} =
    \begin{cases}
      1 \underbrace{10n \ldots 32}_{\frac{d}{2} \, \rtim} 2 \underbrace{10n \ldots 32}_{\frac{d}{2} \, \rtim} 3 4 \ldots n 0
        & \mathrm{if}\ a=0 , 2\mid d\\
      1 \underbrace{10n \ldots 32}_{\frac{d+1}{2} \, \rtim} 3 4 \ldots n 0 \underbrace{10n \ldots 32}_{\frac{d-1}{2} \, \rtim} 2
        & \mathrm{if}\ a=0 , 2\nmid d\\
      12 \ldots n \underbrace{1023 \ldots n}_{(a-1) \, \rtim} 10
      \underbrace{10n \ldots 32}_{\frac{d-a}{2} \, \rtim} 2 \underbrace{10n \ldots 32}_{\frac{d-a}{2} \, \rtim} 3 4 \ldots n 0
        & \mathrm{if}\ 0<a<d , 2\mid (d-a)\\
      12 \ldots n \underbrace{1023 \ldots n}_{(a-1) \, \rtim} 10
      \underbrace{10n \ldots 32}_{\frac{d-a+1}{2} \, \rtim} 3 4 \ldots n 0 \underbrace{10n \ldots 32}_{\frac{d-a-1}{2} \, \rtim} 2
        & \mathrm{if}\ 0<a<d , 2\nmid (d-a)\\
      12 \ldots n \underbrace{1023 \ldots n}_{d \, \rtim} 0
        & \mathrm{if}\ a=d
    \end{cases}}
\end{equation}
for the decompositions~\eqref{eq:split-imaginary-step-zero},
\begin{equation}\label{eq:nasty-1}
  \footnotesize{
  \ell_{1j}^{(a)} =
    \begin{cases}
      123 \ldots n \underbrace{1023 \ldots n}_{(a-1) \, \rtim}
      1023 \ldots j \underbrace{10n \ldots \overline{j+1}\, 23 \ldots j}_{(d-a) \, \rtim} \overline{j+1} \ldots 0
        & \mathrm{if}\ 1\leq a\leq d\\
      123 \ldots j\, \underbrace{10n \ldots \overline{j+1}\, 23 \ldots j}_{d \, \rtim} \overline{j+1} \ldots 0
        & \mathrm{if}\ a=0
    \end{cases}}
\end{equation}
for the decompositions~\eqref{eq:split-imaginary-step-one} with $2\leq j\leq n$,
\begin{equation}\label{eq:nasty-2}
    \footnotesize{
    \ell_{2j}^{(a)} =
    \begin{cases}
      10n \ldots \overline{j+1}\, 23 \ldots (j-1) \underbrace{10n \ldots j\, 23 \ldots (j-1)}_{d \, \rtim} j
        & \mathrm{if}\ a=0\\
      10n \ldots \overline{j+1}\, 23 \ldots (j-1) \underbrace{10n \ldots j\, 23 \ldots (j-1)}_{(d-a-1) \, \rtim}
      10n \ldots j \\  \qquad \qquad \qquad
      \underbrace{10n \ldots 32}_{\frac{a}{2} \, \rtim} 2\underbrace{10n \ldots 32}_{\frac{a}{2} \, \rtim} 34 \ldots j
        & \mathrm{if}\ 0<a<d, 2\mid a\\
      10n \ldots \overline{j+1}\, 23 \ldots (j-1) \underbrace{10n \ldots j\, 23 \ldots (j-1)}_{(d-a-1) \, \rtim}
      10n \ldots j \\  \qquad \qquad \qquad
      \underbrace{10n \ldots 32}_{\frac{a+1}{2} \, \rtim} 34 \ldots j\underbrace{10n \ldots 32}_{\frac{a-1}{2} \, \rtim}2
        & \mathrm{if}\ 0<a<d, 2\nmid a\\
      10n \ldots \overline{j+1}
      \underbrace{10n \ldots 32}_{\frac{d}{2} \, \rtim} 2\underbrace{10n \ldots 32}_{\frac{d}{2} \, \rtim} 34 \ldots j
        & \mathrm{if}\ a=d-\mathrm{even}\\
      10n \ldots \overline{j+1}
      \underbrace{10n \ldots 32}_{\frac{d+1}{2} \, \rtim} 34 \ldots j\underbrace{10n \ldots 32}_{\frac{d-1}{2} \, \rtim}2
        & \mathrm{if}\ a=d-\mathrm{odd}\\
    \end{cases}}
\end{equation}
for the decompositions~\eqref{eq:split-imaginary-step-two} with $2\leq j\leq n$, and
\begin{equation}\label{eq:nasty-3}
    \footnotesize{
    \ell_{3ji}^{(a)}=
    \begin{cases}
      10n \ldots \overline{j+1}\, 23 \ldots (j-1) \underbrace{10n \ldots j\, 23 \ldots (j-1)}_{d \, \rtim} j
         & \mathrm{if}\ a=0\\
      10n \ldots \overline{j+1}\, 23 \ldots (j-1) \underbrace{10n \ldots j\, 23 \ldots (j-1)}_{(d-a-1) \, \rtim}
      10n \ldots j\, \\  \qquad \qquad \qquad
      23 \ldots (i-1) \underbrace{10n \ldots i\, 23 \ldots (i-1)}_{a \, \rtim} i(i+1) \ldots j
         & \mathrm{if}\ 0<a<d\\
      10n \ldots \overline{j+1}\, 23 \ldots (i-1)
      \underbrace{10n \ldots i\, 23 \ldots (i-1)}_{d \, \rtim} i(i+1) \ldots j
         & \mathrm{if}\ a=d
    \end{cases}}
\end{equation}
for the decompositions~\eqref{eq:split-imaginary-step-three} with $2<i\leq j\leq n$.

Clearly, the lexicographically largest word from the lists~\eqref{eq:nasty-0}--\eqref{eq:nasty-3} is
  \[ \ell_{22}^{(0)}=10n \ldots 3\underbrace{10n \ldots 2}_{d \, \rtim} 2, \]
which coincides with the word in the right-hand side of~\eqref{eq:stand-six} for $k=d$ and $r=1$.
Let us compute its standard bracketing:
\begin{equation*}
  \sb[\ell_{22}^{(0)}]=
  [\sb[10n \ldots 3],\sb[\underbrace{10n \ldots 2}_{d \, \rtim}2]]\doteq
  [E_{32}t, E_{23}t^d]\doteq (E_{22}-E_{33})t^{d+1},
\end{equation*}
where we use the induction hypothesis in the second equality.
Moreover, a similar argument also implies that
\begin{equation}\label{eq:lin.dep.brack}
  \sb[\ell_{22}^{(a)}]\doteq (E_{22}-E_{33})t^{d+1}\doteq \sb[\ell_{22}^{(0)}] \qquad \forall\, 0<a\leq d.
\end{equation}

The next lexicographically largest word from the lists~\eqref{eq:nasty-0}--\eqref{eq:nasty-3},
with the words $\{\ell_{22}^{(a)}\}_{a=0}^d$ excluded due to~\eqref{eq:lin.dep.brack}, is
  \[ \ell_{23}^{(0)}=\ell^{(0)}_{333}=10n \ldots 42\underbrace{10n \ldots 32}_{d \, \rtim} 3, \]
which coincides with the word in the right-hand side of~\eqref{eq:stand-six} for $k=d$ and $r=2$.
Let us compute its standard bracketing:
\begin{equation*}
  \sb[\ell_{23}^{(0)}]=
  [\sb[10n \ldots 42],\sb[\underbrace{10n \ldots 32}_{d \, \rtim}3]]\doteq
  [E_{43}t, E_{34}t^d]\doteq (E_{33}-E_{44})t^{d+1},
\end{equation*}
where we use the induction hypothesis in the second equality. Moreover, a similar argument also applies
to the remaining words $\ell_{23}^{(a)}$ and $\ell_{333}^{(a)}$ with $0<a\leq d$:
\begin{equation*}
  \sb[\ell_{23}^{(a)}], \sb[\ell_{333}^{(a)}] \in
  \mathrm{span} \big\{ (E_{22}-E_{33})t^{d+1},(E_{33}-E_{44})t^{d+1} \big\} =
  \mathrm{span} \big\{ \sb[\ell_{22}^{(0)}],\sb[\ell_{23}^{(0)}] \big\} .
\end{equation*}
Proceeding further with the same line of reasoning we find that the $(n-1)$ lexicographically largest
words from the above lists with linearly independent standard bracketings are:
$\ell_{22}^{(0)}, \ell_{23}^{(0)},\ldots, \ell_{2n}^{(0)}$.
This proves~\eqref{eq:stand-six} for $k=d$ and $1\leq r\leq n-1$.

The lexicographically largest word among the remaining lists~\eqref{eq:nasty-0}--\eqref{eq:nasty-1} is
  \[ \ell_{1n}^{(0)}=\ell_0^{(d)}=123 \ldots n\underbrace{1023 \ldots n}_{d \, \rtim}0. \]
Let us evaluate its standard bracketing:
\begin{equation*}
  \sb[\ell_{1n}^{(0)}]=
  [\sb[123 \ldots n],\sb[\underbrace{1023 \ldots n}_{d \, \rtim}0]]\doteq
  [E_{1,n+1},E_{n+1,1}t^{d+1}]=(E_{11}-E_{n+1,n+1})t^{d+1}.
\end{equation*}
As this expression is linear independent with $\{\sb[\ell^{(0)}_{2j}]\}_{j=2}^n$ computed above, we get
$\SL_n((d+1)\delta)=\ell_{1n}^{(0)}$. This completes our proof of~\eqref{eq:stand-six} for $k=d$,
and proves:
\begin{equation*}
  \sb[\SL_r((d+1)\delta)]\doteq
  \begin{cases}
    (E_{r+1,r+1}-E_{r+2,r+2})t^{d+1} & \mathrm{if}\ 1\leq r\leq n-1\\
    (E_{11}-E_{n+1,n+1})t^{d+1} & \mathrm{if}\ r=n\\
  \end{cases} \ .
\end{equation*}

\noindent
$\bullet$ Proof of~\eqref{eq:stand-one}--\eqref{eq:stand-five} for $k=d+1$.

The case of real roots is treated precisely as in our part II of the induction base. Let us
present the proof of~\eqref{eq:stand-three}, leaving the other ones to the interested reader.

Instead of listing all possible decompositions of $(d+1)\delta+\alpha_{2 \to j}$, we start by noting
that the word $\ell(d+1,j)$ from the right-hand side of~\eqref{eq:stand-three} for $k=d+1$ corresponds
to the decomposition
  $(d+1)\delta+\alpha_{2 \to j} =
   (\lfloor \frac{d+1}{2}\rfloor \delta + \alpha_2) + (\lceil \frac{d+1}{2}\rceil \delta + \alpha_{3 \to j})$.
Since $\ell(d+1,j)>10n\ldots 32=\SL_{1}(\delta)$, it suffices to consider in~\eqref{eq:generalized Leclerc}
only those decompositions $(d+1)\delta+\alpha_{2 \to j}=\gamma_1+\gamma_2$ such that each word
$\SL_*(\gamma_1)$, $\SL_*(\gamma_2)$ is either $> 10n \ldots 32$ or is a prefix of $10n \ldots 32$.
By the induction hypothesis, this restricts us to the following list:
\begin{equation}\label{eq:decomp-stand-two-step}
\begin{split}
  & (d+1)\delta+\alpha_{2 \to j}=(\delta)+(d\delta+\alpha_{2 \to j}) ,\\
  & (d+1)\delta+\alpha_{2 \to j}=(a\delta+\alpha_{2})+((d+1-a)\delta+\alpha_{3 \to j}) \,,
    \quad 0\leq a\leq d+1 ,\\
  & (d+1)\delta+\alpha_{2 \to j}=((d+1)\delta+\alpha_{2 \to \jmath})+(\alpha_{(\jmath+1) \to j}) \,,
    \quad 2<\jmath<j  .
\end{split}
\end{equation}
We therefore get the following list of concatenated words:
\begin{equation}\label{eq:list-stand-two-step}
\begin{split}
 & \begin{cases}
     10n\ldots32 \underbrace{10n \ldots 32}_{\frac{d}{2} \, \rtim} 2\underbrace{10n \ldots 32}_{\frac{d}{2} \, \rtim} 34 \ldots j
       & \mathrm{if}\ 2\mid\, d\\
     10n\ldots32\underbrace{10n \ldots 32}_{\frac{d+1}{2} \, \rtim} 34 \ldots j\underbrace{10n \ldots 32}_{\frac{d-1}{2} \, \rtim}2
       & \mathrm{if}\ 2\nmid\, d
  \end{cases} \ , \\
 & \begin{cases}
     \underbrace{10n \ldots 32}_{a \, \rtim}2\underbrace{10n \ldots 32}_{(d+1-a) \, \rtim}34\ldots j
       & \mathrm{if} \ \frac{d+1}{2}\leq a\leq d+1 \\
     \underbrace{10n \ldots 32}_{(d+1-a) \, \rtim}34\ldots j\underbrace{10n \ldots 32}_{a \, \rtim}2
       & \mathrm{if} \ 0\leq a< \frac{d+1}{2}
   \end{cases} \ ,\\
  & \begin{cases}
     \underbrace{10n \ldots 32}_{\frac{d+2}{2} \, \rtim} 34 \ldots \jmath
     \underbrace{10n \ldots 32}_{\frac{d}{2} \, \rtim} 2\, (\jmath+1) \ldots j
       & \mathrm{if}\ 2\mid\,d \\
     \underbrace{10n \ldots 32}_{\frac{d+1}{2} \, \rtim} 2
     \underbrace{10n \ldots 32}_{\frac{d+1}{2} \, \rtim} 34 \ldots \jmath\, (\jmath+1) \ldots j
       & \mathrm{if}\ 2\nmid\,d
   \end{cases} \ .
\end{split}
\end{equation}
It is easy to see that the word $\ell(d+1,j)$ is the lexicographically largest word from
the list~\eqref{eq:list-stand-two-step}. Let us evaluate its standard bracketing:
\begin{equation*}
\begin{split}
  & \sb[\ell(d+1,j)] \doteq
    \big[\sb[\underbrace{10n \ldots 32}_{\lfloor \frac{d+1}{2}\rfloor \, \rtim}2],
         \sb[\underbrace{10n \ldots 32}_{\lceil \frac{d+1}{2}\rceil \, \rtim}34\ldots j]\big]\doteq\\
  & \begin{cases}
      [E_{23}t^{\lfloor \frac{d+1}{2}\rfloor}, E_{3,j+1}t^{\lceil \frac{d+1}{2}\rceil}]
         & \mathrm{if}\ 2<j\leq n\\
      [E_{23}t^{\lfloor \frac{d+1}{2}\rfloor}, E_{31}t^{\lceil \frac{d+3}{2}\rceil}]
         & \mathrm{if}\ j=0
    \end{cases}
    \ \doteq \
    \begin{cases}
      E_{2,j+1}t^{d+1} & \mathrm{if}\ 2<j\leq n\\
      E_{21}t^{d+2}   & \mathrm{if}\ j=0
    \end{cases},
\end{split}
\end{equation*}
where we use the induction hypothesis for
  $\sb[\SL(\lfloor \frac{d+1}{2}\rfloor \delta + \alpha_2)],
   \sb[\SL(\lceil \frac{d+1}{2}\rceil \delta + \alpha_{3 \to j})]$.

This completes our proof of~\eqref{eq:stand-three} for $k=d+1$.
\end{proof}


\subsection{General order}\label{ssec:general}
\

We now compute $\aslaws$ for $\widehat{\fsl}_{n+1}$ with an arbitrary order $<$ on $\wI=\{0,1,\ldots, n\}$.
The key feature is that all $\aslaws$ are determined by those of length $\leq n$. Furthermore, the explicit
description of the degree $\delta$ $\aslaws$ is instrumental for the general pattern.

\medskip
\noindent
\textbf{Notation:} To distinguish from $<$, we shall now use $\prec$ for the standard order on $\wI$:
\begin{equation*}
  1\prec 2\prec 3\prec \dots \prec n\prec 0 \,.
\end{equation*}

We start with the following simple result:

\begin{lemma}\label{Sasha's proposition}
Consider two arches $[a \to \overline{b+1})\subsetneq [a'\to \overline{b'+1})$ such that $b'\ne a'-1$ and
$\min[a'\to \overline{b'+1})\in [a \to \overline{b+1})$. Then: $\SL(\alpha_{a\to b})<\SL(\alpha_{a'\to b'})$.
\end{lemma}

\begin{proof}
We note that this result is a property of the Lalonde-Ram's bijection $\ell$~\eqref{eqn:associated word} for the
simple Lie algebra $\fsl_{\hgt(\alpha_{a'\to b'})+1}$ with simple roots labelled by $[a'\to \overline{b'+1})$.

If $b\ne b'$, consider roots $\gamma_1=\alpha_{a\to b}$ and $\gamma_2=\alpha_{\overline{b+1}\to b'}$ whose sum is
$\alpha=\gamma_1+\gamma_2=\alpha_{a\to b'}$. In view of the remark made above (reduction to a finite case), the convexity
of Proposition~\ref{prop:fin.convex} implies that $\SL(\alpha)$ is ``sandwiched'' between $\SL(\gamma_1)$ and $\SL(\gamma_2)$.
But by our assumption the minimal letter of $\SL(\gamma_1)$ is $\min[a'\to \overline{b'+1})$ which is smaller than
the minimal letter of $\SL(\gamma_2)$. Thus, we get: $\SL(\gamma_1)<\SL(\alpha)<\SL(\gamma_2)$.

By a similar argument, we also conclude that $\SL(\alpha_{a\to b'})<\SL(\alpha_{a'\to b'})$ if $a\ne a'$.
This completes our proof of the desired inequality $\SL(\alpha_{a\to b})<\SL(\alpha_{a'\to b'})$.
\end{proof}

Due to the $D_{n+1}$-symmetry of $\wI$ and $\wDelta^+$, where $D_{n+1}$ denotes a dihedral group,
we can assume, without loss of generality, that
\begin{equation}\label{eq:minimal letters}
  1=\min\big\{a \,\big|\, a\in\wI \,\big\} \quad \mathrm{and} \quad
  i:=\min\big\{a \,\big|\, a\in\wI\setminus \{1\}\big\}\ne 0,
\end{equation}
where $\min$ is taken with respect to our order $<$ on $\wI$.

\begin{lemma}\label{deltaorder}
For $c\in \wI\setminus \{1\}=\{2,\ldots,n,0\}$, define the degree $\delta$ word $\ell_c(\delta)\in \wI^*$~via:
\begin{equation}\label{eq:el-words}
  \ell_c(\delta):=\SL(\alpha_{\overline{c+1} \to \overline{c-1}}) c.
\end{equation}
Then, we have:
\begin{itemize}[leftmargin=0.7cm]

\item[1)]
$\ell_a(\delta)>\ell_b(\delta)$ whenever $i\preceq a \prec b \leq 0$,

\item[2)]
$\ell_a(\delta)<\ell_b(\delta)$ whenever $1 \prec a \prec b \preceq i$,

\end{itemize}
so that $\ell_2(\delta)<\ell_3(\delta)<\dots<\ell_i(\delta)>\ell_{i+1}(\delta)>\dots>\ell_0(\delta)$.
\end{lemma}

We need a simple fact about Lalonde-Ram's bijection~\eqref{eqn:associated word} for a finite type~$A$:

\begin{claim}\label{claim:LR-A-endpoint}
(1) If $b=\min\{a,\overline{a+1},\ldots,\overline{b-1},b\}$, then
$\SL(\alpha_{a \to b}) = b\, \overline{b-1} \dots \overline{a+1}\, a$.

\noindent
(2) If $a=\min\{a,\overline{a+1},\ldots,\overline{b-1},b\}$, then
$\SL(\alpha_{a \to b}) = a\, \overline{a+1} \dots \overline{b-1}\, b$.
\end{claim}

\begin{proof}[Proof of Lemma~\ref{deltaorder}]

The proof is based on the more explicit formulas for $\ell_c(\delta)$:

\noindent
$\bullet$ Case 1: $1\prec c\prec i$.

Consider the costandard factorization $\SL(\alpha_{\overline{c+1}\to \overline{c-1}})=\ell_{1,c}\ell_{2,c}$.
As $\ell_{1,c}$ starts with $1$, $\ell_{2,c}$ starts with $i$, $\deg \ell_{1,c}, \deg\ell_{2,c} \in \wDelta^+$,
and $\deg \ell_{1,c} + \deg\ell_{2,c} = \alpha_{\overline{c+1}\to \overline{c-1}}$, we see that
$\ell_{2,c}=\SL(\alpha_{\overline{c+1}\to e})$ and $\ell_{1,c}=\SL(\alpha_{\overline{e+1}\to \overline{c-1}})$
for some $e\succeq i$. For $e\succ i$, we have
  $\SL(\alpha_{\overline{e+1}\to \overline{c-1}}) < \SL(\alpha_{\overline{i+1}\to \overline{c-1}})$
by Lemma~\ref{Sasha's proposition}. Therefore, we have:
\begin{equation*}
  \SL(\alpha_{\overline{e+1}\to \overline{c-1}})\SL(\alpha_{\overline{c+1}\to e}) <
  \SL(\alpha_{\overline{i+1}\to \overline{c-1}})\SL(\alpha_{\overline{c+1} \to i}) \qquad \forall\, e\succ i \,.
\end{equation*}
Since the word $\SL(\alpha_{\overline{i+1}\to \overline{c-1}})\SL(\alpha_{\overline{c+1} \to i})$ is Lyndon
(as it starts with the smallest letter $1$ which appears only once) and its bracketing is clearly nonzero,
we~conclude:
\begin{equation*}
  \SL(\alpha_{\overline{c+1}\to \overline{c-1}})=
  \SL(\alpha_{\overline{i+1}\to \overline{c-1}}) \SL(\alpha_{\overline{c+1} \to i})=
  \SL(\alpha_{\overline{i+1}\to \overline{c-1}})\, i\, \overline{i-1} \dots \overline{c+1} \,,
\end{equation*}
with the last equality due to Claim~\ref{claim:LR-A-endpoint}. Thus, we obtain:
\begin{equation}\label{eq:el-word-explicit-1}
  \ell_c(\delta)=\SL(\alpha_{\overline{i+1}\to \overline{c-1}})\, i\, \overline{i-1} \dots \overline{c+1}\, c
  \qquad \forall \ 1 \prec c \preceq i \,.
\end{equation}
The desired inequality $\ell_a(\delta)<\ell_b(\delta)$ for $1\prec a\prec b \preceq i$ follows now
from Lemma~\ref{Sasha's proposition}.

\noindent
$\bullet$  Case 2: $i \prec c \preceq 0$.

Arguing as in the previous case, we see that the costandard factorization
$\SL(\alpha_{\overline{c+1}\to \overline{c-1}})=\ell_{1,c}\ell_{2,c}$ has the form
$\ell_{2,c}=\SL(\alpha_{e\to \overline{c-1}})$ and $\ell_{1,c}=\SL(\alpha_{\overline{c+1}\to \overline{e-1}})$
for some $1 \prec e \preceq i$. For $1\prec e \prec i$, we have
  $\SL(\alpha_{\overline{c+1}\to \overline{e-1}}) < \SL(\alpha_{\overline{c+1}\to \overline{i-1}})$
by Lemma~\ref{Sasha's proposition}, and so
  $\SL(\alpha_{\overline{c+1}\to \overline{e-1}}) \SL(\alpha_{e\to \overline{c-1}}) <
   \SL(\alpha_{\overline{c+1}\to \overline{i-1}}) \SL(\alpha_{i\to \overline{c-1}})$.
As the word $\SL(\alpha_{\overline{c+1}\to \overline{i-1}}) \SL(\alpha_{i\to \overline{c-1}})$ is Lyndon
(as it starts with the smallest letter $1$ which appears only once) and clearly has a nonzero bracketing,
we conclude:
\begin{equation*}
  \SL(\alpha_{\overline{c+1}\to \overline{c-1}})=
  \SL(\alpha_{\overline{c+1}\to \overline{i-1}}) \SL(\alpha_{i\to \overline{c-1}}) =
  \SL(\alpha_{\overline{c+1}\to \overline{i-1}})\, i\, \overline{i+1} \dots \overline{c-1}
\end{equation*}
with the last equality due to Claim~\ref{claim:LR-A-endpoint}. Thus, we obtain:
\begin{equation}\label{eq:el-word-explicit-2}
  \ell_c(\delta)=\SL(\alpha_{\overline{c+1}\to \overline{i-1}})\, i\, \overline{i+1} \dots \overline{c-1}\, c
  \qquad \forall \ i \prec c \preceq 0 \,.
\end{equation}
The desired inequality $\ell_a(\delta) > \ell_b(\delta)$ for $i \preceq a \prec b$ follows from
Lemma~\ref{Sasha's proposition} again.
\end{proof}

For $a,b\in \wI$, we introduce $sgn(a-b)\in \{-1,0,1\}$ via:
\begin{equation}\label{eq:sgn-symbol}
  sgn(a-b):=
  \begin{cases}
    1  & \mathrm{if}\ a\succ b \\
    -1 & \mathrm{if}\ a\prec b \\
    0 & \mathrm{if}\ a=b
  \end{cases}.
\end{equation}
The following generalization of Theorem~\ref{thm:sln-standard} is the main result of this section:

\begin{theorem}\label{thm:sln-general}
The $\aslaws$ for $\widehat{\fsl}_{n+1}\ (n\geq 2)$ with any order $<$ on $\wI=\{0,1,\dots,n\}$
satisfying~\eqref{eq:minimal letters} are described by the formulas below ($k\geq 1$):
\begin{align}
  & \Big\{ \SL_1(k\delta), \ldots, \SL_n(k\delta) \Big\} =
    \Big\{ \SL(\alpha_{\overline{c+1}\to \overline{c-1}})\underbrace{\ell_{c+sgn(i-c)}(\delta)}_{(k-1) \, \rtim}c \,\Big|\,
        c\in \wI\setminus \{1\} \Big\} ,
    \label{eq:general-one} \\
  & \SL(k\delta+\alpha_{a\to b}) = \underbrace{\ell_{b+1}(\delta)}_{k \, \rtim}b(b-1)\dots a \,,
    \quad \mathrm{for} \ 1 \prec a \preceq b \prec i ,
    \label{eq:general-three} \\
  & \SL(k\delta+\alpha_{a\to b}) = \underbrace{\ell_{\overline{a-1}}(\delta)}_{k \, \rtim}a\, \overline{a+1} \dots b \,,
    \quad \mathrm{for} \ i \prec a \preceq b \preceq 0 ,
    \label{eq:general-two}
\end{align}
\begin{multline} \label{eq:general-four}
  \SL(k\delta+\alpha_{a\to b}) = \hfill \ \mathrm{for}\ 1 \prec a \prec i \prec b \\
  \begin{cases}
    \underbrace{\ell_i(\delta)}_{\frac{k}{3} \, \rtim} i \underbrace{\ell_i(\delta)}_{\frac{k}{3} \, \rtim}
    \overline{i+1}\dots b \underbrace{\ell_i(\delta)}_{\frac{k}{3} \, \rtim}\overline{i-1}\dots a
      & \mathrm{if}\ 3\mid k \\
    \underbrace{\ell_i(\delta)}_{\frac{k+1}{3} \, \rtim} \overline{i-1}\dots a \underbrace{\ell_i(\delta)}_{\frac{k-2}{3} \, \rtim}
    i \underbrace{\ell_i(\delta)}_{\frac{k+1}{3} \, \rtim} \overline{i+1}\dots b
      & \mathrm{if}\ 3\mid \overline{k+1} \\
    \underbrace{\ell_i(\delta)}_{\frac{k+2}{3} \, \rtim} \overline{i+1}\dots b \underbrace{\ell_i(\delta)}_{\frac{k-1}{3} \, \rtim}
    i \underbrace{\ell_i(\delta)}_{\frac{k-1}{3} \, \rtim}\overline{i-1}\dots a
      & \mathrm{if}\ 3\mid \overline{k+2}
  \end{cases},
  \quad \mathrm{if}\ \overline{i-1}<\overline{i+1}\\
  \begin{cases}
    \underbrace{\ell_i(\delta)}_{\frac{k}{3} \, \rtim} i \underbrace{\ell_i(\delta)}_{\frac{k}{3} \, \rtim}
    \overline{i-1}\dots a \underbrace{\ell_i(\delta)}_{\frac{k}{3} \, \rtim} \overline{i+1}\dots b
      & \mathrm{if}\ 3\mid k  \\
    \underbrace{\ell_i(\delta)}_{\frac{k+1}{3} \, \rtim} \overline{i+1}\dots b \underbrace{\ell_i(\delta)}_{\frac{k-2}{3} \, \rtim}
    i \underbrace{\ell_i(\delta)}_{\frac{k+1}{3} \, \rtim} \overline{i-1}\dots a
      & \mathrm{if}\ 3\mid \overline{k+1} \\
    \underbrace{\ell_i(\delta)}_{\frac{k+2}{3} \, \rtim} \overline{i-1}\dots a \underbrace{\ell_i(\delta)}_{\frac{k-1}{3} \, \rtim}
    i \underbrace{\ell_i(\delta)}_{\frac{k-1}{3} \, \rtim} \overline{i+1}\dots b
      & \mathrm{if}\ 3\mid \overline{k+2}
  \end{cases},
  \quad \mathrm{if}\ \overline{i-1}>\overline{i+1}
\end{multline}
\begin{equation} \label{genreal four 2}
  \SL(k\delta+\alpha_{i\to b}) =
  \begin{cases}
    \underbrace{\ell_i(\delta)}_{\frac{k}{2} \, \rtim} i \underbrace{\ell_i(\delta)}_{\frac{k}{2} \, \rtim} \overline{i+1}\dots b
      & \mathrm{if}\ 2\mid k \\
    \underbrace{\ell_i(\delta)}_{\frac{k+1}{2} \, \rtim} \overline{i+1}\dots b \underbrace{\ell_i(\delta)}_{\frac{k-1}{2} \, \rtim} i
      & \mathrm{if}\ 2\nmid k
  \end{cases},
  \quad \mathrm{for}\ i \prec b \preceq 0
\end{equation}
\begin{equation} \label{genreal four 3}
  \SL(k\delta+\alpha_{a\to i}) =
  \begin{cases}
    \underbrace{\ell_i(\delta)}_{\frac{k}{2} \, \rtim} i \underbrace{\ell_i(\delta)}_{\frac{k}{2} \, \rtim} \overline{i-1}\dots a
      & \mathrm{if}\ 2\mid k \\
    \underbrace{\ell_i(\delta)}_{\frac{k+1}{2} \, \rtim} \overline{i-1}\dots a \underbrace{\ell_i(\delta)}_{\frac{k-1}{2} \, \rtim} i
      & \mathrm{if}\ 2\nmid k
  \end{cases},
  \quad \mathrm{for}\ 1 \prec a \prec i
\end{equation}
\begin{equation}\label{genreal four 4}
  \SL(k\delta+\alpha_{i})=\underbrace{\ell_i(\delta)}_{k \, \rtim}i
\end{equation}
and finally a slightly less explicit formula
\begin{multline}\label{eq:general-five}
  \SL(k\delta+\alpha_{b\to a}) = \ell_{1}\underbrace{\ell_{b\to a}(\delta)}_{(k-1) \, \rtim}\ell_{2} \,,
    \quad \mathrm{for} \ 1\in \big[b\to \overline{a+1}\big) \\
  \mathrm{where}\ \SL(\delta+\alpha_{b\to a})=\ell_{1}\ell_{2}\ \mathrm{is\ the\ costandard\ factorization}\
    \eqref{eqn:cost.factor}\\
  \mathrm{and}\ \ell_{b\to a}(\delta) \ \mathrm{is\ one\ of}\ \ell_c(\delta)\ \mathrm{such\ that}\
    \SL(2\delta+\alpha_{b\to a})=\ell_{1} \ell_{b\to a}(\delta) \ell_{2}.
\end{multline}
\end{theorem}

\begin{remark}\label{rem:order of 3 parts}
(a) The implicit words $\ell_1$ and $\ell_2$ providing the costandard factorization of $\SL(\delta+\alpha_{b\to a})$
in~\eqref{eq:general-five} can actually be described explicitly (see Lemma~\ref{lemma root with 1}):
\begin{equation*}
\begin{split}
  & \ell_1=\SL(\alpha_{b\to \overline{b-2}}) \quad \mathrm{and} \quad \ell_2=\SL(\alpha_{\overline{b-1}\to a})
    \quad \mathrm{if} \quad \SL(\alpha_{\overline{b-1}\to a}) > \SL(\alpha_{b\to \overline{a+1}}) \,,\\
  & \ell_1=\SL(\alpha_{ \overline{a+2}\to a}) \quad \mathrm{and} \quad \ell_2=\SL(\alpha_{b\to \overline{a+1}})
    \quad \mathrm{if} \quad \SL(\alpha_{\overline{b-1}\to a}) < \SL(\alpha_{b\to \overline{a+1}}) \,.
\end{split}
\end{equation*}

\noindent
(b) Likewise, the word $\ell_{b\to a}(\delta)$ featuring in~\eqref{eq:general-five} can be characterized
as the lexicographically largest among those $\ell_c(\delta)$ that satisfy $[\sb[\ell_1],\sb[\ell_c(\delta)]]\ne 0$.
Explicitly, as follows from the proof below, we have (cf.~part~(a) above):
\begin{equation}\label{eq:el-ba-explicit}
  \ell_{b\to a}(\delta)=
  \begin{cases}
     \ell_{b-1+sgn(i-(b-1))}(\delta) & \mathrm{if}\
       \SL(\alpha_{\overline{b-1}\to a}) > \SL(\alpha_{b\to \overline{a+1}})\\
     \ell_{a+1+sgn(i-(a+1))}(\delta) & \mathrm{if}\
       \SL(\alpha_{\overline{b-1}\to a}) < \SL(\alpha_{b\to \overline{a+1}})
  \end{cases} \,.
\end{equation}

\noindent
(c) Let us also record the explicit order between the words $\ell_1,\ell_2,\ell_{b\to a}(\delta)$, cf.~\eqref{eq:auxiliary}:
\begin{equation*}
  \ell_1 < \ell_2 \leq \ell_{b\to a}(\delta) \,.
\end{equation*}

\noindent
(d) For the standard order~\eqref{eq:stand-order}, we clearly recover the formulas from
our previous Theorem~\ref{thm:sln-standard}. We also note that the proof below
significantly simplifies when $i=2$.

\noindent
(e) Finally, we note $\SL(\alpha_{a\to b})$ can be easily reconstructed using either of the
algorithms presented before Lemma~\ref{lemma root with 1}, with $1$ replaced by
$\min\{a,\overline{a+1},\ldots,\overline{b-1},b\}$.
\end{remark}

\begin{remark}\label{rem:el-words-Lyndon}
(a) In the base of induction below we prove that
\begin{equation}\label{eq:ell-words-Lyndon}
  \big\{ \SL_1(\delta),\ldots,\SL_n(\delta) \big\} =
  \big\{ \ell_c(\delta) \,|\, c\in \wI\setminus \{1\} \big\}.
\end{equation}
As easily follows from~(\ref{eq:el-word-explicit-1},~\ref{eq:el-word-explicit-2}),
their standard bracketings are:
\begin{equation}\label{eq:el-word-bracketing}
  \sb[\ell_c(\delta)] \doteq
  \begin{cases}
    (E_{i+1,i+1}-E_{c,c})t & \mathrm{if}\ 1 \prec c \preceq i\\
    (E_{i,i}-E_{c+1,c+1})t & \mathrm{if}\ i \prec c \preceq 0
  \end{cases}.
\end{equation}

\noindent
(b) The standard bracketing $\sb[\SL(\alpha_{a\to b})]$ for $1,i\notin [a\to \overline{b+1})$ is a nonzero
multiple of $E_{a,b+1}$ if $b\ne 0$, $E_{a1}t$ if $a \prec b=0$, $E_{n+1,1}t$ if $a=b=0$. Thus, the lexicographically
largest word among $\SL_*(\delta)$ whose bracketing $\sb[\SL_*(\delta)]$ does not commute with
$\sb[\SL(\alpha_{a\to b})]$ is $\ell_{\overline{b+1}}(\delta)$ if $a \prec i$ and
$\ell_{\overline{a-1}}(\delta)$ if $a \succ i$, due to Lemma~\ref{deltaorder} and~\eqref{eq:el-word-bracketing}.
\end{remark}

\begin{proof}[Proof of Theorem~\ref{thm:sln-general}]
The proof proceeds by induction on $k$.

\medskip
\noindent
\underline{Base of Induction}

The base of induction is $k=1$. In this case, the nontrivial cases are formulas \eqref{eq:general-one}
for $\SL_*(\delta)$ and~\eqref{eq:general-three}--\eqref{genreal four 4} for $\SL(\delta+\alpha_{a\to b})$
with $1 \notin [a\to \overline{b+1})$.

\medskip
\noindent
$\bullet$ Proof of~\eqref{eq:general-one} for $k=1$.

For any $1\leq r\leq n$, consider the costandard factorization $\SL_r(\delta)=\ell_1\ell_2$. For degree reasons,
we have $\ell_1=\SL(\alpha_{\overline{b+1}\to \overline{a-1}})$, $\ell_2=\SL(\alpha_{a\to b})$ for some
$b\ne \overline{a-1}$ such that $1 \in [\overline{b+1} \to a)$ and $i \in [a \to \overline{b+1})$.
If $b=i$, then $1 \prec a \preceq i$ and
\begin{equation*}
  \SL_r(\delta)=\SL(\alpha_{\overline{i+1}\to \overline{a-1}}) \SL(\alpha_{a\to i})=
  \SL(\alpha_{\overline{i+1}\to \overline{a-1}})\, i\, \overline{i-1} \ldots a = \ell_a(\delta) \,,
\end{equation*}
due to~\eqref{eq:el-word-explicit-1} and Claim~\ref{claim:LR-A-endpoint}. Likewise, if $a=i$, then $i \prec b$ and
\begin{equation*}
  \SL_r(\delta)=\SL(\alpha_{\overline{b+1}\to \overline{i-1}}) \SL(\alpha_{i\to b})=
  \SL(\alpha_{\overline{b+1}\to \overline{i-1}})\, i\, \overline{i+1} \ldots b = \ell_b(\delta) \,,
\end{equation*}
due to~\eqref{eq:el-word-explicit-2} and Claim~\ref{claim:LR-A-endpoint}. Finally, if $1 \prec a \prec i \prec b$,
then $\SL_r(\delta)<\ell_c(\delta)$ for any $c \in [a\to \overline{b+1})$, due to Lemma~\ref{Sasha's proposition}
and explicit formulas~(\ref{eq:el-word-explicit-1},~\ref{eq:el-word-explicit-2}). On the other hand,
$\sb[\SL_r(\delta)]=[\sb[\ell_1],\sb[\ell_2]]\doteq (E_{a,a}-E_{b+1,b+1})t$, while the standard bracketing
$\sb[\ell_c(\delta)]$ is given by~\eqref{eq:el-word-bracketing}. Hence, $\sb[\SL_r(\delta)]$ is a linear combination
of standard bracketings of the larger words $\ell_a(\delta), \ell_i(\delta),\ell_b(\delta)$,
a contradiction with $\SL_r(\delta)$ being standard. Thus, any degree $\delta$ $\aslaw$ is of the form
$\ell_c(\delta)$ for $c\ne 1$. This completes the proof of~\eqref{eq:ell-words-Lyndon}, as we have $n$ such~words.

\medskip
\noindent
$\bullet$ Proof of~\eqref{eq:general-three}--\eqref{genreal four 4} for $k=1$.

We skip these proofs as they coincide with those in the step of induction below.

\medskip
\noindent
\underline{Step of Induction}

Let us now prove the step of induction, proceeding by the height of a root. We thus verify
formulas~\eqref{eq:general-one}--\eqref{eq:general-five} for $\aslaws$ $\SL_*(\alpha)$ with $k=r+1$ assuming
the validity of these formulas for $\SL_*(\beta)$ with $\hgt(\beta)<\hgt(\alpha)$.

\medskip
\noindent
\textbf{Notation:} In what follows, we shall denote $[a\to \overline{b+1})$ from~\eqref{eq:arch-elements}
simply by $[a;b]$:
\begin{equation*}
  [a;b]:=\big\{ a, \overline{a+1}, \ldots, \overline{b-1}, b \big\} \, .
\end{equation*}

\noindent
$\bullet$ Proof of~\eqref{eq:general-one} for $k=r+1$.

We consider only decompositions of the form
  $(r+1)\delta=(r_1\delta+\alpha_{a \to b})+((r-r_1)\delta+\alpha_{\overline{b+1} \to \overline{a-1}})$,
due to Remark~\ref{rem:gen Lyndon rmk}. We may further assume that $1\in [\overline{b+1};\overline{a-1}]$.
We start with the following useful result (which will be strengthened in Lemma~\ref{lemma root with 1}):

\begin{claim}\label{bracketing with 1}
If $\ell_1\ell_2$ is the costandard factorization~\eqref{eqn:cost.factor} of
$\SL(\delta+\alpha_{\overline{b+1} \to \overline{a-1}})$ and $1\in [\overline{b+1};\overline{a-1}]$, then
both words $\ell_1$ and $\ell_2$ contain all the letters located on the (counterclockwise oriented) arch
$[\overline{b+1};\overline{a-1}]$.
\end{claim}

\begin{proof}[Proof of Claim~\ref{bracketing with 1}]
First, we note that both $\ell_1,\ell_2$ start with $1$. If $\ell_1$ does not contain all the
letters from $[\overline{b+1};\overline{a-1}]$, then it consists only of letters from $c$ to $d$,
where $1\in [c;d] \subsetneq [\overline{b+1};\overline{a-1}]$. Thus,
$\ell_1<\SL(\alpha_{\overline{b+1} \to \overline{a-1}})$ by Lemma~\ref{Sasha's proposition}, hence
\begin{equation}\label{eq:contradiction with costand}
  \SL(\delta+\alpha_{\overline{b+1} \to \overline{a-1}})=\ell_1\ell_2<
  \SL(\alpha_{\overline{b+1} \to \overline{a-1}})\ell_{e(i;a,b)}(\delta) \,,
\end{equation}
with $e(i;a,b):=a$ if $a\preceq i$ and $e(i;a,b):=b$ if $i\prec a\preceq b$.
However, $\SL(\alpha_{\overline{b+1} \to \overline{a-1}})<\ell_{e(i;a,b)}(\delta)$ by
Lemma~\ref{Sasha's proposition} and their standard bracketings do not commute by~\eqref{eq:el-word-bracketing}:
  $[\sb[\SL(\alpha_{\overline{b+1} \to \overline{a-1}})],\sb[\ell_{e(i;a,b)}(\delta)]]\ne 0$.
Thus, the concatenated word $\SL(\alpha_{\overline{b+1} \to \overline{a-1}})\ell_{e(i;a,b)}(\delta)$
appears in the set from the right-hand side of~\eqref{eq:generalized Leclerc} for the root
$\alpha=\delta+\alpha_{\overline{b+1} \to \overline{a-1}}$, contradicting~\eqref{eq:contradiction with costand}.

If $\ell_2$ does not contain all the letters from $[\overline{b+1};\overline{a-1}]$, then we apply precisely
the same argument to $\ell_2\ell_1$ and use the inequality $\ell_1\ell_2<\ell_2\ell_1$ to get a contradiction.
\end{proof}

For $r_1<r$, we have
  $\SL((r-r_1)\delta+\alpha_{\overline{b+1} \to \overline{a-1}})=
   \ell_1 \underbrace{\ell_{\overline{b+1}\to \overline{a-1}}(\delta)}_{(r-r_1-1) \, \mathrm{times}}\ell_2$
by the induction hypothesis, where $\ell_1\ell_2$ is the costandard factorization of
$\SL(\delta+\alpha_{\overline{b+1} \to \overline{a-1}})$. According to Claim~\ref{bracketing with 1}:
$\sb[\ell_1]\doteq E_{\overline{b+1},c}t^{1-\delta_{\overline{b+1},1}}$ for some
$c \in [a;b]$ or $\sb[\ell_1]\doteq E_{c,a}t$ for some $c \in [a+1;b]$. For any $d\in [a;b]$, one
of the roots $\deg\,\ell_1,\deg\,\ell_2\in \wDelta^+$ does not contain $\alpha_d$, which together
with $\ell_1<\ell_2$, Lemma~\ref{Sasha's proposition}, and Claim~\ref{bracketing with 1} implies:
\begin{equation}\label{eq:ell1-vs-d}
  \ell_1 \leq \SL(\alpha_{\overline{d+1}\to \overline{d-1}}) \,.
\end{equation}
Moreover, the equality in~\eqref{eq:ell1-vs-d} does hold only for $d=b$ if
$\SL(\alpha_{b \to \overline{a-1}})>\SL(\alpha_{\overline{b+1} \to a})$ and for $d=a$
if $\SL(\alpha_{b \to \overline{a-1}})<\SL(\alpha_{\overline{b+1} \to a})$, according to Lemma~\ref{lemma root with 1}.

Thus, if $a\ne b$ and $\SL(\alpha_{b \to \overline{a-1}})>\SL(\alpha_{\overline{b+1} \to a})$,
then for $d\in [a;\overline{b-1}]$ we have:
\begin{multline*}
  \SL((r-r_1)\delta+\alpha_{\overline{b+1} \to \overline{a-1}})\SL(r_1\delta+\alpha_{a \to b})<
   \SL(\alpha_{\overline{d+1}\to \overline{d-1}})< \\
   \SL(\alpha_{\overline{d+1}\to \overline{d-1}})\underbrace{\ell_{d+sgn(i-d)}(\delta)}_{r \, \rtim}d.
\end{multline*}
In the remaining case $d=b$ (with $a\ne b$ and $\SL(\alpha_{b \to \overline{a-1}})>\SL(\alpha_{\overline{b+1} \to a})$),
we have:
\begin{multline*}
  \SL((r-r_1)\delta+\alpha_{\overline{b+1} \to \overline{a-1}})\SL(r_1\delta+\alpha_{a \to b})=\\
  \SL(\alpha_{\overline{b+1} \to \overline{b-1}})
    \underbrace{\ell_{\overline{b+1} \to \overline{a-1}}(\delta)}_{(r-r_1-1) \, \rtim}
    \ell_2\, \SL(r_1\delta+\alpha_{a \to b}) < \\
  \SL(\alpha_{\overline{b+1} \to \overline{b-1}})
     \underbrace{\ell_{\overline{b+1} \to \overline{a-1}}(\delta)}_{r \; \mathrm{times}} b =
  \SL(\alpha_{\overline{b+1} \to \overline{b-1}})
     \underbrace{\ell_{b+sgn(i-b)}(\delta)}_{r \; \mathrm{times}} b \,,
\end{multline*}
cf.~\eqref{eq:el-ba-explicit},
with the inequality implied by
  $\ell_2<\ell_{\overline{b+1} \to \overline{a-1}}(\delta)$,
due to~\eqref{eq:auxiliary} and $a\ne b$. The case of $a\ne b$ and
$\SL(\alpha_{b \to \overline{a-1}})<\SL(\alpha_{\overline{b+1} \to a})$ is treated completely analogously.

On the other hand, if $a=b=d$ and $r_1\geq 0$, then
\begin{equation*}
  \SL(r_1\delta+\alpha_{a \to b})=\SL(r_1\delta+\alpha_a)=\underbrace{\ell_{a+sgn(i-a)}(\delta)}_{r_1 \, \rtim}a
\end{equation*}
by the induction hypothesis (applying~\eqref{eq:general-three} if $a<i$,~\eqref{eq:general-two} if $a>i$,
\eqref{genreal four 4} if $a=i$) and
  $\SL((r-r_1)\delta+\alpha_{\overline{b+1} \to \overline{a-1}})=
   \SL((r-r_1)\delta+\alpha_{\overline{a+1} \to \overline{a-1}})$
is given by
\begin{equation}\label{eq:aux-gen-one}
  \SL((r-r_1)\delta+\alpha_{\overline{a+1} \to \overline{a-1}})=
  \SL(\alpha_{\overline{a+1}\to \overline{a-1}})\underbrace{\ell_{a+sgn(i-a)}(\delta)}_{(r-r_1) \, \rtim} \,.
\end{equation}
To prove the latter claim, we first note that $\ell_1=\SL(\alpha_{\overline{a+1}\to \overline{a-1}})$ and
$\ell_2=\SL_?(\delta)$, while the lexicographically largest word $\SL_?(\delta)$ whose bracketing does not
commute with $\sb[\SL(\alpha_{\overline{a+1}\to \overline{a-1}})]\doteq E_{a+1,a}t^{1-\delta_{a,0}}$ is
precisely $\ell_{a+sgn(i-a)}(\delta)$, due to~\eqref{eq:el-word-bracketing} and Lemma~\ref{deltaorder}.
Therefore, $\ell_2=\ell_{a+sgn(i-a)}(\delta)$.
Second, we also claim that $\ell_{\overline{a+1} \to \overline{a-1}}(\delta)$ equals
$\ell_2=\ell_{a+sgn(i-a)}(\delta)$. To this end, recall that for
$\alpha=2\delta+\alpha_{\overline{a+1}\to \overline{a-1}}$ we have
\begin{equation}\label{eq:middls-vs-tail}
  \SL(\alpha)=\ell_1 \ell_{\overline{a+1} \to \overline{a-1}}(\delta) \ell_2 =
  \SL(\alpha_{\overline{a+1}\to \overline{a-1}})
  \ell_{\overline{a+1} \to \overline{a-1}}(\delta)\ell_{a+sgn(i-a)}(\delta) \,.
\end{equation}

\noindent
$\circ$
If $\ell_{\overline{a+1} \to \overline{a-1}}(\delta)<\ell_{a+sgn(i-a)}(\delta)$, then
  $\SL(\alpha_{\overline{a+1}\to \overline{a-1}})
      \ell_{\overline{a+1} \to \overline{a-1}}(\delta)\ell_{a+sgn(i-a)}(\delta)<
   \SL(\alpha_{\overline{a+1}\to \overline{a-1}})\ell_{a+sgn(i-a)}(\delta)\ell_{a+sgn(i-a)}(\delta)=:\wt{\ell}$
and the bracketing of the latter is
\begin{multline*}
  \sb[\wt{\ell}]=
  [\sb[\SL(\alpha_{\overline{a+1}\to \overline{a-1}})\ell_{a+sgn(i-a)}(\delta)],\sb[\ell_{a+sgn(i-a)}(\delta)]]\doteq \\
  [\sb[\SL(\alpha_{\overline{a+1}\to \overline{a-1}})],\sb[\ell_{a+sgn(i-a)}(\delta)]]\cdot t\ne 0 \,.
\end{multline*}
We get a contradiction, since $\wt{\ell}$ is one of the concatenations
(corresponding to the decomposition $\alpha=(\delta+\alpha_{\overline{a+1}\to \overline{a-1}})+(\delta)$)
in the right-hand side of~\eqref{eq:generalized Leclerc} for $\alpha$.

\noindent
$\circ$
If $\ell_{\overline{a+1} \to \overline{a-1}}(\delta)>\ell_{a+sgn(i-a)}(\delta)$, then the costandard
factorization~\eqref{eqn:cost.factor} of $\SL(\alpha)$ in~\eqref{eq:middls-vs-tail} must be of the form
$\SL(\alpha)=\ell'_1\ell'_2$ with $\ell'_2=\ell_{a+sgn(i-a)}(\delta)$ and
  $\ell'_1=\SL(\alpha_{\overline{a+1}\to \overline{a-1}})\ell_{\overline{a+1} \to \overline{a-1}}(\delta)$.
We get a contradiction again, since $\ell'_1$ is an $\SL$-word and so
  $\ell'_1=\SL(\deg\, \ell'_1)=\SL(\delta+\alpha_{\overline{a+1}\to \overline{a-1}})=
   \SL(\alpha_{\overline{a+1}\to \overline{a-1}})\ell_{a+sgn(i-a)}(\delta)$.

This completes our proof of~\eqref{eq:aux-gen-one}.
Assuming $\SL((r-r_1)\delta+\alpha_{\overline{b+1} \to \overline{a-1}})<\SL(r_1\delta+\alpha_{a \to b})$ and
combining all the above, we obtain the following inequalities for the corresponding concatenation
  $\ell:=\SL((r-r_1)\delta+\alpha_{\overline{b+1} \to \overline{a-1}})\SL(r_1\delta+\alpha_{a \to b})$:
\begin{equation}\label{eq:ell-vs-imaginary}
  \ell\leq \SL(\alpha_{\overline{d+1}\to \overline{d-1}})\underbrace{\ell_{d+sgn(i-d)}(\delta)}_{r \, \rtim}d
  \qquad  \forall\, d \in [a;b] \,.
\end{equation}
We also note that~\eqref{eq:ell-vs-imaginary} still holds for $r_1=r$, due to Lemma~\ref{Sasha's proposition}.

The standard bracketings of the words from the right-hand side of~\eqref{eq:general-one} are:
\begin{equation}\label{r+1 imaginary}
  \sb[\SL(\alpha_{\overline{c+1}\to \overline{c-1}})\underbrace{\ell_{c+sgn(i-c)}(\delta)}_{r \, \rtim}c]\doteq
  \begin{cases}
     (E_{cc}-E_{c+1,c+1})t^{r+1} & \mathrm{if}\ 1<c\leq n\\
     (E_{n+1,n+1}-E_{11})t^{r+1} & \mathrm{if}\ c=0
  \end{cases} \,.
\end{equation}
We shall now compute the standard bracketing of $\ell$. We have two possibilities
(due to the inequalities of Remark~\ref{rem:order of 3 parts}(c)):
\begin{itemize}[leftmargin=0.7cm]

\item[1)]
The costandard factorization~\eqref{eqn:cost.factor} of $\ell$ is of the form:
\begin{equation*}
  \ell=\ell'_1\ell'_2 \quad \mathrm{with} \quad
  \ell'_1=\SL((r-r_1)\delta+\alpha_{\overline{b+1} \to \overline{a-1}}) \,,\,
  \ell'_2=\SL(r_1\delta+\alpha_{a \to b}) \,.
\end{equation*}
Hence, the standard bracketing of $\ell$ is:
\begin{multline*}
  \sb[\ell]=[\sb[\ell'_1],\sb[\ell'_2]]\doteq (E_{aa}-E_{b+1,b+1})t^{r+1}\doteq \\
  (E_{aa}-E_{a+1,a+1})t^{r+1}+(E_{a+1,a+1}-E_{a+2,a+2})t^{r+1}+\dots+(E_{bb}-E_{b+1,b+1})t^{r+1}.
\end{multline*}
Thus, if $\ell$ is not a word from the right-hand side of~\eqref{eq:general-one} for $k=r+1$,
then $\sb[\ell]$ is a linear combination of the standard bracketings of the larger words
$\{\ell_d(\delta) \,|\, d\in [a;b]\}$, cf.~(\ref{eq:ell-vs-imaginary},~\ref{r+1 imaginary}).
Hence, the word $\ell$ can not be standard.

\item[2)]
The costandard factorization~\eqref{eqn:cost.factor} of $\ell$ is of the form:
\begin{equation*}
  \ell=\ell'_1\ell'_2 \quad \mathrm{with} \quad
  \ell'_1=\ell_1 \underbrace{\ell_{\overline{b+1}\to \overline{a-1}}(\delta)}_{(r-r_1-1) \, \rtim} \,,\,
  \ell'_2=\ell_2\, \SL(r_1\delta+\alpha_{a \to b}) \,.
\end{equation*}
Hence, the standard bracketing of $\ell$ is either $\sb[\ell]\doteq (E_{cc}-E_{b+1,b+1})t^{r+1}$ for
$c\in [a;b]$ or $\sb[\ell]\doteq (E_{aa}-E_{cc})t^{r+1}$ for $c\in [a+1;b]$. Thus, analogously to 1),
if $\ell$ is not a word from the right-hand side of~\eqref{eq:general-one} for $k=r+1$, then
$\sb[\ell]$ is a linear combination of the standard bracketings of the larger words
$\{\ell_d(\delta) \,|\, d\in [a;b]\}$, cf.~(\ref{eq:ell-vs-imaginary},~\ref{r+1 imaginary}).
Therefore, the word $\ell$ can not be standard.

\end{itemize}
Finally, if $\SL((r-r_1)\delta+\alpha_{\overline{b+1} \to \overline{a-1}})>\SL(r_1\delta+\alpha_{a \to b})$,
then the concatenation $\wt{\ell}$ arising from the decomposition
$(r+1)\delta=(r_1\delta+\alpha_{a \to b})+((r-r_1)\delta+\alpha_{\overline{b+1} \to \overline{a-1}})$ is
\begin{equation}\label{eq:other ordering Ok}
  \wt{\ell}=\SL(r_1\delta+\alpha_{a \to b})\SL((r-r_1)\delta+\alpha_{\overline{b+1} \to \overline{a-1}})<\ell \,,
\end{equation}
due to Lemma~\ref{lemma:lyndon}. By induction hypothesis, we have  $\sb[\wt{\ell}]\doteq (E_{pp}-E_{qq})t^{r+1}$ for some
$p,q\in [a;\overline{b+1}]$. The latter is a linear combination of standard bracketings of the larger words
$\{\ell_d(\delta) \,|\, d\in [a;b]\}$, cf.~\eqref{eq:ell-vs-imaginary}--\eqref{eq:other ordering Ok},
hence $\wt{\ell}$ is not standard either.

\medskip
\noindent
$\bullet$ Proof of~\eqref{eq:general-three} for $k=r+1$.

Consider $\alpha=(r+1)\delta+\alpha_{a\to b}$ with $1 \prec a \preceq b \prec i$. Its possible decompositions are
$\alpha=(r_1\delta+\alpha_{a\to c})+(r_2\delta+\alpha_{\overline{c+1}\to b})$ with $r_1+r_2=$ $r$ or $r+1$,
depending on $c$.

First, we show that decompositions with $c\notin [a;b]$ give rise to concatenated words which are lexicographically
smaller than the word in the right-hand side of~\eqref{eq:general-three} for $k=r+1$. There are four cases to consider:
$1\in [a;c]$ or $1\in [\overline{c+1};b]$, treating separately $r_1=0, r_1\geq 1$ in the first case and
$r_2=0, r_2\geq 1$ in the second case.

1) If $1\in [a;c]\ne \wI$ and $r_1=0$, then $1\in [a;c]\subset [e+1;e-1]$ for any $e\in [c+1;a-1]$,
and so $\SL(\alpha_{a\to c})\leq \SL(\alpha_{(e+1)\to (e-1)})$ by Lemma~\ref{Sasha's proposition}.
As $1=\min\, \wI$, we get:
  $\SL(\alpha_{a\to c})\, 1 < \SL(\alpha_{(e+1)\to (e-1)})\, e=\ell_e(\delta)<\ell_a(\delta)<\ell_{b+1}(\delta)$
with the last two inequalities due to Lemma~\ref{deltaorder}. We note that $\SL(\alpha_{a\to c})\, 1$ cannot be
a proper prefix of $\ell_{b+1}(\delta)$ (as the former word contains the letter $1$ twice) and
$\SL(r\delta+\alpha_{\overline{c+1}\to b})$ starts with $1$. Thus, the concatenation
$\SL(\alpha_{a\to c})\SL(r\delta+\alpha_{\overline{c+1}\to b})$ is lexicographically smaller than
$\ell_{b+1}(\delta)$, hence, smaller than the right-hand side of~\eqref{eq:general-three} for $k=r+1$.

2) If $1\in [\overline{c+1};b]$ and $r_2=0$, then $1\in [\overline{c+1};b]\subset [\overline{b+2};b]$, and so
$\SL(\alpha_{\overline{c+1}\to b})\leq \SL(\alpha_{\overline{b+2}\to b})$ by Lemma~\ref{Sasha's proposition}.
Thus, $\SL(\alpha_{\overline{c+1}\to b})\, 1 < \SL(\alpha_{\overline{b+2}\to b})(b+1)=\ell_{b+1}(\delta)$.
The rest of the argument proceeds exactly as in 1) above.

3) If $1\in [a;c]\ne \wI$ and $r_1\geq 1$, then
  $\SL(r_1\delta+\alpha_{a\to c})=\ell_{1}\underbrace{\ell_{a\to c}(\delta)}_{(r_1-1) \, \rtim}\ell_{2}$
with $\ell_1$ and $\ell_2$ defined through the costandard factorization $\SL(\delta+\alpha_{a\to c})=\ell_{1}\ell_{2}$.
We claim that $\ell_1<\ell_{b+1}(\delta)$, from which the argument proceeds exactly as in 1) above.
Indeed, according to Lemma~\ref{lemma root with 1}, $\ell_1$ is given by one of the following two formulas:
\begin{itemize}[leftmargin=1cm]

\item[(A)]
$\ell_1=\SL(\alpha_{a\to d})$ for $d\in [c\to (a-1))$;

\item[(B)]
$\ell_1=\SL(\alpha_{d\to c})$ for $d\in [(c+2)\to a)$.

\end{itemize}
According to Lemmas~\ref{Sasha's proposition},~\ref{deltaorder}, we thus get:
$\ell_1\leq \SL(\alpha_{a\to (a-2)})<\ell_{a-1}(\delta)<\ell_{b+1}(\delta)$ in case (A)
and $\ell_1\leq \SL(\alpha_{(c+2)\to c})<\ell_{c+1}(\delta)<\ell_{b+1}(\delta)$ in case (B), as stated~above.

4) If $1\in [\overline{c+1};b]\ne \wI$ and $r_2\geq 1$, then
  $\SL(r_2\delta+\alpha_{\overline{c+1}\to b})=
   \ell_{1}\underbrace{\ell_{\overline{c+1}\to b}(\delta)}_{(r_2-1) \, \rtim}\ell_{2}$
with $\ell_1$ and $\ell_2$ defined through the costandard factorization
$\SL(\delta+\alpha_{\overline{c+1}\to b})=\ell_{1}\ell_{2}$. We claim that $\ell_{1}<\ell_{b+1}(\delta)$,
from which the argument proceeds exactly as in 1) above. Indeed, according to Lemma~\ref{lemma root with 1},
$\ell_1$ is given by one of the following two formulas:
\begin{itemize}[leftmargin=1cm]

\item[(A)]
$\ell_1=\SL(\alpha_{d\to b})$ for $d\in [\overline{b+2};\overline{c+1}]$;

\item[(B)]
$\ell_1=\SL(\alpha_{\overline{c+1}\to d})$ for $d\in [(b+1)\to c)$.

\end{itemize}
According to Lemmas~\ref{Sasha's proposition},~\ref{deltaorder}, we thus get:
$\ell_1\leq \SL(\alpha_{\overline{b+2}\to b})<\ell_{b+1}(\delta)$ in case (A) and
  $\ell_1<\ell_2 = \SL(\alpha_{\overline{d+1}\to b})\leq \SL(\alpha_{\overline{b+2}\to b})<\ell_{b+1}(\delta)$
in case (B), as claimed~above.

Therefore, it suffices to consider only the following decompositions in~\eqref{eq:generalized Leclerc}:
\begin{equation}\label{abr}
  \alpha=(r_1\delta+\alpha_{a\to c})+((r+1-r_1)\delta+\alpha_{(c+1)\to b}) \,,
  \quad a \preceq c \prec b \,,\, 0\leq r_1\leq r+1 ,
\end{equation}
\begin{equation}\label{abi}
  \alpha=(r_1\delta)+((r+1-r_1)\delta+\alpha_{a\to b}) \,,
  \quad 1\leq r_1\leq r+1.
\end{equation}

\noindent
$\circ$
Case 1: Concatenations arising through~\eqref{abr}.

1) If $0<r_1<r+1$, then the corresponding concatenated word starts with $\ell_{c+1}(\delta)$, due to
the induction hypothesis and the inequality $\ell_{c+1}(\delta)<\ell_{b+1}(\delta)$ of Lemma~\ref{deltaorder}.
Thus, this concatenation is $<$ the right-hand side of~\eqref{eq:general-three} for $k=r+1$.

2) If $r_1=r+1$, then the corresponding concatenated word again starts with $\ell_{c+1}(\delta)$,
but now because the first letter of $\ell_{c+1}(\delta)$ is smaller than any of $c+1,\ldots,b$.
Therefore, this concatenation is $<$ the right-hand side of~\eqref{eq:general-three} for $k=r+1$.

3) If $r_1=0$, then the concatenation equals
$\underbrace{\ell_{b+1}(\delta)}_{(r+1) \, \rtim}b(b-1)\dots (c+1)\, \SL(\alpha_{a\to c})$.
But $\SL(\alpha_{a\to c})\leq c (c-1) \dots a$ (either they differ in the first letters, or
Claim~\ref{claim:LR-A-endpoint} applies), hence, this concatenation is $\leq$ the right-hand side
of~\eqref{eq:general-three} for $k=r+1$.

\noindent
$\circ$
Case 2: concatenations arising through~\eqref{abi}.

\noindent
First, we record the standard bracketing $\sb[\SL((r+1-r_1)\delta+\alpha_{a\to b})]\doteq E_{a,b+1}t^{r+1-r_1}$.

1) If $r_1>1$, then according to \eqref{r+1 imaginary} the only words from the right-hand side
of~\eqref{eq:general-one} with $k=r_1$ whose standard bracketing does not commute with the above
$\sb[\SL((r+1-r_1)\delta+\alpha_{a\to b})]$ start with $\SL(\alpha_{\overline{c+1}\to \overline{c-1}})1$
for $c=a-1,a,b,b+1$. Each of these words is lexicographically smaller than $\ell_{b+1}(\delta)$. Hence,
the corresponding concatenation is $<$ the right-hand side of~\eqref{eq:general-three} for $k=r+1$.

2) If $r_1=1$, then we should rather use formula~\eqref{eq:el-word-bracketing} for the bracketings.

\noindent
$\circ$
If $b \prec (i-1)$, then the only $\ell_?(\delta)$ whose standard bracketing does not commute with
$\sb[\SL(r\delta+\alpha_{a\to b})]$ are $\ell_a(\delta)$ and $\ell_{b+1}(\delta)$. As
$\ell_a(\delta)<\ell_{b+1}(\delta)$ by Lemma~\ref{deltaorder}, the resulting concatenation
is $\leq$ the right-hand side of~\eqref{eq:general-three} for $k=r+1$.

\noindent
$\circ$
If $b=i-1$, then the only $\ell_?(\delta)$ whose standard bracketing does not commute with
$\sb[\SL(r\delta+\alpha_{a\to b})]$ are $\ell_a(\delta)$ and $\{\ell_{c}(\delta)|c\geq i\}$.
As $\ell_i(\delta)$ is the maximal of these words (Lemma~\ref{deltaorder}), the concatenation
is still $\leq$ the right-hand side of~\eqref{eq:general-three} for~$k=r+1$.

We note that in both cases above the equality is possible (when $\ell_{b+1}(\delta)$ is used).

\noindent
This completes our proof of~\eqref{eq:general-three} for $k=r+1$.

\medskip
\noindent
$\bullet$ Proof of~\eqref{eq:general-two} for $k=r+1$.

The argument is completely analogous to the one used in the previous case
(we leave details to the interested reader).

\medskip
\noindent
$\bullet$ Proof of~\eqref{eq:general-four}--\eqref{genreal four 4} for $k=r+1$.

Let us prove the most complicated formula~\eqref{eq:general-four} for the case $\alpha=(r+1)\delta+\alpha_{a\to b}$
with $1 \prec a \prec i \prec b$ and $\overline{i-1}<\overline{i+1}$ (the proofs for the other cases are analogous).

There exists a degree $\alpha$ Lyndon word with a nonzero bracketing that starts with $\SL_1(\delta)=\ell_i(\delta)$.
Therefore, it suffices to consider in~\eqref{eq:generalized Leclerc} only those decompositions
$\alpha=(r_1\delta+\beta_1)+(r_2\delta+\beta_2)$ such that each word $\SL(r_1\delta+\beta_1)$, $\SL(r_2\delta+\beta_2)$
is either $>\ell_i(\delta)$ or is a prefix of $\ell_i(\delta)$. This excludes the following cases (with $p=1,2$):

1) $\beta_{p}=\alpha_{a\to c}$ with $1\in [a;c]\ne \wI$, as in this case we have
$\SL(\alpha_{a\to c})\, 1  < \ell_i(\delta)$ and $\ell_1\, 1 < \ell_i(\delta)$ with $\ell_1$ arising through the
costandard factorization $\SL(\delta+\alpha_{a\to c})=\ell_1\ell_2$, cf.~our verification of~\eqref{eq:general-three} above;

2) $\beta_{p}=\alpha_{c\to b}$ with $1\in [c;b]\ne \wI$, as in this case we have
$\SL(\alpha_{c\to b})\, 1  < \ell_i(\delta)$ and $\ell_1\, 1 < \ell_i(\delta)$ with $\ell_1$ arising through the
costandard factorization $\SL(\delta+\alpha_{c\to b})=\ell_1\ell_2$, cf.~our verification of~\eqref{eq:general-three} above;

3) $\beta_{p}=k\delta$ with $k>1$, as
  $\SL(\alpha_{\overline{c+1}\to \overline{c-1}})\,1 < \SL(\alpha_{\overline{c+1}\to \overline{c-1}})\,c =
   \ell_c(\delta)\leq \ell_i(\delta)\ \forall\,c$;

4) $\beta_p=\alpha_{a\to c}$ with $c\in [a\to (i-1))$ and $r_p>0$, as $\SL(r_p\delta+\beta_p)$ then starts
with $\ell_{c+1}(\delta)$ which has the same length but is lexicographically smaller than $\ell_i(\delta)$;

5) $\beta_p=\alpha_{\overline{c+1}\to b}$ with $c\in [\overline{i+1}\to b)$ and $r_p>0$, as $\SL(r_p\delta+\beta_p)$
then starts with $\ell_{c}(\delta)$ which has the same length but is lexicographically smaller than $\ell_i(\delta)$.

\noindent
Furthermore, if $\beta_p=\alpha_{a\to c}$ with $c\in [a\to (i-1))$ and $r_p=0$, then the corresponding
concatenation $\SL((r+1)\delta+\alpha_{(c+1)\to b})\SL(\alpha_{a\to c})$ is $\leq$ the right-hand side
of~\eqref{eq:general-four} for $k=r+1$, due to the inequality
$\overline{i-1} \ldots (c+1)\SL(\alpha_{a\to c})\leq \overline{i-1} \ldots (c+1)\,c \ldots a$
(implied by Claim~\ref{claim:LR-A-endpoint}) and the induction hypothesis. Likewise, if
$\beta_p=\alpha_{\overline{c+1}\to b}$ with $c\in [\overline{i+1}\to b)$ and $r_p=0$, then the corresponding
concatenation $\SL((r+1)\delta+\alpha_{a\to c})\SL(\alpha_{\overline{c+1}\to b})$ is $\leq$
the right-hand side of~\eqref{eq:general-four} for $k=r+1$, due to the similar inequality
$\overline{i+1} \dots c\, \SL(\alpha_{\overline{c+1}\to b})\leq \overline{i+1} \dots b$ and the induction hypothesis.

Therefore, it suffices to consider only the following decompositions in~\eqref{eq:generalized Leclerc}:
\begin{equation}\label{eq:aib-decomp}
\begin{split}
  & \alpha=(r_1\delta + \alpha_{a\to \overline{i-1}})+((r+1-r_1)\delta+\alpha_{i \to b}) \,, \quad 0\leq r_1\leq r+1 \\
  & \alpha=(r_1\delta+\alpha_{a\to i})+((r+1-r_1)\delta + \alpha_{\overline{i+1}\to b}) \,,\quad 0\leq r_1\leq r+1 \\
  & \alpha=(\delta)+(r\delta + \alpha_{a\to b}).
\end{split}
\end{equation}
Clearly, we can choose only $\SL_1(\delta)=\ell_i(\delta)$ in the latter case.
By the induction hypothesis, all the corresponding concatenations have the following specific form:
\begin{multline}\label{eq:triple-ell-factor}
  \ell=\underbrace{\ell_i(\delta)}_{p \, \rtim} \ell_1 \underbrace{\ell_i(\delta)}_{q \, \rtim} \ell_2
        \underbrace{\ell_i(\delta)}_{m \, \rtim}\ell_3
  \quad \mathrm{with} \\
  p+q+m=r+1 \quad \mathrm{and} \quad
  \{\ell_1 \,,\, \ell_2 \,,\, \ell_3\}=\{\overline{i-1} \dots a \,,\, i \,,\, \overline{i+1}\dots b\}.
\end{multline}
Since the corresponding concatenation $\ell$ is Lyndon (Lemma~\ref{lemma:lyndon}) and $\ell_i(\delta)$ starts
with $1$ which is smaller than the first letter of the words $\ell_1,\ell_2,\ell_3$, we must have
\begin{equation}\label{eq:pqm}
  p\geq q \quad \mathrm{and} \quad p\geq m.
\end{equation}
Let us consider three cases:

$\circ$ Case 1: $3 \,|\, (r+1)$.
According to~\eqref{eq:pqm}, we have $p\geq \frac{r+1}{3}$. To get the lexicographically
largest word, we need to pick $p$ the smallest possible: $p=\frac{r+1}{3}$. As $p\geq q,m$ and $p+q+m=r+1$, we have
$p=q=m=\frac{r+1}{3}$. Additionally, $\ell$ being Lyndon implies $\ell_1<\ell_2$ and $\ell_1<\ell_3$ if $p=q=m$.
It thus follows that $\ell_1=i$. As we assumed $\overline{i+1}>\overline{i-1}$, the largest word occurs if
$\ell_2=\overline{i+1}\dots b > \ell_3=\overline{i-1} \dots a$. Thus, we end up exactly with the word in the
right-hand side of~\eqref{eq:general-four} for $k=r+1$:
\begin{equation*}
  \ell_{\max} \, =
  \underbrace{\ell_i(\delta)}_{\frac{r+1}{3} \, \rtim} i \underbrace{\ell_i(\delta)}_{\frac{r+1}{3} \, \rtim}
  \overline{i+1}\dots b \underbrace{\ell_i(\delta)}_{\frac{r+1}{3} \, \rtim}\overline{i-1}\dots a.
\end{equation*}
This word arises from the decomposition
  $\alpha=(\frac{2(r+1)}{3}\delta+\alpha_{i\to b})+(\frac{r+1}{3}\delta+\alpha_{a\to \overline{i-1}})$.
The latter provides the costandard factorization of $\ell_{\max}$, in particular, $\sb[\ell_{\max}]\ne 0$.

$\circ$ Case 2: $3 \,|\, (r+2)$.
According to~\eqref{eq:pqm}, we have $p\geq \frac{r+2}{3}$. To get the lexicographically
largest word, we need to pick $p$ the smallest possible: $p=\frac{r+2}{3}$. Then, we have
$\{q,m\}=\{\frac{r+2}{3},\frac{r-1}{3}\}$. As $\ell$ is Lyndon and $q=p$ or $m=p$, $\ell_1\leq \ell_2$
or $\ell_1\leq \ell_3$, respectively. Thus, $\ell_1$ equals $i$ or $\overline{i-1} \dots a$, and to get
the lexicographically largest word, we need to pick $\ell_1=\overline{i-1} \dots a$ and $q=\frac{r-1}{3}$.
Then $m=\frac{r+2}{3}$, and $\ell$ being Lyndon implies that $\ell_3=\overline{i+1} \dots b$, so that $\ell_2=i$.
Thus, we end up exactly with the word in the right-hand side of~\eqref{eq:general-four} for $k=r+1$:
\begin{equation*}
  \ell_{\max} \, =
  \underbrace{\ell_i(\delta)}_{\frac{r+2}{3} \, \rtim} \overline{i-1} \dots a
  \underbrace{\ell_i(\delta)}_{\frac{r-1}{3} \, \rtim} i
  \underbrace{\ell_i(\delta)}_{\frac{r+2}{3} \, \rtim} \overline{i+1}\dots b.
\end{equation*}
This word arises from the decomposition
  $\alpha=(\frac{2r+1}{3}\delta+\alpha_{a\to i})+(\frac{r+2}{3}\delta+\alpha_{\overline{i+1}\to b})$.
The latter provides the costandard factorization of $\ell_{\max}$, in particular, $\sb[\ell_{\max}]\ne 0$.

$\circ$ Case 3: $3 \,|\, r$.
According to~\eqref{eq:pqm}, we have $p\geq \frac{r}{3}+1$. To get the lexicographically largest word,
we need to pick $p$ the smallest possible and then $\ell_1$ the maximal possible: $p=\frac{r}{3}+1$ and
$\ell_1=\overline{i+1}\dots b$. As $\ell_1$ is then larger than $\ell_2,\ell_3$ and $\ell$ is Lyndon, we
must have $q,m<p=\frac{r}{3}+1$. Evoking $p+q+m=r+1$, we thus get $q=m=\frac{r}{3}$. It is then
straightforward to see (using the induction hypothesis) that the only possible concatenation corresponds
to $\ell_2=i, \ell_3=\overline{i-1} \dots a$. Thus, we end up exactly with the word in the right-hand side
of~\eqref{eq:general-four} for $k=r+1$:
\begin{equation*}
  \ell_{\max} \, =
  \underbrace{\ell_i(\delta)}_{\frac{r+3}{3} \, \rtim} \overline{i+1} \dots b
  \underbrace{\ell_i(\delta)}_{\frac{r}{3} \, \rtim} i
  \underbrace{\ell_i(\delta)}_{\frac{r}{3} \, \rtim} \overline{i-1}\dots a.
\end{equation*}
This word arises from the decomposition
  $\alpha=(\frac{2r}{3}\delta+\alpha_{a\to i})+(\frac{r+3}{3}\delta+\alpha_{\overline{i+1}\to b})$.
The latter provides the costandard factorization of $\ell_{\max}$, in particular, $\sb[\ell_{\max}]\ne 0$.

\medskip
\noindent
$\bullet$ Proof of~\eqref{eq:general-five} for $k=r+1$.

The last root to consider is $\alpha_{b\to a}+(r+1)\delta$, where $1\in [b;a]\ne \wI$. First, let us
prove the aforementioned fact about the order of $\ell_1$, $\ell_{b \to a}(\delta)$, and $\ell_2$
(see~Remark~\ref{rem:order of 3 parts}(c)):
\begin{equation}\label{eq:auxiliary}
  \ell_{1}<\ell_{2}\leq \ell_{b \to a}(\delta).
\end{equation}
To prove this we need to look at the word $\SL(2\delta+\alpha_{b\to a})$. The first inequality is clear.
According to Claim \ref{bracketing with 1}, $\ell_2$ is either $\ell_*(\delta)$ or one of the words
$\SL(\alpha_{d \to a}), \SL(\alpha_{b \to c})$ with $d\in [\overline{a+2};\overline{b-1}], c\in [a;\overline{b-2}]$,
respectively. Let us consider these three cases:

\noindent
$\circ$
If $\ell_2=\ell_*(\delta)$, then one gets $\ell_{b \to a}(\delta)=\ell_2$ exactly as in our proof of~\eqref{eq:aux-gen-one}.

\noindent
$\circ$
If $\ell_2=\SL(\alpha_{d \to a})$, then in fact
  $\ell_1=\SL(\alpha_{b \to \overline{b-2}})<\ell_2=\SL(\alpha_{\overline{b-1}\to a})$,
due to Lemma~\ref{lemma root with 1}. Also
  $\SL(\alpha_{\overline{b-1}\to a})<
   \SL(\alpha_{\overline{b-1} \to \overline{b-3}})\, \overline{b-2}=\ell_{\overline{b-2}}(\delta)$
by Lemma~\ref{Sasha's proposition}.

1) If $i\in[2;\overline{b-2}]$, then $\sb[\ell_{\overline{b-2}}(\delta)]\doteq (E_{i,i} - E_{b-1,b-1})t$ by~\eqref{eq:el-word-bracketing},
which does not commute with $\sb[\ell_1]=\sb[\SL(\alpha_{b \to \overline{b-2}})]\doteq E_{b,b-1}t^{1-\delta_{b,1}}$.
Thus, the word $\ell_1\ell_{\overline{b-2}}(\delta)\ell_2$ is Lyndon and its bracketing is
  $\sb[\ell_1\ell_{\overline{b-2}}(\delta)\ell_2]=[\sb[\ell_1\ell_{\overline{b-2}}(\delta)],\sb[\ell_2]]=
   [[\sb[\ell_1],\sb[\ell_{\overline{b-2}}(\delta)]],\sb[\ell_2]] \doteq [\sb[\ell_1], \sb[\ell_2]]t\neq 0$.
Therefore, $\ell_2<\ell_{\overline{b-2}}(\delta)\leq \ell_{b \to a}(\delta)$.

2) If $i\in[\overline{b-1};n]$, then $\ell_2<\ell_{\overline{b-2}}(\delta)<\ell_{\overline{b-1}}(\delta)$ by Lemma \ref{deltaorder}.
Also $\sb[\ell_{\overline{b-1}}(\delta)]\doteq (E_{i+1,i+1} - E_{b-1,b-1})t$
by~\eqref{eq:el-word-bracketing}, which again does not commute with $\sb[\ell_1]\doteq E_{b,b-1}t^{1-\delta_{b,1}}$. Thus, the word
$\ell_1\ell_{\overline{b-1}}(\delta)\ell_2$ is Lyndon and moreover, arguing as in 1), we also get
$\sb[\ell_1\ell_{\overline{b-1}}(\delta)\ell_2]\ne 0$.
Therefore, $\ell_2<\ell_{\overline{b-1}}(\delta)\leq \ell_{b \to a}(\delta)$.

\noindent
$\circ$
If $\ell_2=\SL(\alpha_{b \to c})$, then in fact
  $\ell_1=\SL(\alpha_{\overline{a+2} \to a})<\ell_2=\SL(\alpha_{b\to \overline{a+1}})$,
due to Lemma~\ref{lemma root with 1}. Also
  $\SL(\alpha_{b\to \overline{a+1}})<
   \SL(\alpha_{\overline{a+3} \to \overline{a+1}})\, \overline{a+2}=\ell_{\overline{a+2}}(\delta)$
by Lemma~\ref{Sasha's proposition}.

1) If $i\in[\overline{a+2};n]$, then
  $\sb[\ell_{\overline{a+2}}(\delta)]\doteq (E_{i+1,i+1} - E_{a+2,a+2})t$ by~\eqref{eq:el-word-bracketing},
which does not commute with $\sb[\ell_1]=\sb[\SL(\alpha_{\overline{a+2} \to a})]\doteq E_{a+2,a+1}t$.
Thus, the word $\ell_1\ell_{\overline{a+2}}(\delta)\ell_2$ is Lyndon and its bracketing is
  $\sb[\ell_1\ell_{\overline{a+2}}(\delta)\ell_2]=[\sb[\ell_1\ell_{\overline{a+2}}(\delta)],\sb[\ell_2]]=
   [[\sb[\ell_1],\sb[\ell_{\overline{a+2}}(\delta)]],\sb[\ell_2]] \doteq [\sb[\ell_1], \sb[\ell_2]]t\neq 0$.
Therefore, $\ell_2<\ell_{\overline{a+2}}(\delta)\leq \ell_{b \to a}(\delta)$.

2) If $i\in[2;\overline{a+1}]$, then $\ell_{\overline{a+2}}(\delta)<\ell_{\overline{a+1}}(\delta)$ by Lemma \ref{deltaorder}
so that $\ell_2<\ell_{\overline{a+1}}(\delta)$. Note that $\sb[\ell_{\overline{a+1}}(\delta)]\doteq (E_{i,i} - E_{a+2,a+2})t$
by~\eqref{eq:el-word-bracketing}, which again does not commute with $\sb[\ell_1]\doteq E_{a+2,a+1}t$. Thus, the word $\ell_1\ell_{\overline{a+1}}(\delta)\ell_2$
is Lyndon and moreover, arguing as in 1), we also get $\sb[\ell_1\ell_{\overline{a+1}}(\delta)\ell_2]\ne 0$.
Therefore, $\ell_2<\ell_{\overline{a+1}}(\delta)\leq \ell_{b \to a}(\delta)$.

This completes our proof of~\eqref{eq:auxiliary}.

\medskip

We also note the following inequality:
\begin{equation}\label{eq:ba-k0}
  \SL(\alpha_{b\to a})\leq \ell_1 < \SL(\delta+\alpha_{b\to a})=\ell_1\ell_2.
\end{equation}
According to Lemma~\ref{lemma root with 1}, $\ell_1$ is either $\SL(\alpha_{b \to \overline{b-2}})$ or
$\SL(\alpha_{\overline{a+2} \to a})$. Evoking Lemma~\ref{Sasha's proposition}, we thus get
$\SL(\alpha_{b\to a})\leq \ell_1 < \ell_1\ell_2$ in both cases, as claimed in~\eqref{eq:ba-k0}.

\medskip

To prove our \underline{key} Lemma~\ref{lemma root with 1} below, we need an explicit algorithm for computing
the words $\SL(\alpha_{b \to a})$. This is essentially a description of Lalonde-Ram's bijection
\eqref{eqn:associated word} for a finite type $A$, generalizing our former Claim~\ref{claim:LR-A-endpoint}
to the case when the minimal letter on the arch $[b;a]$ is not $b$ or $a$, and it utilizes the
argument from our proof of~(\ref{eq:el-word-explicit-1},~\ref{eq:el-word-explicit-2}).
We provide two algorithms: building $\SL(\alpha_{b\to a})$ either from right to left or from left to right
by stacking ``segmental'' words accordingly.

\medskip
\noindent
\underline{Right-to-Left Algorithm for $\SL(\alpha_{b\to a})$ with $1 \in [b;a]$.}

This algorithm (which crucially uses the fact that each letter appears at most once) reads off the word
$\SL(\alpha_{b\to a})$ from right to left, stacking ``segmental'' words accordingly. First, we note that
$1$ will be the first letter. Then, we choose the second smallest letter $1\ne c\in [b;a]$.
If $c \in [2;a]$, then we place the word $u_1:=c\,\overline{c+1} \dots a$ in the very end of $\SL(\alpha_{b\to a})$,
while for $c \in [b; 0]$ we place the word $u_1:=c\, \overline{c-1}\dots b$ in the very end of $\SL(\alpha_{b\to a})$.
Next, we apply the same algorithm to the arch $[b;c-1]$ or $[\overline{c+1};a]$, respectively. In other words, we take
the second smallest letter among the remaining ones, and place the resulting word $u_2$ right before $u_1$, and so on.

\medskip
\noindent
\underline{Left-to-Right Algorithm for $\SL(\alpha_{b\to a})$ with $1 \in [b;a]$.}

Since the lexicographical order compares words from left to right, it is convenient to restate the above algorithm by
rather building $\SL(\alpha_{b\to a})$ from left to right. The first letter is clearly $1$, while the second letter
is the $\max\{0,2\}$. If it is $0$, then either $n \notin [b; a]$ in which case we just place the segment $23\dots a$
after $0$, or $n\in [b; a]$ and we compare $n$ and $2$, do the same operation, and proceed further. Let us rephrase
the above algorithm. Pick the largest letter among $2$ and $0$ and add after $1$ the longest Lyndon segment
$23\dots c$ with $c\in [2; a]$ (if $2>0$) or $0n\dots d$ with $d\in [b; 0]$ (if $2<0$). Then,
compare $\overline{c+1}$ with $0$ or $\overline{d-1}$ with $2$ accordingly, and so on.
This reconstructs $\SL(\alpha_{b\to a})$ by stacking ``segmental'' words from left to right after $1$.

\medskip

Let us now describe the costandard factorization of $\SL(\delta+\alpha_{b\to a})$ with $1\in [b;a]$.

\begin{lemma}\label{lemma root with 1}
Let $\SL(\delta+\alpha_{b \to a})=\ell_1\ell_2$ be the costandard factorization, $1\in [b;a]$.

\noindent
(a) If $\SL(\alpha_{\overline{b-1}\to a})>\SL(\alpha_{b\to \overline{a+1}})$, then:
$\ell_1=\SL(\alpha_{b\to \overline{b-2}})$, $\ell_2=\SL(\alpha_{\overline{b-1}\to a})$.

\noindent
(b) If $\SL(\alpha_{\overline{b-1}\to a})<\SL(\alpha_{b\to \overline{a+1}})$, then:
$\ell_1=\SL(\alpha_{ \overline{a+2}\to a})$, $\ell_2=\SL(\alpha_{b\to \overline{a+1}})$.
\end{lemma}

\begin{remark}\label{rem:explanation}
For $a=\overline{b-2}$ (equivalently, $b=\overline{a+2}$), we get $\ell_1=\SL(\alpha_{b\to \overline{b-2}})$
while the above formulas for $\ell_2$ should be understood as follows:
  $$\ell_2=\ell_{b\to \overline{b-2}}(\delta)=\ell_{b-1+sgn(i-(b-1))}(\delta).$$
\end{remark}

\begin{proof}[Proof of Lemma~\ref{lemma root with 1}]
For $a=\overline{b-2}$, the above formulas (cf.~Remark~\ref{rem:explanation}) are obvious,
since according to Claim~\ref{bracketing with 1} there is only one decomposition to consider:
\begin{equation*}
  \alpha_{b \to \overline{b-2}}+\delta=(\alpha_{b \to \overline{b-2}})+(\delta) \,,
\end{equation*}
and
  $\SL(\alpha_{b \to \overline{b-2}})<\ell_{\overline{b-1}}(\delta)\leq \ell_{b-1+sgn(i-(b-1))}(\delta)$,
cf.\ Lemma~\ref{deltaorder}, Remark~\ref{rem:order of 3 parts}(b).

If $a\ne \overline{b-2}$ and $\SL(\alpha_{\overline{b-1}\to a})>\SL(\alpha_{b\to \overline{a+1}})$,
then we claim that:
\begin{equation}\label{eq:comb-alg}
  \SL(\alpha_{\overline{b-1}\to a})>\SL(\alpha_{b\to \overline{b-2}}) \,.
\end{equation}
Indeed, let us construct all three $\SL$-words
  $\SL(\alpha_{\overline{b-1}\to a})$, $\SL(\alpha_{b\to \overline{a+1}})$, $\SL(\alpha_{b\to \overline{b-2}})$
using the above ``Left-to-Right Algorithm''. Then, $\SL(\alpha_{\overline{b-1}\to a})>\SL(\alpha_{b\to \overline{a+1}})$
implies that at the leftmost spot where these words differ either the former has $\overline{b-1}$ while the latter has
some $c<\overline{b-1}$ or the latter has $\overline{a+1}$ while the former has some $c>\overline{a+1}$.
In either of these cases, we clearly have $\SL(\alpha_{\overline{b-1}\to a})>\SL(\alpha_{b\to \overline{b-2}})$.

According to~\eqref{eq:comb-alg} and Lemma~\ref{lemma:lyndon}, the word
$\SL(\alpha_{b\to \overline{b-2}})\SL(\alpha_{\overline{b-1}\to a})$ is Lyndon. Its costandard
factorization~\eqref{eqn:cost.factor} is precisely given by $\ell_1=\SL(\alpha_{b\to \overline{b-2}})$
and $\ell_2=\SL(\alpha_{\overline{b-1}\to a})$, since both words start with $1$ (and have no more $1$'s).
Hence, the standard bracketing
  $\sb[\SL(\alpha_{b\to \overline{b-2}})\SL(\alpha_{\overline{b-1}\to a})]=
   [\sb[\SL(\alpha_{b\to \overline{b-2}})],\sb[\SL(\alpha_{\overline{b-1}\to a})]]\ne 0$.
We thus conclude that
  $\SL(\delta+\alpha_{b \to a})\geq \SL(\alpha_{b\to \overline{b-2}})\SL(\alpha_{\overline{b-1}\to a})$.
We also note that combining~\eqref{eq:comb-alg} with Lemma~\ref{Sasha's proposition}, we obtain:
\begin{equation}\label{eq:key-com-inequalities}
  \SL(\alpha_{b\to c}) \leq \SL(\alpha_{b\to \overline{b-2}}) < \SL(\alpha_{\overline{b-1}\to a})
  \qquad \forall\, c \in [a;\overline{b-2}] \,.
\end{equation}
Combining Claim~\ref{bracketing with 1} with \eqref{eq:key-com-inequalities}, we get
  $\SL(\delta+\alpha_{b \to a})\leq \SL(\alpha_{b\to \overline{b-2}})\SL(\alpha_{\overline{b-1}\to a})$.
Therefore, we actually have the equality
  $$\SL(\delta+\alpha_{b \to a})=\SL(\alpha_{b\to \overline{b-2}})\SL(\alpha_{\overline{b-1}\to a})$$
and the two words in the right-hand side determine the costandard factorization of $\SL(\delta+\alpha_{b \to a})$,
as shown above. This completes our proof of part (a).

The proof of part~(b) is analogous and is left to the interested reader.
\end{proof}

\begin{corollary}\label{cor:ell1-invariant}
In the setup of Lemma~\ref{lemma root with 1}, we have:
\begin{equation}\label{eq:ell1-inv}
  \ell_1=\min \big\{ \SL(\alpha_{b\to \overline{b-2}}), \SL(\alpha_{ \overline{a+2}\to a}) \big\} \,.
\end{equation}
\end{corollary}

\begin{proof}
For $a=\overline{b-2}$, the claim is vacuous by Lemma~\ref{lemma root with 1}.
If $\SL(\alpha_{\overline{b-1}\to a})>\SL(\alpha_{b\to \overline{a+1}})$, then
$\ell_1=\SL(\alpha_{b\to \overline{b-2}})<\SL(\alpha_{\overline{b-1}\to a})$ by~\eqref{eq:comb-alg} and
Lemma~\ref{lemma root with 1}. But $\SL(\alpha_{\overline{b-1}\to a})\leq \SL(\alpha_{\overline{a+2}\to a})$
by Lemma~\ref{Sasha's proposition} as $1\in [\overline{b-1};a]\subseteq [\overline{a+2};a]$ for
$a\prec \overline{b-2}$. Combining the above, we obtain:
$\ell_1=\SL(\alpha_{b\to \overline{b-2}})<\SL(\alpha_{\overline{a+2}\to a})$.

The case $\SL(\alpha_{\overline{b-1}\to a})<\SL(\alpha_{b\to \overline{a+1}})$ is completely analogous.
\end{proof}

With the inequalities~(\ref{eq:auxiliary},~\ref{eq:ba-k0}) and Lemma~\ref{lemma root with 1} at hand, we
shall finally proceed to the proof of~\eqref{eq:general-five} for $k=r+1$. To this end, we consider all possible
decompositions of $\alpha=(r+1)\delta+\alpha_{b\to a}$ with $1\in [b;a]$ case-by-case:

\medskip
\noindent
1) $\alpha=(r_1\delta + \alpha_{b\to c}) + ((r+1-r_1)\delta+\alpha_{\overline{c+1}\to a})$, with $c\in [b \to a)$.

Let us assume that $1\in [b;c]$ (the case $1\in [\overline{c+1};a]$ is analogous).
The corresponding concatenation $\ell$ is
  $\leq \ell'_1 \underbrace{\ell_{b\to c}(\delta)}_{(r_1-1) \, \rtim} \ell'_2\, \SL((r+1-r_1)\delta+\alpha_{\overline{c+1}\to a})$
if $r_1>0$, or
  $\leq \SL(\alpha_{b \to c})\SL((r+1)\delta+\alpha_{\overline{c+1}\to a})$
if $r_1=0$. Here, $\SL(\delta+\alpha_{b\to c})=\ell'_1\ell'_2$ is the costandard factorization.
According to~(\ref{eq:auxiliary},~\ref{eq:ba-k0}), we have:
$\SL(\alpha_{b \to c})\leq \ell'_1<\ell'_2\leq \ell_{b\to c}(\delta)$, where both equalities hold
iff either of them holds. As $c\in [a\to b)$ and $b\ne \overline{a-1}$, we have
$\SL(\alpha_{b \to c})\ne \ell'_1$, due to Lemma~\ref{lemma root with 1}.
Thus $\SL(\alpha_{b \to c})<\ell'_1$, so that $\SL(k_1\delta+\alpha_{b \to c})<\SL(k_2\delta+\alpha_{b \to c})$, hence
$\SL(k_1\delta+\alpha_{b \to c})1<\SL(k_2\delta+\alpha_{b \to c})$ and the former is not a prefix of the latter
for any $0\leq k_1<k_2$. Therefore,
  $\ell \leq \ell'_1 \underbrace{\ell_{b\to c}(\delta)}_{r \, \rtim} \ell'_2\, \SL(\alpha_{\overline{c+1}\to a})=
   \SL((r+1)\delta + \alpha_{b\to c}))\SL(\alpha_{\overline{c+1}\to a})$.
By Lemma~\ref{lemma root with 1}, $\ell'_2$ is either $\SL(\alpha_{\overline{b-1}\to c})$
or $\SL(\alpha_{b\to \overline{c+1}})$. We consider these cases:

$\circ$
If $\ell'_2=\SL(\alpha_{\overline{b-1}\to c})$ and $a\ne \overline{b-2}$, then
$\ell'_2\, \SL(\alpha_{\overline{c+1}\to a})\leq \SL(\alpha_{\overline{b-1}\to a})$ by Proposition~\ref{prop:generalized Leclerc}.
Moreover, by Lemma~\ref{lemma root with 1} and its proof, we also have
$\SL(\alpha_{\overline{b-1}\to c})>\SL(\alpha_{b\to \overline{c+1}})$ and
$\SL(\alpha_{\overline{b-1}\to c})>\SL(\alpha_{b\to \overline{b-2}})=\ell'_1$.
We thus obtain a sequence of inequalities:
  $\SL(\alpha_{\overline{b-1}\to a}) > \SL(\alpha_{\overline{b-1}\to c}) >
   \SL(\alpha_{b\to \overline{b-2}})\geq \SL(\alpha_{b\to \overline{a+1}})$.
Hence, applying Lemma~\ref{lemma root with 1} once again to $\SL(\delta+\alpha_{b\to a})$,
we see that its costandard factorization has prefix $\ell_1=\ell'_1$, suffix $\ell_2=\SL(\alpha_{\overline{b-1}\to a})$,
and therefore $\ell_{b\to a}(\delta)=\ell_{b\to c}(\delta)$. Thus, we derive the desired inequality:
\begin{equation*}
  \ell\leq
  \ell'_1 \underbrace{\ell_{b\to c}(\delta)}_{r \, \rtim}  \SL(\alpha_{\overline{b-1}\to c})\SL(\alpha_{\overline{c+1}\to a}) \leq
  \ell'_1 \underbrace{\ell_{b\to c}(\delta)}_{r \, \rtim}  \SL(\alpha_{\overline{b-1}\to a}) =
  \ell_1 \underbrace{\ell_{b\to a}(\delta)}_{r \, \rtim}  \ell_2 \,.
\end{equation*}
Moreover, the equality is possible for $r_1=r+1$ and a specific $c\in [b\to a)$ such that
$\SL(\alpha_{\overline{b-1} \to a}) = \SL(\alpha_{\overline{b-1}\to c})\SL(\alpha_{\overline{c+1}\to a})$ is the costandard factorization.

Let us now consider the case $\ell'_2=\SL(\alpha_{\overline{b-1}\to c})$ and $a=\overline{b-2}$.
If $\SL(\alpha_{\overline{b-1}\to c})\SL(\alpha_{\overline{c+1}\to a})\leq \ell_{b\to c}(\delta)$, then
$\ell\leq \ell_1\underbrace{\ell_{b\to a}(\delta)}_{r+1 \, \rtim}=\ell_1\underbrace{\ell_{b\to a}(\delta)}_{r \, \rtim}\ell_2$ still holds.
Otherwise, if $\SL(\alpha)=\ell$, we would have
  $\SL(\alpha)=\SL((r+1)\delta+\alpha_{b\to c})\SL(\alpha_{\overline{c+1}\to a})=\ell'_1 \underbrace{\ell_{b\to c}(\delta)}_{r \, \rtim}  \SL(\alpha_{\overline{b-1}\to c})\SL(\alpha_{\overline{c+1}\to a})$,
due to Proposition~\ref{prop:generalized Leclerc}(a). However, the costandard factorization of the above word passes to the right of $\ell'_1$, and
the costandard factorization of the resulting suffix passes to the right of the first $\ell_{b\to c}(\delta)$, a contradiction with Remark~\ref{rem:gen Lyndon rmk}.
Hence, $\ell$ cannot be standard Lyndon.

$\circ$
If $\ell'_2=\SL(\alpha_{b\to \overline{c+1}})$, then $\deg\, \ell'_2+\alpha_{\overline{c+1}\to a}\notin \wDelta^+$ and so
$\sb[\ell'_2\, \SL(\alpha_{\overline{c+1}\to a})]=0$. Likewise, by the degree reasons and evoking inequalities~\eqref{eq:auxiliary}, we find:
\begin{equation*}
  \sb[\SL((r+1)\delta + \alpha_{b\to c}))\SL(\alpha_{\overline{c+1}\to a})]=
  \sb[ \ell'_1 \underbrace{\ell_{b\to c}(\delta)}_{r \, \rtim} \ell'_2\, \SL(\alpha_{\overline{c+1}\to a})]=0
\end{equation*}
as the costandard factorization of this concatenation passes on the left of $\ell'_2$ or some $\ell_{b\to c}(\delta)$.
Since $\ell\leq \SL((r+1)\delta + \alpha_{b\to c}))\SL(\alpha_{\overline{c+1}\to a})$, we see that if
$\SL(\alpha)=\ell$, then we would have $\SL(\alpha)=\SL((r+1)\delta + \alpha_{b\to c}))\SL(\alpha_{\overline{c+1}\to a})$,
due to Proposition~\ref{prop:generalized Leclerc}(a). However, the rightmost word cannot be standard Lyndon as its standard
bracketing was shown above to be $0$. Hence, a contradiction with $\SL(\alpha)=\ell$.

\medskip
\noindent
2) $\alpha=(r_1\delta+\alpha_{b\to c}) + ((r-r_1)\delta+\alpha_{\overline{c+1}\to a})$, where $1\in [b;c]$ and $1\in [c+1;a]$.

Let $\SL(\delta+\alpha_{b\to a})=\ell_1\ell_2$ be the costandard factorization.
We claim that one of length $n$ prefixes of the words $\SL(\delta+\alpha_{b\to c})$, $\SL(\delta+\alpha_{\overline{c+1}\to a})$
is $\leq \ell_1$. Indeed, assume that $\ell_1=\SL(\alpha_{b\to \overline{b-2}})$ (the case $\ell_1=\SL(\alpha_{\overline{a+2}\to a})$
is treated similarly). Then, the length $n$ prefix $\ell'_1$ of $\SL(\delta+\alpha_{b\to c})$ is $\ell_1$ or $\SL(\alpha_{\overline{c+2}\to c})$,
due to Lemma~\ref{lemma root with 1}. Note that $\SL(\alpha_{\overline{c+2}\to c})=\ell_1$ if $c=\overline{b-2}$, while the inequality
$\SL(\alpha_{\overline{c+2}\to c})<\SL(\alpha_{b\to \overline{c+1}})$ for $c\ne \overline{b-2}$ is proved similarly to~\eqref{eq:comb-alg}.
Combining the latter inequality with $\SL(\alpha_{b\to \overline{c+1}})\leq \SL(\alpha_{b\to \overline{b-2}})=\ell_1$
due to Lemma~\ref{Sasha's proposition}, we obtain $\ell'_1<\ell_1$ as claimed.
Henceforth, we assume that $\ell_1=\SL(\alpha_{b\to \overline{b-2}})$, leaving the other case to the reader.

First consider $0<r_1<r$. If the length $n$ prefix $\ell'_1$ of $\SL(\delta+\alpha_{b\to c})$ is $<\ell_1$, then the corresponding concatenation
$\ell$ is lexicographically smaller than the right-hand side of~\eqref{eq:general-five} for $k=r+1$. If $\ell'_1=\ell_1$, then we get a costandard
factorization $\SL(\delta+\alpha_{b\to c})=\ell_1\ell_3$ and so $\ell_{b\to c}(\delta)=\ell_{b\to a}(\delta)$.
If $c\ne \overline{b-2}$, then $\ell_3< \ell_{b\to c}(\delta)$ by~\eqref{eq:auxiliary} and so $\ell_3\, 1\leq \ell_{b\to c}(\delta)=\ell_{b\to a}(\delta)$.
Thus, we get the desired inequality:
\begin{equation*}
  \ell\leq
  \SL(r_1\delta+\alpha_{b\to c})\SL((r-r_1)\delta+\alpha_{(c+1)\to a}) <
  \ell_1\underbrace{\ell_{b \to a}(\delta)}_{r \; \mathrm{times}} \ell_2 \,.
\end{equation*}
If $c=\overline{b-2}$, then
  $\ell_3=\ell_{b\to \overline{b-2}}(\delta)=\ell_{b-1+sgn(i-(b-1))}(\delta)\geq \ell_{\overline{b-2}}(\delta)$,
with the last inequality due to Lemma~\ref{deltaorder}.
Let $\SL(\delta+\alpha_{\overline{b-1}\to a})=\ell_4 \ell_5$ be the costandard factorization. Then,
  $\ell_4\leq \SL(\alpha_{\overline{b-1}\to \overline{b-3}})<
   \SL(\alpha_{\overline{b-1}\to \overline{b-3}})\,\overline{b-2}=\ell_{\overline{b-2}}(\delta)$,
due to~\eqref{eq:ell1-inv}. Hence, the corresponding concatenation $\ell$ satisfies the desired inequality:
\begin{equation*}
  \ell \leq
  \ell_1 \underbrace{\ell_{b \to \overline{b-2}}(\delta)}_{r_1\, \rtim} \ell_4
  \underbrace{\ell_{ \overline{b-1}\to a}(\delta)}_{(r-r_1-1) \, \rtim}\ell_5<
  \ell_1\underbrace{\ell_{b \to a}(\delta)}_{r \, \rtim}\ell_2 \,.
\end{equation*}

For $r_1=r$, it suffices to consider the case when $\SL(\delta+\alpha_{b\to c})$ starts with $\ell_1$.
The case $c\ne \overline{b-2}$ is treated as above. If $c=\overline{b-2}$,
then $\ell_2=\SL(\alpha_{\overline{c+1}\to a})$ and $\ell_3=\ell_{b \to c}(\delta)=\ell_{b \to a}(\delta)$,
and thus the resulting concatenation $\ell\leq \ell_1\underbrace{\ell_{b \to a}(\delta)}_{r \; \mathrm{times}}\ell_2$.

Finally, if $r_1=0$ and $\SL(\delta+\alpha_{\overline{c+1}\to a})=\ell_4\ell_5$ is the costandard factorization, then
$\SL(\alpha_{b\to c})\leq \ell'_1\leq \ell_1$. If $c\neq \overline{b-2}$, then $\SL(\alpha_{b\to c})<\ell_1$, so that $\SL(\alpha_{b\to c})1<\ell_1$
and the former is not a prefix of the latter, implying $\ell<\ell_1$. If $c=\overline{b-2}$, then $\SL(\alpha_{b\to c})=\ell_1$ and
$\ell_4\leq \SL(\alpha_{\overline{b-1}\to \overline{b-3}})<\ell_{\overline{b-2}}(\delta)\leq \ell_{b\to a}(\delta)$ by above, so that:
\begin{equation*}
  \ell\leq \SL(\alpha_{b\to c})\SL(r\delta+\alpha_{\overline{c+1}\to a})\leq
  \ell_1 \ell_4 \underbrace{\ell_{ \overline{c+1}\to a}(\delta)}_{(r-1) \, \rtim}\ell_5<
  \ell_1 \ell_{b \to a}(\delta) < \ell_1 \underbrace{\ell_{b\to a}(\delta)}_{r \, \rtim} \ell_2 \,.
\end{equation*}

\medskip
\noindent
3) $\alpha=(r_1\delta) + ((r+1-r_1)\delta+\alpha_{b \to a})$.

If $a\ne \overline{b-2}$, then (using the induction hypothesis) the corresponding concatenated word $\ell$ is
  $\leq \ell_1 \underbrace{\ell_{b\to a}(\delta)}_{(r-r_1) \, \rtim} \ell_2\,
   \SL(\alpha_{\overline{c+1}\to \overline{c-1}})\underbrace{\ell_{c+sgn(i-c)}(\delta)}_{(r_1-1) \, \rtim}c$
if $r_1\leq r$, or
  $\leq \SL(\alpha_{b\to a})\SL(\alpha_{\overline{c+1}\to \overline{c-1}})\underbrace{\ell_{c+sgn(i-c)}(\delta)}_{r \, \rtim}c$
if $r_1=r+1$, for some $c\ne 1$.
Due to the inequalities $\SL(\alpha_{b\to a})<\ell_1<\ell_2<\ell_{b\to a}(\delta)$,
cf.~(\ref{eq:auxiliary},~\ref{eq:ba-k0}), we obtain ($\forall\, c\ne 1$):
\begin{equation*}
  \ell\leq
  \ell_1 \underbrace{\ell_{b\to a}(\delta)}_{(r-1) \, \rtim} \ell_2\, \SL(\alpha_{\overline{c+1}\to \overline{c-1}})c <
  \ell_1 \underbrace{\ell_{b\to a}(\delta)}_{r \, \rtim} \ell_2 \,.
\end{equation*}

Let us now treat the case $a=\overline{b-2}$, for which we utilize the non-commutativity of the corresponding bracketings.
We consider the cases $r_1=1$ and $r_1>1$ separately.

If $r_1=1$, then the corresponding concatenation $\ell$ is
  $\leq \ell_1 \underbrace{\ell_{b\to a}(\delta)}_{r \, \rtim} \ell_c(\delta)$,
where $\ell_2=\ell_{b\to a}(\delta)=\ell_{b-1+sgn(i-(b-1))}(\delta)$ by Remark~\ref{rem:explanation}.
Here, $\sb[\ell_c(\delta)]$ does not commute with $\sb[\SL(r\delta+\alpha_{b \to \overline{b-2}})]$,
which is equivalent to $[\sb[\ell_c(\delta)],E_{b,b-1}]\ne 0$. The latter guarantees that
$\ell_c(\delta)\leq \ell_{b\to a}(\delta)$, due to~\eqref{eq:el-word-bracketing} and Lemma~\ref{deltaorder}:

$\circ$
if $b \prec i$ then $c=b-1,b$ and $\ell_{c}(\delta)\leq \ell_b(\delta)=\ell_{b-1+sgn(i-(b-1))}(\delta)$;

$\circ$
if $b=i,i+1,i+2$, then $\ell_{b-1+sgn(i-(b-1))}(\delta)=\ell_i(\delta)\geq \ell_c(\delta)$;

$\circ$
if $b\succ i+2$, then $c=b-1,b-2$ and $\ell_{c}(\delta)\leq \ell_{b-2}(\delta)=\ell_{b-1+sgn(i-(b-1))}(\delta)$.
Hence, we derive the desired inequality:
\begin{equation*}
  \ell\leq
  \ell_1 \underbrace{\ell_{b\to a}(\delta)}_{r \, \rtim} \ell_c(\delta) \leq
  \ell_1 \underbrace{\ell_{b\to a}(\delta)}_{(r+1) \, \rtim}=
  \ell_1 \underbrace{\ell_{b\to a}(\delta)}_{r \, \rtim} \ell_2 \,.
\end{equation*}

For $r_1>1$, the argument is precisely the same and is based on the inequalities
$\SL(\alpha_{\overline{c+1}\to \overline{c-1}}) < \ell_c(\delta) \leq \ell_{b\to a}(\delta)$.
Here, the second inequality is proved as above, but using~\eqref{r+1 imaginary} instead of~\eqref{eq:el-word-bracketing}.

This completes the proof of  \eqref{eq:general-five}. In the particular case $r=1$,
this proves the formula $\SL(2\delta+\alpha_{b \to a}) = \ell_1 \ell_{b \to a}(\delta) \ell_2$
implicitly used in the statement of~\eqref{eq:general-five}.
\end{proof}


\medskip

\section{Properties of orders}\label{sec:orders}

To account for $\dim(\widehat{\fg}_{k\delta})=|I|$ in~\eqref{eq:aff-dim}, let us extend $\widehat{\Delta}^+$ to $\widehat{\Delta}^{+,\mathrm{ext}}$:
\begin{equation}\label{eq:extended-affine-roots}
  \widehat{\Delta}^{+,\mathrm{ext}}:=
  \widehat{\Delta}^{+,\mathrm{re}} \sqcup \big\{(k\delta,r) \, \big|\, k\geq 1, 1\leq r\leq |I|\big\}.
\end{equation}
We define $\SL((k\delta,r)):=\SL_r(k\delta)$ accordingly. Consider the order on $\widehat{\Delta}^{+,\mathrm{ext}}$
induced from the lexicographical order on $\aslaws$, cf.~\eqref{eqn:induces}:
\begin{equation}\label{eqn:affine-induces}
  \alpha < \beta \quad \Longleftrightarrow \quad \SL(\alpha) < \SL(\beta) \ \text{ lexicographically}.
\end{equation}
In this section, we investigate some properties of this order using Theorem~\ref{thm:sln-general}.

\begin{example}
The only case when $\widehat{\Delta}^{+,\mathrm{ext}}=\wDelta^+$ is the case of $\widehat{\fsl}_2$.
Using the formulas of Proposition~\ref{prop:sl2-case} (with the order $1<0$), we see that~\eqref{eqn:affine-induces} recovers
the usual order:
\begin{equation*}
  \alpha_1 < \alpha_1+\delta < \alpha_1 + 2\delta < \cdots < \cdots < 3\delta < 2\delta < \delta <
  \cdots < 2\delta+\alpha_0 < \delta+\alpha_0 < \alpha_0 \,.
\end{equation*}
\end{example}


\subsection{Important counterexample}\label{ssec:counterexample}
\

Unlike the orders on $\wDelta^{+,\mathrm{ext}}$ in the theory of affine quantum groups (\cite{B,KT}),
arising through the affine braid group action, the order~\eqref{eqn:affine-induces} does separate
imaginary roots. Explicitly, for type $A^{(1)}_n\ (n>1)$ and any order on $\wI$, one always~has:
\begin{equation*}
  (k_1\delta,n) < \alpha < (k_2\delta,1)
  \quad \mathrm{for\ some}\ \alpha\in \wDelta^{+,\mathrm{re}} \,,\,  k_1,k_2\geq 1.
\end{equation*}
It is thus natural to ask (motivated by Levendorsky-Soibelman convexity property):

\medskip
\noindent
\textbf{Question:} Is it true that we cannot have a pattern
\begin{equation*}
  (k_2\delta,n)<\beta_2<\beta_1<(k_1\delta,1) \quad \mathrm{with} \quad
  \beta_1,\beta_2\in \wDelta^{+,\mathrm{re}} \,,\, \beta_1+\beta_2=(k_1+k_2)\delta.
\end{equation*}

\noindent
The answer is actually negative, as shown by the following simplest counterexample.

\medskip
\noindent
\textbf{Counterexample:}
Consider the affine Lie algebra $\widehat{\fsl}_5$ with the standard order $1<2<3<4<0$ on $\wI$.
For $k,m>0$, set $\beta_1=k\delta+\alpha_4, \beta_2=m\delta+(\alpha_0+\alpha_1+\alpha_2+\alpha_3)$
and $k_1=1,k_2=k+m$. According to Theorem~\ref{thm:sln-standard}, we have:
\begin{equation*}
\begin{split}
  & \SL_1(\delta)=10432 \,,\qquad  \SL_4((k+m)\delta)=1234 \underbrace{10234}_{(k+m-1) \, \rtim} 0, \\
  & \SL(\beta_1)= \underbrace{10423}_{k \, \rtim} 4 \,,\qquad \SL(\beta_2)= 1023 \underbrace{10423}_{m \, \rtim}.
\end{split}
\end{equation*}
Thus, indeed $(k_2\delta,4)<\beta_2<\beta_1<(\delta,1)$ with respect to the order~\eqref{eqn:affine-induces}
on $\widehat{\Delta}^{+,\mathrm{ext}}$.


\subsection{Chain monotonicity in type $A^{(1)}_n$}\label{ssec:monotonicity}
\

For $\alpha\in \wDelta^{+,\mathrm{re}}$, define the \emph{chain} $\mathrm{Ch}_\alpha$ as
the sequence $\alpha,\alpha+\delta,\alpha+2\delta,\ldots \in \wDelta^{+,\mathrm{re}}$.

\begin{proposition}\label{prop:order-prop-1}
For any $\alpha\in \wDelta^{+,\mathrm{re}}$, the chain $\mathrm{Ch}_\alpha$ is monotonous:
\begin{equation*}
  \SL(\alpha) < \SL(\alpha+\delta) < \SL(\alpha+2\delta) < \cdots
  \quad \mathrm{or} \quad
  \SL(\alpha) > \SL(\alpha+\delta) > \SL(\alpha+2\delta) > \cdots
\end{equation*}
\end{proposition}

\begin{proof}
Without loss of generality, we can assume that~\eqref{eq:minimal letters} holds, so that the formulas
of Theorem~\ref{thm:sln-general} apply. The proof follows by a simple case-by-case analysis:

\medskip
\noindent
$\bullet$ $\alpha=\alpha_{a\to b}$ with $i \prec a \preceq b \preceq 0$.

According to~\eqref{eq:general-two}, we have
  $\SL(k\delta+\alpha_{a\to b}) = \underbrace{\ell_{\overline{a-1}}(\delta)}_{k \, \rtim} a\, \overline{a+1} \dots b$
for all $k\geq 1$. As $a\, \overline{a+1} \dots b$ starts with a letter $a$ which is
larger than $1$, the first letter of $\ell_{\overline{a-1}}(\delta)$, we obtain
$\SL(k\delta+\alpha_{a\to b})>\SL((k+1)\delta+\alpha_{a\to b})$ for any $k\geq 1$.
In the remaining case $k=0$, we also have $\SL(\alpha_{a\to b})>\SL(\delta+\alpha_{a\to b})$,
as $\SL(\alpha_{a\to b})$ starts with a letter $\min\{a,\ldots, b\}$ which is larger than $1$,
the first letter of $\SL(\delta+\alpha_{a\to b})$.

\medskip
\noindent
$\bullet$ $\alpha=\alpha_{a\to b}$ with $1 \prec a \preceq b \prec i$.

The proof of $\SL(k\delta+\alpha_{a\to b})>\SL((k+1)\delta+\alpha_{a\to b})$ for any $k\geq 0$ is
exactly the same as above, with $\ell_{b+1}(\delta)$ used instead of $\ell_{\overline{a-1}}(\delta)$.

\medskip
\noindent
$\bullet$ $\alpha=\alpha_{a\to b}$ with $1 \prec a \prec i \prec b$.

Combining formula~\eqref{eq:general-four} with the inequalities
$\overline{i\pm 1}>i>1=$ first letter of $\ell_i(\delta)$, we obtain
$\SL(k\delta+\alpha_{a\to b})>\SL((k+1)\delta+\alpha_{a\to b})$ for any $k\geq 1$.
In the remaining case $k=0$, we also have $\SL(\alpha_{a\to b})>\SL(\delta+\alpha_{a\to b})$,
as $1\notin [a;b]$.

\medskip
\noindent
$\bullet$ $\alpha=\alpha_{a\to b}$ with $a=i$ or $b=i$ and $1\notin [a;b]$.

The proof of $\SL(k\delta+\alpha_{a\to b})>\SL((k+1)\delta+\alpha_{a\to b})$ for any $k\geq 0$
is exactly the same as above, where we now use one of~\eqref{genreal four 2}--\eqref{genreal four 4}
instead of~\eqref{eq:general-four}.

\medskip
\noindent
$\bullet$ $\alpha=\alpha_{b\to a}$ with $1\in [b;a]$.

According to~\eqref{eq:general-five}, we have
  $\SL(k\delta+\alpha_{b\to a})=\ell_1 \underbrace{\ell_{b\to a}(\delta)}_{(k-1) \, \rtim} \ell_2$
for all $k\geq 1$. Here, we have $\ell_2\leq \ell_{b\to a}(\delta)$, due to~\eqref{eq:auxiliary},
so that $\ell_2<\ell_{b\to a}(\delta)\ell_2$. Thus, we obtain
$\SL(k\delta+\alpha_{b\to a})<\SL((k+1)\delta+\alpha_{b\to a})$ for any $k\geq 1$.
In the remaining case $k=0$, we also have $\SL(\alpha_{b\to a})<\SL(\delta+\alpha_{b\to a})$,
due to~\eqref{eq:ba-k0}.
\end{proof}

\begin{remark}\label{rem:monot-increas}
It follows from the proof that the chain $\mathrm{Ch}_\alpha$ monotonously increases if
$\alpha=k\delta+\alpha_{a\to b}$ with $\min\{\wI\}\in [a;b]$, and monotonously decreases~otherwise.
\end{remark}

\begin{remark}
For any $k\geq 1$ and $c\ne 1$, we also have
  $\SL(\alpha_{\overline{c+1}\to \overline{c-1}})\underbrace{\ell_{c+sgn(i-c)}(\delta)}_{(k-1) \, \rtim}c>
   \SL(\alpha_{\overline{c+1}\to \overline{c-1}})\underbrace{\ell_{c+sgn(i-c)}(\delta)}_{k \, \rtim}c$,
cf.~\eqref{eq:general-one}. Since the order among length $n$ words
$\{\SL(\alpha_{\overline{c+1}\to \overline{c-1}}) \,|\, c\ne 1\}$ determines the order among the $n$ words
in the right-hand side of~\eqref{eq:general-one} for any $k$, we also see that $\{\SL(k\delta,r)\}_{k\geq 1}$
monotonously decreases:
\begin{equation*}
  \SL(\delta,r)>\SL(2\delta,r)>\SL(3\delta,r)>\cdots
  \qquad \forall\, 1\leq r\leq n.
\end{equation*}
\end{remark}


\subsection{Pre-convexity in type $A^{(1)}_n$}\label{ssec:preconvexity}
\

Motivated by Definition~\ref{def:convex}, we shall call an order $<$ on $\wDelta^{+,\mathrm{re}}$
\emph{pre-convex} if
\begin{equation}\label{eq:pre-convex}
  \alpha<\alpha+\beta<\beta \quad \mathrm{or} \quad \beta<\alpha+\beta<\alpha
  \qquad \forall\ \alpha,\beta,\alpha+\beta\in \wDelta^{+,\mathrm{re}}.
\end{equation}

\begin{proposition}\label{prop:order-prop-2}
The restriction of~\eqref{eqn:affine-induces} to $\wDelta^{+,\mathrm{re}}$ is pre-convex.
\end{proposition}

\begin{proof}
Without loss of generality, we can assume that~\eqref{eq:minimal letters} holds, so that the formulas
of Theorem~\ref{thm:sln-general} apply. The proof follows by a direct case-by-case analysis:

\medskip
\noindent
$\bullet$ $\alpha=\alpha_{a\to b}+k\delta$, $\beta=\alpha_{(b+1)\to c}+r\delta$ for
$1 \prec a \preceq b \prec c \prec i$.

$\circ$ Case 1: $k,r>0$.
In this case, we have
  $\SL(\alpha)=\underbrace{\ell_{b+1}(\delta)}_{k \, \rtim} b (b-1) \dots a$,
  $\SL(\beta)=\underbrace{\ell_{c+1}(\delta)}_{r \, \rtim} c(c-1) \dots (b+1)$,
  $\SL(\alpha+\beta)=\underbrace{\ell_{c+1}(\delta)}_{(k+r) \, \rtim} c(c-1) \dots a$.
The inequality $\SL(\alpha)<\SL(\alpha+\beta)$ is a consequence of $\ell_{c+1}(\delta)>\ell_{b+1}(\delta)$
(due to Lemma \ref{deltaorder}), while the inequality $\SL(\alpha+\beta)<\SL(\beta)$ is obvious as
$\ell_{c+1}(\delta)$ starts with $1<c$.

$\circ$ Case 2: $k=0,r>0$.
In this case, we have
  $\SL(\beta)=\underbrace{\ell_{c+1}(\delta)}_{r \, \rtim} c(c-1) \dots (b+1)$,
  $\SL(\alpha+\beta)=\underbrace{\ell_{c+1}(\delta)}_{r \, \rtim} c(c-1) \dots a$,
while $\SL(\alpha)$ starts with a letter $>1$.
Therefore, we immediately get $\SL(\alpha)>\SL(\alpha+\beta)>\SL(\beta)$.

$\circ$ Case 3: $k>0,r=0$.
In this case, we have
  $\SL(\alpha)=\underbrace{\ell_{b+1}(\delta)}_{k \, \rtim} b (b-1) \dots a$,
  $\SL(\alpha+\beta)=\underbrace{\ell_{c+1}(\delta)}_{k \, \rtim} c(c-1) \dots a$,
while $\SL(\beta)$ starts with a letter $>1$. Evoking the inequality $\ell_{c+1}(\delta)>\ell_{b+1}(\delta)$,
we immediately get $\SL(\alpha)<\SL(\alpha+\beta)<\SL(\beta)$.

$\circ$ Case 4: $k=r=0$.
In this case, $\alpha,\beta,\alpha+\beta\in \Delta^+$, hence the claim follows from Proposition~\ref{prop:fin.convex}
(a priori we do not know which of the two possible orders holds).

\medskip
\noindent
$\bullet$ $\alpha=\alpha_{a\to b}+k\delta$, $\beta=\alpha_{\overline{b+1}\to c}+r\delta$ for
$i \prec a \preceq b \prec c \preceq 0$.

$\circ$ Case 1: $k,r>0$.
In this case, we have
  $\SL(\alpha)=\underbrace{\ell_{a-1}(\delta)}_{k \, \rtim} a \, \overline{a+1} \dots b$,
  $\SL(\beta)=\underbrace{\ell_{b}(\delta)}_{r \, \rtim} \overline{b+1}\, \overline{b+2} \dots c$,
  $\SL(\alpha+\beta)=\underbrace{\ell_{a-1}(\delta)}_{(k+r) \, \rtim} a \, \overline{a+1} \dots c$.
The inequality $\SL(\beta)<\SL(\alpha+\beta)$ is a consequence of $\ell_{a-1}(\delta)>\ell_{b}(\delta)$
(due to Lemma \ref{deltaorder}), while the inequality $\SL(\alpha+\beta)<\SL(\alpha)$ is obvious as
$\ell_{a-1}(\delta)$ starts with $1$ which is $<a$.

$\circ$ Case 2: $k=0,r>0$.
In this case, we have
  $\SL(\beta)=\underbrace{\ell_{b}(\delta)}_{r \, \rtim} \overline{b+1}\, \overline{b+2} \dots c$,
  $\SL(\alpha+\beta)=\underbrace{\ell_{a-1}(\delta)}_{r \, \rtim} a \, \overline{a+1} \dots c$,
while $\SL(\alpha)$ starts with a letter $>1$. Evoking the inequality $\ell_{a-1}(\delta)>\ell_{b}(\delta)$,
we immediately get $\SL(\beta)<\SL(\alpha+\beta)<\SL(\alpha)$.

$\circ$ Case 3: $k>0,r=0$.
In this case, we have
  $\SL(\alpha)=\underbrace{\ell_{a-1}(\delta)}_{k \, \rtim} a \, \overline{a+1} \dots b$,
  $\SL(\alpha+\beta)=\underbrace{\ell_{a-1}(\delta)}_{k \, \rtim} a \, \overline{a+1} \dots c$,
while $\SL(\beta)$ starts with a letter $>1$. Therefore, we get
  $$\SL(\alpha)<\SL(\alpha+\beta)<\SL(\beta) .$$

$\circ$ Case 4: $k=r=0$.
In this case, the claim follows from Proposition~\ref{prop:fin.convex} again.

\medskip
\noindent
$\bullet$ $\alpha=\alpha_{a\to (i-1)}+k\delta$, $\beta=\alpha_{i}+r\delta$ for $1 \prec a \prec i$.

$\circ$ Case 1: $k>0, r\geq 0$.
In this case, we have
  $\SL(\alpha)=\underbrace{\ell_{i}(\delta)}_{k \, \rtim} \overline{i-1}\, \overline{i-2} \dots a$,
  $\SL(\alpha+\beta)=
   \begin{cases}
     \underbrace{\ell_i(\delta)}_{\frac{k+r}{2} \, \rtim} i \underbrace{\ell_i(\delta)}_{\frac{k+r}{2} \, \rtim} \overline{i-1}\dots a
       & \mathrm{if}\ 2\mid (k+r) \\
     \underbrace{\ell_i(\delta)}_{\frac{k+r+1}{2} \, \rtim} \overline{i-1}\dots a
     \underbrace{\ell_i(\delta)}_{\frac{k+r-1}{2} \, \rtim} i
       & \mathrm{if}\ 2\nmid (k+r)
   \end{cases}$,
  and $\SL(\beta)=\underbrace{\ell_{i}(\delta)}_{r \, \rtim} i$.

If $2\mid (k+r)$ and $k>\frac{k+r}{2}>r$, then clearly $\SL(\alpha)<\SL(\alpha+\beta)<\SL(\beta)$.
If $2\mid (k+r)$ and $k\leq \frac{k+r}{2}\leq r$, then clearly $\SL(\alpha)>\SL(\alpha+\beta)>\SL(\beta)$.

If $2\nmid (k+r)$ and $k\geq \frac{k+r+1}{2}>r$, then clearly $\SL(\alpha)<\SL(\alpha+\beta)<\SL(\beta)$.
If $2\nmid (k+r)$ and $k<\frac{k+r+1}{2}\leq r$, then clearly $\SL(\alpha)>\SL(\alpha+\beta)>\SL(\beta)$.

$\circ$ Case 2: $k=0, r>0$.
In this case,
  $\SL(\alpha)$ starts with a letter $>1$, $\SL(\beta)=\underbrace{\ell_{i}(\delta)}_{r \, \rtim} i$,
  $\SL(\alpha+\beta)=
   \begin{cases}
     \underbrace{\ell_i(\delta)}_{\frac{r}{2} \, \rtim} i \underbrace{\ell_i(\delta)}_{\frac{r}{2} \, \rtim} \overline{i-1}\dots a
       & \mathrm{if}\ 2\mid r \\
     \underbrace{\ell_i(\delta)}_{\frac{r+1}{2} \, \rtim} \overline{i-1}\dots a
     \underbrace{\ell_i(\delta)}_{\frac{r-1}{2} \, \rtim} i
       & \mathrm{if}\ 2\nmid r
   \end{cases}$.
Therefore, we immediately get $\SL(\alpha)>\SL(\alpha+\beta)>\SL(\beta)$.

$\circ$ Case 3: $k=r=0$.
In this case, the claim follows from Proposition~\ref{prop:fin.convex} again.
In fact, we get $\SL(\alpha)>\SL(\alpha+\beta)>\SL(\beta)$ since $\SL(\alpha)>\SL(\beta)$ (as $i<a,\ldots,i-1$).

\medskip
\noindent
$\bullet$ $\alpha=\alpha_{a\to b}+k\delta$, $\beta=\alpha_{(b+1)\to i}+r\delta$ for
$1 \prec a \preceq b \prec i-1$.

$\circ$ Case 1: $k,r>0$.
Combining~(\ref{eq:general-three},~\ref{genreal four 3}) and Lemma \ref{deltaorder}, we obtain:
  $$\SL(\alpha)=
    \underbrace{\ell_{b+1}(\delta)}_{k \, \mathrm{times}} b\, \overline{b-1}\dots a<
    \ell_i(\delta)<\SL(\beta) \,,\, \SL(\alpha+\beta) \,.$$
It thus remains to prove that $\SL(\alpha+\beta)<\SL(\beta)$. This is obvious unless $k=1$ and $2 \nmid r$,
as $\SL(\alpha+\beta)$ contains more copies of $\ell_i(\delta)$'s in the beginning than $\SL(\beta)$,
due to~\eqref{genreal four 3} and $\lceil \frac{k+r}{2} \rceil > \lceil \frac{r}{2} \rceil$.
Meanwhile, for $k=1$ and $2 \nmid r$ we have:
  $$\SL(\alpha+\beta)=\underbrace{\ell_i(\delta)}_{\frac{r+1}{2}} i
    \underbrace{\ell_i(\delta)}_{\frac{r+1}{2}} \overline{i-1}\dots a <
    \underbrace{\ell_i(\delta)}_{\frac{r+1}{2}} \overline{i-1}\dots
    \overline{b+1}\underbrace{\ell_i(\delta)}_{\frac{r-1}{2}}i = \SL(\beta) \,.$$

$\circ$ Case 2: $k=0, r>0$.
We have $\SL(\alpha_{(b+1)\to i}+r\delta) \SL(\alpha_{a \to b})\leq \SL(\alpha_{a\to i}+r\delta)$,
due to Proposition~\ref{prop:generalized Leclerc}. Therefore:
  $\SL(\beta) < \SL(\beta) \SL(\alpha)\leq \SL(\alpha_{a\to i}+r\delta)=\SL(\alpha+\beta)$.
On the other hand, $\SL(\alpha)$ starts with $\min\{a,\ldots,b\}$ which is $>1=$ the first letter
of $\SL(\alpha+\beta)$. Hence, $\SL(\beta)<\SL(\alpha+\beta)<\SL(\alpha)$.

$\circ$ Case 3: $r=0, k>0$.
In this case, we have $\SL(\alpha)<\SL(\alpha+\beta)<\SL(\beta)$, due to
$\ell_{b+1}(\delta)<\ell_i(\delta)$ (by Lemma~\ref{deltaorder}) and $1<i$.

$\circ$ Case 4: $k=r=0$.
In this case, the claim follows from Proposition~\ref{prop:fin.convex} again.
In fact, we get $\SL(\alpha)>\SL(\alpha+\beta)>\SL(\beta)$ since $\SL(\alpha)>\SL(\beta)$ (as $i<a,\ldots,i-1$).

\medskip
\noindent
$\bullet$ $\alpha=\alpha_{a\to b}+k\delta$, $\beta=\alpha_{(b+1)\to c}+r\delta$ for
$1 \prec a \preceq b \prec i-1$ and $i \prec c \preceq 0$.

The proof is absolutely analogous to the previous case, but we should now look at $r$ mod $3$
(rather than $r$ mod $2$) and use formula~\eqref{eq:general-four} instead of~\eqref{genreal four 3}.

\medskip
\noindent
$\bullet$ $\alpha=\alpha_{a\to (i-1)}+k\delta$, $\beta=\alpha_{i\to b}+r\delta$ for
$1 \prec a \prec i \prec b \preceq 0$.

$\circ$ Case 1: $k,r>0$.
Let us compare the multiplicity of the word $\ell_i(\delta)$ in the beginning of our words:
it is $k$ for $\SL(\alpha)$, $\lceil \frac{r}{2} \rceil$ for $\SL(\beta)$, and $\lceil \frac{k+r}{3} \rceil$
for $\SL(\alpha+\beta)$. If $r=2k+3$ or $r>2k+4$, then $k<\lceil \frac{k+r}{3} \rceil<\lceil \frac{r}{2} \rceil$
(as $\lceil \frac{k+r}{3} \rceil \leq \frac{k+r+2}{3}<\frac{r}{2}\leq \lceil \frac{r}{2} \rceil$ for $r>2k+4$),
and so $\SL(\beta)<\SL(\alpha+\beta)<\SL(\alpha)$.
If $r<2k-3$, then likewise $k>\lceil \frac{k+r}{3} \rceil>\lceil \frac{r}{2} \rceil$
(as $\lceil \frac{k+r}{3} \rceil \geq  \frac{k+r}{3}>\frac{r+1}{2} \geq \lceil \frac{r}{2} \rceil$),
and so $\SL(\alpha)<\SL(\alpha+\beta)<\SL(\beta)$.
Thus, it remains to consider $r\in\{2k-3,2k-2,2k-1,2k,2k+1,2k+2,2k+4\}$. Let us illustrate the argument
for $r=2k-2$, while the other six cases are treated completely analogously. For $r=2k-2$,
$\lceil \frac{r}{2} \rceil<k=\lceil \frac{k+r}{3} \rceil$, and so it suffices to prove that
$\SL(\alpha)<\SL(\alpha+\beta)$. Comparing formulas~(\ref{eq:general-three},~\ref{eq:general-four}),
we see that either $\SL(\alpha)$ is a proper prefix of $\SL(\alpha+\beta)$ if $\overline{i-1}>\overline{i+1}$,
or its first letter after $k$ copies of $\ell_i(\delta)$ is smaller than that of $\SL(\alpha+\beta)$
if $\overline{i-1}<\overline{i+1}$. Thus $\SL(\alpha)<\SL(\alpha+\beta)$.

$\circ$ Case 2: $k=0, r>0$. Comparing the first letters, we get $\SL(\alpha)>\SL(\alpha+\beta)$.
It thus remains to prove $\SL(\alpha+\beta)>\SL(\beta)$. For $r=3$ or $r>4$, this follows from
$\lceil \frac{r}{2} \rceil > \lceil \frac{r}{3} \rceil$. The cases $r\in \{1,2,4\}$ are treated
similarly to $r=2k-2$ in Case 1.

$\circ$ Case 3: $k>0,r=0$. Comparing the first letters, we get $\SL(\beta)>\SL(\alpha+\beta)$,
while $\SL(\alpha+\beta)>\SL(\alpha)$ is verified alike $\SL(\alpha+\beta)>\SL(\beta)$ in Case 2.

$\circ$ Case 4: $k=r=0$.
In this case, the claim follows from Proposition~\ref{prop:fin.convex} again.
In fact, we get $\SL(\alpha)>\SL(\alpha+\beta)>\SL(\beta)$ since $\SL(\alpha)>\SL(\beta)$ (as $i<a,\ldots,i-1$).

\medskip

The next four cases are absolutely similar to the previous four:

\medskip
\noindent
$\bullet$ $\alpha=\alpha_{i}+k\delta$, $\beta=\alpha_{\overline{i+1}\to b}+r\delta$ for
$i \prec b \preceq 0$.

\medskip
\noindent
$\bullet$ $\alpha=\alpha_{i\to b}+k\delta$, $\beta=\alpha_{\overline{b+1}\to c}+r\delta$ for
$i \prec b \prec c \preceq 0$.

\medskip
\noindent
$\bullet$ $\alpha=\alpha_{a\to b}+k\delta$, $\beta=\alpha_{\overline{b+1}\to c}+r\delta$ for
$1 \prec a \prec i \prec b \prec c \preceq 0$.

\medskip
\noindent
$\bullet$ $\alpha=\alpha_{a\to i}+k\delta$, $\beta=\alpha_{\overline{i+1}\to b}+r\delta$ for
$1 \prec a \prec i \prec b \preceq 0$.

\medskip

Finally, let us treat the remaining three cases that utilize~\eqref{eq:general-five} and its proof.

\medskip
\noindent
$\bullet$
$\alpha=(\alpha_{a\to b}+k\delta)$, $\beta=(\alpha_{\overline{b+1}\to c}+r\delta)$
for $1\in [a;b]$ and $1 \notin [\overline{b+1};c]$.

\noindent
$\circ$ Case 1: $c\in [\overline{b+1};\overline{a-2}]$.

If $k>0,r>0$, then $\SL(\alpha)<\SL(\alpha)\SL(\beta)\leq \SL(\alpha+\beta)$ with the second inequality proved
in case 1) of our proof of~\eqref{eq:general-five}. Thus, it remains to prove $\SL(\alpha+\beta)<\SL(\beta)$.
By Corollary~\ref{cor:ell1-invariant}, $\SL(\alpha+\beta)$ starts with
  $\min\{\SL(\alpha_{a\to \overline{a-2}})\,1,\SL(\alpha_{\overline{c+2}\to c})\,1\}<
   \SL(\alpha_{\overline{c+2}\to c})\overline{c+1}=\ell_{\overline{c+1}}(\delta)$.
On the other hand, $\SL(\beta)$ starts with $\ell_i(\delta)\geq \ell_{\overline{c+1}}(\delta)$ if $i \in [(b+1)\to c)$,
with $\ell_{\overline{c+1}}(\delta)$ if $1\prec b+1 \preceq c \prec i$, with
$\ell_{b}(\delta)>\ell_{\overline{c+1}}(\delta)$ for $i\prec b+1 \preceq c$ (by Lemma~\ref{deltaorder}).
This completes the proof of $\SL(\alpha)<\SL(\alpha+\beta)<\SL(\beta)$ for $k,r>0$.
The inequalities are similar when $k\ne r=0$ or $r\neq k=0$.
Finally, for $k=r=0$ the claim follows from Proposition~\ref{prop:fin.convex}.
In fact, since $1$ is the minimal element of $\wI$, we get $\SL(\alpha)<\SL(\alpha+\beta)<\SL(\beta)$.

\noindent
$\circ$ Case 2: $c\in [a;0]$.

Note that $\SL(\alpha_{\overline{b+1}\to c}+r\delta)>\SL(\alpha_{\overline{b+1}\to c}+(r+k+1)\delta)$ by Remark~\ref{rem:monot-increas}.
We also have $\SL(\alpha_{\overline{b+1}\to c}+(r+k+1)\delta)>\SL(\alpha_{a\to c}+(r+k+1)\delta)$, due to an already established
pre-convexity for roots $(\alpha_{a\to c}+(r+k+1)\delta)+\alpha_{\overline{b+1}\to \overline{a-1}}=\alpha_{\overline{b+1}\to c}+(r+k+1)\delta$.
Combining the two inequalities above, we obtain $\SL(\alpha+\beta)<\SL(\beta)$.
Evoking Corollary~\ref{cor:ell1-invariant}, we see that $\SL(\alpha)$ starts with
$\ell_1 1\leq \SL(\alpha_{\overline{a-2}\to a}) 1<\ell_{\overline{a-1}}(\delta)$, so that $\SL(\alpha)<\ell_{\overline{a-1}}(\delta)$.
On the other hand, $\SL(\alpha+\beta)$ is lexicographically larger than $\ell_{\overline{a-1}}(\delta)$, due to explicit formulas of
Theorem~\ref{thm:sln-general} and Lemma~\ref{deltaorder}. Combining these inequalities, we obtain $\SL(\alpha)<\SL(\alpha+\beta)$.

\medskip
\noindent
$\bullet$
$\alpha=(\alpha_{a\to b}+k\delta)$, $\beta=(\alpha_{\overline{b+1}\to c}+r\delta)$
for $1\notin [a;b]$ and $1\in [\overline{b+1};c]$.

The proof in this case is completely analogous to the previous one.

\medskip
\noindent
$\bullet$
$\alpha=\alpha_{a\to b}+k\delta$, $\beta=\alpha_{\overline{b+1}\to c}+r\delta$
for $1\in [a;b]$ and $1 \in [\overline{b+1};c]$.

According to~\eqref{eq:ba-k0}, we have:
$\SL(\alpha_{a\to b}+k\delta)\geq \SL(\alpha_{a\to b})$ and
$\SL(\alpha_{\overline{b+1}\to c}+r\delta)\geq \SL(\alpha_{\overline{b+1}\to c})$ for $k,r\geq 0$.
Thus, $\SL(\alpha_{a\to b}+k\delta)\geq \SL(\alpha_{a\to \overline{c+1}})$ and
$\SL(\alpha_{\overline{b+1}\to c}+r\delta)\geq \SL(\alpha_{\overline{a-1}\to c})$ by Lemma~\ref{deltaorder}
as $1\in [a;\overline{c+1}] \subset [a;b]$ and $1\in [\overline{a-1};c] \subseteq [\overline{b+1};c]$.
Evoking the proof of Lemma \ref{lemma root with 1}, see~\eqref{eq:comb-alg}, we conclude that one of the
words $\SL(\alpha_{\overline{a-1}\to c})$ and $\SL(\alpha_{a\to \overline{c+1}})$ is $>$
$\SL(\alpha_{a\to c}+(k+r+1)\delta)$. This implies that $\max\{\SL(\alpha),\SL(\beta)\}>\SL(\alpha+\beta)$.
The other inequality is obvious: $\min\{\SL(\alpha),\SL(\beta)\}<\SL(\alpha+\beta)$,
cf.~our treatment of case 2) in the proof of~\eqref{eq:general-five}.
This competes the proof for any $k,r\geq 0$.
\end{proof}


\medskip

\appendix
\section{Computer code}\label{sec:appendix}

The generalized Leclerc's algorithm of Proposition~\ref{prop:generalized Leclerc} is easy to program.
This allows one to find $\aslaws$ for any affine type (which is especially useful for exceptional types
$F_4$ and $E_{6,7,8}$) and any order on the alphabet $\wI$, arguing by induction on the height of
an affine root. Here are the clickable codes:
\begin{itemize}[leftmargin=0.7cm]

\item
\href{https://replit.com/@IeghorAvdieiev2/Real-roots?v=1}{Python Code 1}

\item
\href{https://replit.com/@IeghorAvdieiev2/Imaginary-roots?v=1}{Python Code 2}

\end{itemize}
The first code computes $\SL(\alpha)$ for $\alpha\in \wDelta^{+,\mathrm{re}}$ with
$kh<\hgt(\alpha)<(k+1)h$ (here, $h=\hgt(\delta)$ is the Coxeter number of $\fg$) using the
algorithm of Proposition~\ref{prop:generalized Leclerc}(a). The second code evaluates
$\{\SL_r((k+1)\delta))\}_{r=1}^{|I|}$ using the algorithm of Proposition~\ref{prop:generalized Leclerc}(b).

\begin{remark}
To code the algorithm of Proposition~\ref{prop:generalized Leclerc} it is key to define a function that evaluates
standard bracketing of $\aslaws$ and a function that checks bracketings for linear independence. The code works
inductively and proceeds block-wise evaluating $\SL_*(\alpha)$ for $kh<\hgt(\alpha)\leq (k+1)h$ at each step.
\end{remark}

Assuming we already have
all $\SL_*(\alpha)$ for $\hgt(\alpha)\leq kh$, the code lists all possible decompositions $\alpha=\gamma_1+\gamma_2$.
For $\alpha\in \wDelta^{+,\mathrm{re}}$ the code counts bracketings and finds the biggest word with the nonzero bracketing.
For $\alpha=(k+1)\delta$, the code finds the biggest $|I|=rk(\mathfrak{g})$ words with linear independent bracketings.



\medskip

\begin{thebibliography}{XXX}

\bibitem[B]{B}
J.~Beck,
  {\em Convex bases of PBW type for quantum affine algebras},
Commun.\ Math.\ Phys.\ 165 (1994), no.~1, 193--199.

\bibitem[D]{D}
I.~Damiani,
  {\em A basis of type Poincare-Birkhoff-Witt for the quantum algebra of $\widehat{\mathsf{sl}}(2)$},
J.\ Algebra 161 (1993), no.~2, 291--310.

\bibitem[G]{G}
J.~Green,
  {\em Quantum groups, Hall algebras and quantized shuffles},
Finite reductive groups (Luminy 1994), 273--290, Birkh\"auser Prog.\ Math.\ 141.

\bibitem[KT]{KT}
S.~Khoroshkin, V.~Tolstoy,
  {\em The universal $R$-matrix for quantum untwisted affine Lie algebras},
Funct.\ Anal.\ Appl.\ 26 (1992), no.~1, 69--71.

\bibitem[LR]{LR}
P.~Lalonde P., A.~Ram,
  {\em Standard Lyndon bases of Lie algebras and enveloping algebras},
Trans.\ Amer.\ Math.\ Soc.\ 347 (1995), no.~5, 1821--1830.	

\bibitem[L]{L}
B.~Leclerc,
  {\em Dual canonical bases, quantum shuffles and $q$-characters},
Math.\ Z.\ 246 (2004), no.~4, 691--732.	

\bibitem[Lo]{Lo}
M.~Lothaire,
   {\em Combinatorics of words},
Cambridge University Press, Cambridge, 1997.
	
\bibitem[NT]{NT}
A.~Negu{\c t}, A.~Tsymbaliuk,
  {\em Quantum loop groups and shuffle algebras via Lyndon words},
Adv.\ Math.\ 439 (2024), Paper No.~109482.

\bibitem[P]{P}
P.~Papi,
  {\em A characterization of a special ordering in a root system},
Proc.\ Amer.\ Math.\ Soc.\ 120 (1994), no.~3, 661--665.

\bibitem[R1]{R1}
M.~Rosso,
  {\em Quantum groups and quantum shuffles},
Invent.\ Math.\ 133 (1998), no.~2, 399--416.

\bibitem[R2]{R2}
M.~Rosso,
  {\em Lyndon bases and the multiplicative formula for $R$-matrices},
preprint (2002).

\bibitem[S]{S}
P.~Schauenburg,
  {\em A characterization of the Borel-like subalgebras of quantum enveloping algebras},
Comm.\ Algebra 24 (1996), no.~9, 2811--2823.

\end{thebibliography}
\end{document}